\documentclass[english,reqno]{article}
\usepackage[T1]{fontenc}
\usepackage[utf8]{inputenc}
\usepackage{lmodern}
\usepackage[left=3.5cm,right=3.5cm,top=2.5cm,bottom=2.5cm]{geometry}

\usepackage[leqno]{amsmath}
\usepackage{amssymb}
\usepackage{amsmath}
\usepackage{amsthm}
\usepackage{pgfplots}
\pgfplotsset{compat=1.14}
\usepackage{enumitem}%
\usepackage{hyperref}
\usepackage{amssymb}

\usepackage{graphicx}
\usepackage{tikz}

\usepackage{float}
\usepackage{todonotes}

\usepackage{pgfplots}
\usepackage{subfig}

\newcounter{lem}
\newtheorem{lemma}[lem]{Lemma}
\newtheorem{remark}[lem]{Remark}
\newtheorem{proposition}[lem]{Proposition}
\newtheorem{theorem}[lem]{Theorem}
\newtheorem{definition}[lem]{Definition}
\newtheorem{notation}[lem]{Notation}
\newtheorem{assumption}[lem]{Assumption}
\newtheorem{corollary}[lem]{Corollary}

\usepackage{mathtools}
\usepackage{bbm}
\usepackage{bm}
\usepackage[]{algorithm2e}
\usepackage[affil-it]{authblk}

\usepackage{tikzsymbols}
\usepackage{subfig}

\DeclarePairedDelimiter\floor{\lfloor}{\rfloor}
\DeclareMathOperator{\E}{\mathbb{E}}%

\newcommand*\indic[1]{\mathbbm{1}_{\{ #1 \}}}
\newcommand*\indica[1]{\mathbbm{1}_{ #1 }}

\newcommand\mydots{\hbox to 0.7em{.\hss.\hss.}}
\def \limhb#1#2#3{\mathrel{\mathop{\kern 0pt#1}\limits_{#2}^{#3}}}

\newcommand{\PM}{\bm{\mathrm{N}}}

\usepackage{xargs}

\makeatletter

\newcommandx{\Flow}[2][1=,2=]{  \varphi^{#1}_{#2}}

\newcommandx{\K}[2][1=,2=]{K^{#1}_{#2}}
\renewcommandx{\H}[2][1=,2=]{H^{#1}_{#2}}
\renewcommandx{\r}[2][1=,2=]{r^{#1}_{#2}}
\newcommandx{\XI}[2][1=,2=]{\xi^{#1}_{#2}}

\usepackage{ifmtarg}

\newcommandx{\Y}[3][1=,2=,3=]{
	Y^{#1\@ifmtarg{#1}{}{\@ifmtarg{#2}{}{,}}#2}_{#3}}

\newcommand{\EXP}[1]{\exp{\left( #1 \right) }}

\begin{document}
	\date{\today}
\title{Hopf bifurcation in a Mean-Field model of spiking neurons}
\author[1]{Quentin Cormier}
\author[1]{Etienne Tanr\'e}
\author[1]{Romain Veltz}
\affil[1]{Université Côte d’Azur, Inria, France}
\maketitle
\begin{abstract}
	We study a family of non-linear McKean-Vlasov SDEs 
		driven by a Poisson measure, modelling the 
		mean-field asymptotic of a network of
		generalized Integrate-and-Fire neurons.
		We give sufficient conditions  to have periodic solutions through
		a Hopf bifurcation. 
		Our spectral 
		conditions involve the location of the roots of an explicit holomorphic function. 
		The 
		proof relies on two main ingredients. First, we introduce a discrete time Markov 
		Chain 
		modeling the phases of the successive spikes of a neuron. The invariant measure 
		of this 
		Markov Chain is related to the shape of the periodic solutions.  Secondly, we use 
		the 
		Lyapunov-Schmidt method to obtain self-consistent oscillations. We illustrate the 
		result with a toy model for which all the spectral conditions can be analytically 
		checked.
\end{abstract}
\section{Introduction}
We consider a mean-field model of spiking neurons. Let $f: \mathbb{R}_+ 
\rightarrow 
\mathbb{R}_+$, $b: \mathbb{R}_+
\rightarrow \mathbb{R}$ and let $\PM(du, dz)$ be a Poisson measure on
$\mathbb{R}^2_+$ with
intensity the
Lebesgue measure $du dz$. Consider the following McKean-Vlasov SDE
\begin{equation}
	X_t = X_0 + \int_0^t{b(X_u) du} + J \int_0^t{\E f(X_u) du}  -
	\int_0^t{\int_{\mathbb{R}_+}{X_{u-}
			\indic{z \leq f(X_{u-})} \PM(du, dz)}}.
	\label{NL-equation}
\end{equation}
Here, $J \geq 0$ is a deterministic constant (it models the strength of the
interactions) and the
initial condition $X_0$ is independent of the Poisson measure.

Informally, the
SDE
\eqref{NL-equation} can be understood in the following sense:
between the jumps, $X_t$ solves the scalar ODE
\(\dot{X}_t=
b(X_t) + J \E f(X_t) \)
and $X_t$ jumps to $0$ at rate $f(X_t)$. We assume that $b(0) \geq 0$ 
and that $X_0 
	\geq 0$, such 
	that the dynamics lives on $\mathbb{R}_+$.
This SDE is non-linear in the sense of McKean-Vlasov, because of the
interaction term $\E f(X_t)$
which
depends on the law of $X_t$. Let $\nu(t, dx) := \mathcal{L}(X_t)$ be the law
of $X_t$. It solves
the following non-linear Fokker-Planck equation, in the sense of measures:
\begin{align}
	\label{eq:NL PDE in the sense of measures}
	& \partial_t \nu(t, dx)  + \partial_x \left[ (b(x) + J r_t) \nu(t, dx) \right] + f(x)
	\nu(t, dx) = r_t \delta_0 (dx) \\
	& \nu(0, dx) = \mathcal{L}(X_0), \quad r_t = \int_{\mathbb{R}_+}{f(x) \nu(t,
		dx)}. \nonumber
\end{align}
Here $\delta_0$ is the Dirac measure in $0$. If furthermore
$\mathcal{L}(X_t)$ has a density for all
$t$, that is $\mathcal{L}(X_t) = \nu(t, x)dx$
then $\nu(t, x)$ solves the following strong form of \eqref{eq:NL PDE in the
	sense of measures}
\begin{align*}
	& \partial_t \nu(t, x)  + \partial_x \left[ (b(x) + J r_t) \nu(t, x) \right] + f(x)
	\nu(t, x) = 0,\\
	& \nu(0, x)dx = \mathcal{L}(X_0), \quad r_t = \int_{\mathbb{R}_+}{f(x) \nu(t,
		x)dx},
\end{align*}
with the boundary condition
\[\forall t > 0, \quad \left( b(0) + J r_t \right)  \nu(t, 0) = r_t. \]
We study the existence of periodic solution of this  non-linear 
Fokker-Planck equation.
We give sufficient
conditions for
the existence of a Hopf bifurcation around a stationary solution of
\eqref{eq:NL PDE in the sense
	of measures}.

\subsubsection*{Associated particle system}
Equations \eqref{NL-equation} and \eqref{eq:NL PDE in the sense of
	measures} appeared (see e.g. \cite{de_masi_hydrodynamic_2015}) as the limit of 
	the 
following networks of neurons. For each $N
\geq 1$, consider \textit{i.i.d.} initial potentials
$(X^{i, N}_0)_{i \in
	\{1,\cdots,N\}}$ with law $\mathcal{L}(X_0)$.
The \textit{càdlàg} process
$(X^{i,
	N}_t)_{i \in \{1,\cdots,N\}} \in \mathbb{R}^N$ is a PDMP: between the jumps each
$X^{i,
	N}_t$ solves the ODE $\dot{X}^{i,
	N}_t= b(X^{i,
	N}_t)$ and ``spikes'' with rate $f(X^{i,
	N}_t)$. When a spike occurs, say neuron $i$ spikes at (random)
time $\tau$, its
potential is reset to $0$ while the others receive a ``kick'' of size $\frac{J}{N}$:
\[ 
X^{i, N}_{\tau_+} = 0,\quad \text{ and } \quad \forall k \neq i, \quad X^{k,
	N}_{\tau_+} = X^{k,
	N}_{\tau_-} +
\frac{J}{N}.    
\]
This completely defines the particle system. As $N$ goes to infinity, a 
phenomenon of
\textit{propagation of chaos} occurs.
In particular, each neuron, say $(X^{1, N}_t)_{t \geq 0}$, converges in 
law to the solution 
of
\eqref{NL-equation}.
We refer to \cite{fournier_toy_2016} for a proof of such convergence result under 
stronger assumptions.
There is a qualitative
difference between the particle
system and the solution of the limit equation \eqref{NL-equation}: for a
fixed value of $N$, the
particle system is Harris ergodic (see \cite{duarte_model_2014}, where this result is 
proved 
under stronger assumptions on $b$ and $f$) and so it
admits a unique,
globally attractive, invariant measure. Thus, there are no stable oscillations
when the number
of particles is finite.
For the limit equation however, the long time behavior is richer: for
fixed values of
the parameters there can be multiple invariant measures (see \cite{CTV} and
\cite{cormier2020meanfield} for some explicit examples) and, as shown here,
there can exist
periodic
solutions
	(see Figure~\ref{fig:oscillation x10}).

\subsubsection*{Literature}
From a mathematical point of view, this model has been
first introduced by \cite{de_masi_hydrodynamic_2015}, after many
considerations by physicists (see for instance \cite{plesser_noise_2000},
\cite{GerstnetNeuronalDynamic} and \cite{MR2795699} and references
therein). 
In \cite{fournier_toy_2016}, the existence of solution of  \eqref{NL-equation}, 
path-wise uniqueness and convergence of the particle system are addressed. The 
long time behavior 
of the solution to \eqref{NL-equation} 
is
studied in \cite{CTV} in the case of weak  interactions: $b$ and $f$
being fixed, the authors prove that there exists a constant $\bar{J}$
(depending on $b$ and $f$) such that for all $J < \bar{J}$,
\eqref{NL-equation} admits a unique globally attractive invariant measure.
Finally in \cite{cormier2020meanfield}, the local stability of
an invariant measure is studied with no further assumptions on the size of
the interactions $J$. It is proved that the stability of an invariant measure is
given by the location of the roots of some holomorphic function. In 
\cite{lcherbach2020metastability}, the authors study a ``metastable'' behavior of the 
particle system. They give examples of drifts $b$ and rate functions $f$ where the 
particle 
system follows the long time behavior of the mean-field model for an exponential 
large time, before 
finally converging to its (unique) invariant probability measure.

The model studied in the current paper belongs to the class of 
generalized integrate-and-fire
neurons, whose most celebrated example is the ``fixed threshold'' model
(see for instance \cite{MR2853216}, \cite{DIRT1} and the references therein). 
Many of the
techniques developed here also apply to this variant.
However, it would require additional work to 
	overcome the specific difficulties due to the fixed threshold setting. In particular, 
	there are no simple explicit expressions 
	of the kernels introduced in the current paper.

In \cite{drogoul_hopf_2017}, numerical evidences are given for the existence
of a Hopf bifurcation in a close setting: the dynamics
between the jumps is (as in~\cite{de_masi_hydrodynamic_2015}) given by
\[ \dot{X}_t = -(X_t - \E X_t) + J \E f(X_t). \]
In particular the potentials of each neuron are attracted to their common mean. This
models ``electrical synapses'', while $J \E f(X_t)$ models the chemical
synapses. Oscillations with both electrical and chemical synapses is also
studied in a different model in \cite{PhysRevE.100.042412}. In this work, the
mean-field equation is a 2D-ODE and so the analysis of the Hopf bifurcation
is standard. Finally, oscillations with multi-populations such as with both
excitatory and inhibitory neurons have been extensively studied in
neuroscience. For instance in \cite{MR3646433}, it is shown that
multi-populations of mean-field Hawkes processes can oscillate. Again, the
dynamics is reduced to a finite dimension ODE.

It is well-known that the long time behavior of McKean-Vlasov SDEs can be 
significantly different from markovian SDEs. 
In \cite{MR781411} and \cite{MR808166}, the 
author gives simple examples of such non-linear SDEs which oscillate. Again, in 
these 
examples, the dynamics can be reduced to an ordinary differential equation. 
To go beyond ODEs, the framework of Delay differential equation is often used: see 
for 
instance \cite{MR904818} for the study of Hopf bifurcations for such 
equations, based on the Lyapunov-Schmidt method.
In \cite{MR4061402, luon2018periodicity} the authors study periodic solutions of a 
McKean-Vlasov 
SDE using a 
slow-fast approach. Another approach is to use the center manifold theory to 
reduce infinite 
dimensional problem to manifold of finite dimension: we refer to \cite{MR2759609} 
(see also 
\cite{MR3023444} for an application to some McKean-Vlasov SDE).
Finally, in \cite{MR2859263} an abstract framework is 
presented to 
study Hopf bifurcations for some classes of regular PDEs. Even though our proof is 
not 
based on the 
PDE \eqref{eq:NL PDE in the sense of measures} (but on the Volterra integral 
equation 
described below), we follow the methodology of \cite{MR2859263} to obtain our 
main result.
	\subsubsection*{Regularity of the drift and of the jump function.}
	We make the following regularity assumptions on $b$ and $f$.
	\begin{assumption}
		\label{as:linear drift}
		The drift $b: \mathbb{R}_+ \rightarrow \mathbb{R}$ is $\mathcal{C}^2$, with 
		$b(0) \geq 0$ and
		$\sup_{x \geq 0} |b'(x)| + |b''(x)| < \infty.$
	\end{assumption}
	\begin{assumption}
		\label{as:hyp on f}
		The function $f: \mathbb{R}_+ \rightarrow \mathbb{R}_+$ is $\mathcal{C}^2$, 
		strictly increasing, 
		with $\sup_{x \geq 1} \frac{|f''(x)|}{f(x)} < \infty$ and there exists a constant 
		$C_f$ 
		such 
		that
		\begin{enumerate}[label=\ref{as:hyp on f}(\alph*), leftmargin=*]
			\item \label{as:f(xy)} for all $x,y \geq 0$, $f(xy) \leq C_f(1 + f(x))(1+f(y))$.
			\item \label{as:technical assumption on f}
			for all $A >  0$, $\sup_{x \geq 0} A f'(x) - f(x) < \infty$.
			\item \label{as:technical b f}for all $x \geq 0$, $|b(x)| \leq C_f(1 + f(x))$.
		\end{enumerate}
	\end{assumption}
	\begin{remark}
		\label{rk:comportement at infty}
		If a non-decreasing function $f$ satisfies Assumption~\ref{as:f(xy)}, there 
		exists 
		another constant $C_f$ such that for all $x, y \geq 0$, $f(x + y) \leq C_f ( 1 + 
		f(x) + f(y))$.
		Moreover, it also implies that $f$ grows at most at a polynomial rate: there 
		exists a constant $p> 0$ such that
		\[ \sup_{x \geq 1} f(x) / x^p < \infty. \]
	\end{remark}
	Note that for instance, for all $p \geq  1$, the function $f: x \mapsto x^p$ 
	satisfies Assumption~\ref{as:hyp on f}.
	More generally, any continuous function such that $f(x) \sim_{x \rightarrow 
	\infty}x^p$ 
	for some $p 
	\geq 0$ satisfies Assumption~\ref{as:f(xy)}.
	
	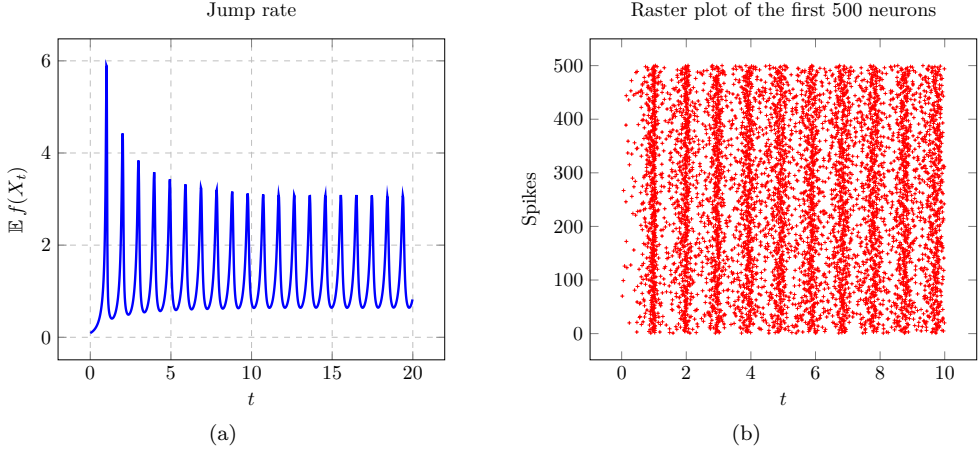
\begin{figure}[h]
		\subfloat[]{
			\begin{tikzpicture}[scale=0.75]
				\begin{axis}[xlabel={$t$}, ylabel={$\E f(X_t)$}, title={Jump rate}, grid=major, 
					grid
					style=dashed]
					\addplot[no marks, blue, very thick] table[ignore chars={(,)},col 
					sep=comma]
					{jumprates.dat};
				\end{axis}
			\end{tikzpicture}
		}
		\qquad
		\subfloat[]{
			\begin{tikzpicture}[scale=0.75]
				\begin{axis}[xlabel={$t$}, ylabel={Spikes}, title={Raster plot of the first 500 
						neurons}]
					\addplot[mark={+}, color={red}, only marks, mark size=1pt] table[ignore 
					chars={(,)},col
					sep=comma]
					{raster.dat};
				\end{axis}
			\end{tikzpicture}
		}
		\caption{
			Consider the following example where for all $x \geq 0$, $f(x) = x^{10}$, $b(x) 
			= 2 - 2x$ and $J 
			= 
			0.8$.	
			Using a Monte-Carlo method, we simulate the particle systems with $N = 8 
			\cdot 10^5$ neurons, 
			starting at $t = 0$ with i.i.d. uniformly distributed random variables on  $[0,1]$. 
			Stable 
			oscillations appear. (a) Empirical mean number of spikes per unit of time. 
			(b) Each red
			cross corresponds to a spike of one of the first 500 neurons (spike raster 
			plot).}
		\label{fig:oscillation x10}
	\end{figure}

\subsubsection*{The Volterra integral equation} 
As in \cite{CTV, cormier2020meanfield}, we study the long time behavior of the 
solution of \eqref{NL-equation} through its ``linearized'' version: given a 
non-negative scalar
function $\boldsymbol{a} \in L^\infty(\mathbb{R}_+; \mathbb{R}_+)$, consider the
non-homogeneous linear SDE:
\begin{equation}
	\label{non-homegeneous-SDE}
	\forall t \geq s, \quad \Y[\boldsymbol{a}][\nu][t, s] = Y_s + 
	\int_s^t{[b(\Y[\boldsymbol{a}][\nu][u, s]) + a_u] du} -  
	\int_s^t{\int_{\mathbb{R}_+}{\Y[\boldsymbol{a}][\nu][u-, s] \indic{z \leq 
	f(\Y[\boldsymbol{a}][\nu][u-, s])} \PM(du, dz)}},
\end{equation}
starting with law $\nu$ at time $s$. That is, equation 
\eqref{non-homegeneous-SDE} is
\eqref{NL-equation} where the interactions $J \E f(X_u)$ have been replaced by the
``external current'' $a_u$. For all $t \geq s$ and for all $\boldsymbol{a} \in
L^\infty(\mathbb{R}_+; \mathbb{R}_+)$, consider $\tau^{\boldsymbol{a}, \nu}_s$ the 
time
of the first jump of $\Y[\boldsymbol{a}][\nu]$ after $s$
\begin{equation} \tau^{\boldsymbol{a}, \nu}_s := \inf \{t \geq s: 
\Y[\boldsymbol{a}][\nu][t,
	s] \neq \Y[\boldsymbol{a}][\nu][t-, s] \}.
	\label{eq:definition de tau first jump time}
\end{equation}
We introduce the spiking rate $\r[\nu][\boldsymbol{a}](t, s)$, the survival function 
$\H[\nu][\boldsymbol{a}](t, s)$ and the density of the first jump 
$\K[\nu][\boldsymbol{a}](t, s)$ to be
\begin{equation}
	\r[\nu][\boldsymbol{a}](t, s) := \E f(\Y[\boldsymbol{a}][\nu][t, s]), \quad
	\H[\nu][\boldsymbol{a}](t, s) := \mathbb{P}(\tau^{\boldsymbol{a}, \nu}_s > t), \quad
	\K[\nu][\boldsymbol{a}](t, s) := - \frac{d}{dt}  \mathbb{P}(\tau^{\boldsymbol{a}, 
	\nu}_s > t).
	\label{def:jump rate, survival, density of survival}
\end{equation}
\begin{notation}
	We detail our conventions and notations.
	\begin{enumerate}
		\item
		We use bold letters \(\boldsymbol{a}\) for time dependent currents 
		and 
			regular greek letters \(\alpha\) for constant currents. 
		\item
		When $\nu = \delta_x$, we write
		\(  
		\r[x][\boldsymbol{a}](t, s) :=  \r[\delta_x][\boldsymbol{a}](t, s).  
		\)
		\item When $\nu = \delta_0$, 
		we  remove the $x$ 
		superscript and write
		\(
		\r[][\boldsymbol{a}](t, s) :=  \r[\delta_0][\boldsymbol{a}](t, s). 
		\)
		\item
		When $\boldsymbol{a}$ is constant and equal to $\alpha \geq 0$, it holds that
		$r^\nu_\alpha(t,s) = r^\nu_\alpha(t-s, 0)$ and we simply note $r^\nu_\alpha(t) :=
		r^\nu_\alpha(t, 0)$. 
		\item Finally, we extend the  function \(r^{\nu}_{\boldsymbol{a}}\), 
		for $s > t$ by setting
		\( \forall s > t, \quad r^{\nu}_{\boldsymbol{a}}(t, s) := 0. \) 
	\end{enumerate}
	We use the same conventions for \(H^{\nu}_{\boldsymbol{a}}\) and 
	\(K^{\nu}_{\boldsymbol{a}}\).
\end{notation}
It is known from \cite[Prop.~19]{CTV} (see also \cite[Prop. 6]{cormier2020meanfield} 
for a 
shorter proof) that $r^\nu_{\boldsymbol{a}}$ is the solution of the following Volterra 
integral 
equation
\begin{equation} r^\nu_{\boldsymbol{a}}(t, s) = K^\nu_{\boldsymbol{a}}(t, s) + 
\int_s^t{ K_{\boldsymbol{a}}(t, u)  r^\nu_{\boldsymbol{a}}(u, s) du}.
	\label{eq:the volterra integral equation}
\end{equation}
Moreover, by \cite[Lem.~17]{CTV}, one has
\begin{equation} 1 = H^\nu_{\boldsymbol{a}}(t, s) + \int_s^t{H_{\boldsymbol{a}}(t, u) 
		r^\nu_{\boldsymbol{a}}(u, s) du}.
	\label{eq:the time of the last spike density sum to 1}
\end{equation}
Following \cite{gripenberg_volterra_1990}, given $c_1, c_2: \mathbb{R}^2 \rightarrow
\mathbb{R}$ two measurable functions, it is convenient to use the notation
\[ (c_1 * c_2) (t, s) = \int_s^t{ c_1(t, u) c_2(u, s) du}, \]
such that \eqref{eq:the volterra integral equation} and \eqref{eq:the time of the last 
	spike density sum to 1} simply write
\[   r^\nu_{\boldsymbol{a}} = K^\nu_{\boldsymbol{a}} +K_{\boldsymbol{a}} * 
r^\nu_{\boldsymbol{a}} \quad \text{ and } \quad  1 = H^\nu_{\boldsymbol{a}} + 
H_{\boldsymbol{a}} *  r^\nu_{\boldsymbol{a}}.   \]
\subsubsection*{The invariant measures of \eqref{NL-equation}.}
Let $\alpha > 0$, define
\begin{equation}
	\label{eq: def sigma alpha}
	\sigma_\alpha := \inf \{x \geq 0,~b(x) + \alpha = 0 \},
\end{equation}
and
\begin{equation}
	\label{eq:the invariant measure for alpha} 
	\nu^\infty_\alpha(x) := \frac{\gamma(\alpha)}{b(x) + \alpha} 
	\EXP{-\int_0^x{\frac{f(y)}{b(y) + \alpha} 
			dy}} \indica{[0,\sigma_\alpha)}(x), 
\end{equation}
where $\gamma(\alpha)$ is the normalizing factor, such that
$\int_{\mathbb{R}_+}{\nu^\infty_\alpha(x) dx} = 1$. Note that $\gamma(\alpha)$ is the 
jump rate
under $\nu^\infty_\alpha$:
\begin{equation}
	\label{eq:gamma alpha}
	\gamma(\alpha) = \nu^\infty_\alpha(f). 
\end{equation}
By \cite[Prop. 26]{CTV}, 
$\nu^\infty_\alpha$ is the unique invariant measure of
the linear SDE \eqref{non-homegeneous-SDE} driven by the 
constant ``external current'' $\boldsymbol{a} \equiv \alpha$. 
We define here a central quantity in our work
\begin{equation}
	\label{eq:the function J of alpha}
	J(\alpha) := \frac{\alpha}{\gamma(\alpha) }.
\end{equation}
It is readily seen that $\nu^\infty_\alpha$ is an invariant measure of the non-linear
equation \eqref{NL-equation} with $J = J(\alpha)$. Reciprocally, for a fixed value of 
$J$,
the number of invariant measures of \eqref{NL-equation} is the number of solutions
\( \alpha \geq 0\) to
the scalar equation
\begin{equation} \alpha = J\gamma(\alpha).
	\label{eq:J equals J(alpha)}
\end{equation}
Any such invariant measure is characterized by its corresponding value of $\alpha$.

\subsubsection*{Stability of an invariant measure}
Fix $J \geq 0$ and consider $\alpha > 0$ a solution of \eqref{eq:J equals 
J(alpha)}. So  
	$\nu^\infty_\alpha$ is an invariant probability measure of \eqref{NL-equation}. A 
	sufficient condition 
for $\nu^\infty_\alpha$ to be locally 
stable is given in \cite{cormier2020meanfield}.

First, consider $H_\alpha(t)$, defined by \eqref{def:jump rate, survival, density of 
survival} 
(with $\nu = \delta_0$ and $\boldsymbol{a} \equiv \alpha$).
For $z \in \mathbb{C}$, we denote by $\Re(z)$ and $\Im(z)$ its real and imaginary 
parts.
The Laplace transform of $H_\alpha(t)$ is defined for $z$ with $\Re(z) > 
-f(\sigma_\alpha)$ ($\sigma_\alpha$ is given by \eqref{eq: def sigma alpha}):
\[ \widehat{H}_\alpha(z) := \int_0^\infty{ e^{-zt } H_\alpha(t) dt}. \]
We assumed that $b(0) \geq 0$ (see Assumption~\ref{as:linear drift}) 
and that $\alpha > 
	0$.  
	So 
	$\sigma_\alpha 
	> 0$. Because  $f$ is strictly increasing (see Assumption~\ref{as:hyp on f}), we 
	deduce that 
	$f(\sigma_\alpha) > 0$. Let
\begin{equation}
	\lambda^*_\alpha := - \sup\{ \Re(z)|~ \Re(z) > 
	-f(\sigma_\alpha),~\widehat{H}_\alpha(z) 
	=  0 \}.
	\label{eq:optimal rate of convergence lambda^*}
\end{equation}
By \cite[Lem. 34, 36]{CTV}, it holds that $\lambda^*_\alpha > 0$. The 
constant 
$\lambda^*_\alpha$ is related to the rate of convergence of $(Y^{\alpha, \delta_0}_{t, 
	0})$ to its invariant measure $\nu^\infty_\alpha$. In particular we have
\begin{equation}
	\label{eq:the jump rate converge with alpha}
	\forall  \lambda < \lambda^*_\alpha,\quad \sup_{t \geq 0} |r_\alpha(t) - 
	\gamma(\alpha)|e^{\lambda t} < \infty. 
\end{equation} 
This describes the long time behavior of an isolated neuron subject to a constant 
current $\alpha > 
0$.
\begin{assumption}
	Assume that the deterministic flow is not degenerate at $\sigma_\alpha$:
	\label{hyp:alpha is non degenerate}
	\begin{align}
		& \sigma_\alpha < \infty \quad  \text{ and }  \quad  b'(\sigma_\alpha) < 0  
		\label{hyp: sigma_alpha is 
			finite}\\
		\text{ \textbf{ or } } \quad  \quad   \quad  & \sigma_\alpha = \infty \quad \text{ 
		and }  \quad  \inf_{x 
			\geq 0}{b(x) + \alpha} > 0. \label{hyp:sigma_alpha is not finite}
	\end{align}
\end{assumption}
	Recall that $r^x_\alpha(t)$ is given by \eqref{def:jump rate, survival, density of 
	survival} (with 
	$\boldsymbol{a} \equiv \alpha, \nu = \delta_x$ and $s = 0$). Following 
	\cite{cormier2020meanfield}, 
	define
	\begin{equation}
		\forall t \geq 0, \quad 	\Theta_\alpha(t) := \int_0^\infty{ \frac{d}{dx} 
		r^{x}_\alpha(t) 
			\nu^\infty_\alpha(dx)}.
		\label{eq:formule donnant Theta.}
	\end{equation}
	By \cite[Prop. 19]{cormier2020meanfield}, $x \mapsto r^x_\alpha(t)$ is 
	$\mathcal{C}^1$ and integrable 
	with respect to $\nu^\infty_\alpha$. Moreover, we have 
\begin{theorem}[\cite{cormier2020meanfield}]
	Grant Assumptions~\ref{as:linear drift}, \ref{as:hyp on f} and \ref{hyp:alpha is non 
	degenerate}.
	It holds that for all $\lambda < \lambda^*_\alpha$ one has $t \mapsto e^{\lambda 
	t} 
	\Theta_\alpha(t) \in L^1(\mathbb{R}_+)$, so that $z \mapsto 
	\widehat{\Theta}_\alpha(z)$ is 
	holomorphic on $\Re(z) > \lambda^*_\alpha$.
	Assume that
	\begin{equation}
		\sup \{\Re (z) ~|~ z \in \mathbb{C}, ~\Re(z) > - \lambda^*_\alpha,~  J(\alpha) 
		\widehat{\Theta}_\alpha(z) = 1  \} < 0,
		\label{eq:criteria_stability}
	\end{equation}
	then the invariant measure $\nu^\infty_\alpha$ is locally stable.
\end{theorem}
We refer to \cite[Def. 16]{cormier2020meanfield} for definition of local stability, 
in particular for the choice of the distance between two probability measures.

Assume that there exists $\alpha > 0$ such that 
\eqref{eq:criteria_stability} holds, and 
	so 
	$\nu^\infty_\alpha$ is locally stable. 
There are two ``canonical'' ways to break \eqref{eq:criteria_stability} at some
\textit{bifurcation point} $\alpha_0$:
either there exists some $\tau_0 > 0$ 
such that $J(\alpha_0) \widehat{\Theta}_{\alpha_0}(\pm \frac{i}{\tau_0}) = 1$
or $J(\alpha_0) \widehat{\Theta}_{\alpha_0}(0) = 1$. The first case is the 
subject of this paper: we 
give explicit conditions to have a Hopf bifurcation.

In the second case, the following
lemma shows that $J'(\alpha_0) = 0$. So, at least in the non-degenerate case where
$J''(\alpha_0) \neq 0$, the function $\alpha \mapsto J(\alpha)$ is not strictly 
monotonic in
the neighborhoods of $\alpha_0$: this is a fold  bifurcation which 
typically
leads to
\textit{bistability} (or \textit{multistability}, etc.).
\begin{lemma}
	\label{lem:link between the derivative of J and Theta}
	Grant Assumptions~\ref{as:linear drift} and \ref{as:hyp on f}, and 
	consider $\alpha_1 > 0$ 
		such that 
		Assumption~\ref{hyp:alpha is non degenerate} holds in $\alpha_1$.
		Then the function $\alpha \mapsto J(\alpha)$ is $\mathcal{C}^2$ in a 
		neighborhood of $\alpha_1$ with
	\[ J'(\alpha) = \frac{1 - J(\alpha) \widehat{\Theta}_\alpha(0) }{\gamma(\alpha)}.\]
\end{lemma}
\begin{proof}
		First, recall that $J(\alpha) = \frac{\alpha}{\gamma(\alpha)}$. So it suffices to 
		show that $\alpha \mapsto \gamma(\alpha)$ is $\mathcal{C}^2$ and that 
		$\gamma'(\alpha) = 
		\widehat{\Theta}_\alpha(0)$. Note that if $\alpha_1 > 0$ satisfies 
		Assumption~\ref{hyp:alpha is non 
			degenerate}, then there 
		exists a neighborhood of $\alpha_1$ such that for any $\alpha$ in this 
		neighborhood, $\alpha$ also 
		satisfies Assumption~\ref{hyp:alpha is non degenerate}. 
		We will see later by Proposition~\ref{prop:regularity of rho} that $\alpha 
		\mapsto 
		\gamma(\alpha)$ is 
		$\mathcal{C}^2$ in the neighborhood for $\alpha_1$. This can also be checked directly: by \cite[eq. 
		(31)]{CTV}, it 
		holds that 
		$\gamma(\alpha)^{-1} = \widehat{H}_\alpha(0)$ and we can conclude using the 
		explicit expression 
		satisfied by 
		$H_\alpha(t)$ (see \eqref{eq:explicit expression H} below). To end the proof, 
		using again that 
		$\gamma(\alpha)^{-1} = 
		\widehat{H}_\alpha(0)$, we have to deduce that
	\[  \frac{d}{d \alpha} \widehat{H}_\alpha(0) = - \frac{
		\widehat{\Theta}_\alpha(0)}{ \left[ \gamma(\alpha) \right]^2 }. \]
	Following \cite{cormier2020meanfield}, let:
	\begin{equation}
		\Psi_\alpha(t) := - \int_0^{\sigma_\alpha}{ \frac{d}{dx} H^x_\alpha(t)  
		\nu^\infty_\alpha(x) dx}. \label{eq:expression de Psi_alpha}
	\end{equation}
	It holds that (see \cite{cormier2020meanfield})
	\begin{equation}
		\Psi_\alpha(t) = \gamma(\alpha) \int_0^\infty{  H_\alpha(t+u) 
		\frac{f(\varphi^\alpha_{t+u}(0)) - f(\varphi^\alpha_u(0))}{b(\varphi^\alpha_u(0)) + 
		\alpha} du }.
		\label{eq:une formule pour Psi_alpha}
	\end{equation}
	Moreover, let
	\begin{equation}
		\label{eq:definition de Xi alpha}
		\forall t \geq 0,\quad  \Xi_\alpha(t) := \frac{d}{dt} \Psi_\alpha(t).  
	\end{equation}
	We have
	\begin{equation}
		\forall t \geq 0,\quad   \Xi_\alpha(t) = \int_0^{\sigma_\alpha}{ \frac{d}{dx} 
		K^x_\alpha(t)
			\nu^\infty_\alpha(x) dx}. \label{eq:expression de Xi_alpha}
	\end{equation}
	So, using \eqref{eq:the volterra integral equation} (with $\nu = \delta_x$ and 
	$\boldsymbol{a} 
	\equiv \alpha$) we deduce that (see \cite[eq. (43)]{cormier2020meanfield} for 
	more details):
	\begin{equation}
		\label{eq:link Theta Xi and r}
		\Theta_\alpha = \Xi_\alpha + r_\alpha * \Xi_\alpha.
	\end{equation} 
	Note that $\Psi_\alpha(0) = 0$, $\lim_{t \rightarrow \infty} \Psi_\alpha(t) = 0$ and 
	so
	$\widehat{\Xi}_\alpha(0) = 0$. Let
	\[ \xi_\alpha(t) := r_\alpha(t) - \gamma(\alpha). \]
	Using \eqref{eq:the jump rate converge with alpha} (with $\nu = \delta_0$), we 
	deduce that  
	$\xi_\alpha \in L^1(\mathbb{R}_+)$. So \eqref{eq:link Theta Xi and r} yields
	\begin{equation}
		\label{eq:08-20-21}
		 \Theta_\alpha = \Xi_\alpha + \gamma(\alpha) \Psi_\alpha + \xi_\alpha * \Xi_\alpha. 
	\end{equation}
	We deduce that  $\widehat{\Theta}_\alpha(0) = \gamma(\alpha) 
	\widehat{\Psi}_\alpha(0)$.
	Finally, we have
	\begin{align*}
		\frac{d}{d \alpha} \widehat{H}_\alpha(0) &= \int_0^\infty{ \frac{d}{d \alpha} 
		H_\alpha(t) dt} \\
		&= - \int_0^\infty{ H_\alpha(t) \int_0^t{ \frac{f(\varphi^\alpha_t(0)) - 
		f(\varphi^\alpha_\theta(0))}{b(\varphi^\alpha_\theta(0)) + \alpha}  d \theta} dt} \\
		&= - \int_0^\infty{ \int_0^\infty{ H_\alpha(u + \theta)  
		\frac{f(\varphi^\alpha_{u+\theta}(0)) - 
					f(\varphi^\alpha_\theta(0))}{b(\varphi^\alpha_\theta(0)) + \alpha}  } 
					d\theta du} \quad \text{ (using 
			Fubini} \\
		& \qquad \text{ and the change of variables $u = t- \theta$).}\\
		&= - \frac{ \widehat{\Psi}_\alpha(0)}{\gamma(\alpha)}.
	\end{align*}
	This ends the proof.
\end{proof}

The paper is structured as follows: in Section~\ref{sec:main result}, we state the 
spectral 
assumptions and the main result, Theorem~\ref{th:main result hopf}. We give a 
layout of its proof at 
the end of Section~\ref{sec:main result}. In Section~\ref{sec:proof of main result}, 
we give the 
proof of 
Theorem~\ref{th:main result hopf}. Finally, in Section~\ref{sec:toy 
	model}, we give an explicit example of a drift $b$ and a rate function $f$ for 
	which 
such Hopf 
bifurcations occur and the spectral assumptions can be analytically checked.

\section{Assumptions and main result}
\label{sec:main result}
Following \cite{CTV,cormier2020meanfield}, we assume that the law of the initial 
condition belongs to
\[ \mathcal{M}(f^2) := \{ \nu \in \mathcal{P}(\mathbb{R}_+): \quad 
\int_{\mathbb{R}_+}{f^2(x) \nu(dx)} < \infty  \}. \]
For such initial condition, under Assumptions~\ref{as:linear drift} and \ref{as:hyp on 
f}, 
the non-linear SDE \eqref{NL-equation} has a unique path-wise solution (see 
\cite[Th. 
9]{cormier2020meanfield}).
\begin{definition}
	\label{def:def of a periodic solution}
	A family of probability measures $(\nu(t))_{t \in [0,T]}$ is said to be a 
	$T$-periodic 
	solution of \eqref{NL-equation} if
	\begin{enumerate}
		\item $\nu(0) \in \mathcal{M}(f^2)$.
		\item For all $t \in [0,T]$, $\nu(t) = \mathcal{L}(X_t)$ where \((X_t)_{t \in [0,T]}\) 
		is 
		the solution of 
		\eqref{NL-equation} starting from $X_0 \sim \nu(0)$.
		\item It holds that $\nu(T) = \nu(0)$.
	\end{enumerate}
\end{definition}
In this case, we can obviously extend \((\nu(t))\) for \(t\in\mathbb{R}\) by periodicity.
Considering now the solution \((X_t)_{t\geq 0}\) of 
\eqref{NL-equation} defined for 
\(t\geq 0\), it remains true that \(\nu(t) = \mathcal{L}(X_t)\) for any \(t\geq 0\).

We study the existence of periodic solutions $t \mapsto \mathcal{L}(X_t)$ where
$(X_t)$ is the solution of \eqref{NL-equation},  near a non-stable invariant
measure $\nu^\infty_{\alpha_0}$.
We assume that the stability criterion \eqref{eq:criteria_stability} is not
	satisfied for $\alpha_0$:
\begin{assumption}
	\label{ass:i / b0 roots of the characteristic equation}
	Assume that there exists $\alpha_0 > 0$ and $\tau_0 > 0$ such that
	\[ J(\alpha_0)  \widehat{\Theta}_{\alpha_0} (\tfrac{i}{\tau_0}) = 1 \quad \text{ and } 
	\quad \frac{d}{dz} \widehat{\Theta}_{\alpha_0} (\tfrac{i}{\tau_0}) \neq 0 . \]
\end{assumption}
\begin{assumption}[Non-resonance condition]
	\label{ass:nonresonance condition}
	Assume that for all $n \in \mathbb{Z} \backslash \{ -1, 1\} $,
	\[  J(\alpha_0) \widehat{\Theta}_{\alpha_0} (\tfrac{i n}{\tau_0}) \neq 1.  \]
\end{assumption}
\begin{remark}[Local uniqueness of the invariant measure in the neighborhood of
	$\alpha_0$]
	\label{rk:uniqueness of the invariant measure around alpha0}
	Under Assumption~\ref{ass:nonresonance condition}, we have  $J(\alpha_0) 
	\widehat{\Theta}_\alpha(0) \neq 1$ and so, by Lemma \ref{lem:link between the 
	derivative of J and 
		Theta}, it holds that $J'(\alpha_0) \neq 0$. 
		Fix $J$ in the neighborhood of $J(\alpha_0)$.
	Recall that the values of \(\alpha\) such that \(\nu^\infty_\alpha\) is an invariant 
	measure of \eqref{NL-equation} are precisely the solutions of \(J(\alpha) = J\).
	So, in the neighborhood of $\alpha = \alpha_0$, the invariant measure of 
	\eqref{NL-equation} is 
	unique.
\end{remark}
\begin{lemma}
	\label{lem:definition of mu(alpha)}
	Under Assumption~\ref{ass:i / b0 roots of the characteristic equation}, there 
	exists
	$\eta_0, \varrho_0 > 0$ and a function $\mathfrak{Z}_0 \in \mathcal{C}^1((\alpha_0 
	-
	\eta_0,
	\alpha_0 + \eta_0); \mathbb{C})$ with $\mathfrak{Z}_0(\alpha_0) = \frac{i}{\tau_0}$ 
	such
	that for all
	$z \in \mathbb{C}$ with $|z - \frac{i}{\tau_0}| < \varrho_0$ and for all $\alpha > 0$ 
	with
	$|\alpha -
	\alpha_0| < \eta_0$ we have
	\begin{equation}
		\label{eq:link between Theta and Z0}
		J(\alpha) \widehat{\Theta}_{\alpha}(z) = 1 \iff z = \mathfrak{Z}_0(\alpha). 
	\end{equation} 
\end{lemma}
\begin{proof} It suffices to apply the implicit function theorem to $(\alpha, z) 
\mapsto
	J(\alpha) \widehat{\Theta}_{\alpha}(z) - 1$.
\end{proof}
\begin{assumption}
	\label{ass:hopf condition}
	Assume that $\alpha \mapsto \mathfrak{Z}_0(\alpha)$ crosses the imaginary part 
	with
	non-vanishing speed, that is
	\[  \Re \mathfrak{Z}_0'( \alpha_0) \neq 0, \quad \text{ where } \quad 
	\mathfrak{Z}_0'(\alpha)
	= \frac{d}{d \alpha} \mathfrak{Z}_0 (\alpha). \]
\end{assumption}
\begin{remark}
	Using \eqref{eq:link between Theta and Z0}, Assumption~\ref{ass:hopf condition} 
	is 
	equivalent to
	\[ \Re \left( \frac{ \frac{\partial }{\partial \alpha }   \left. \left(   J(\alpha) 
		\widehat{\Theta}_{\alpha} \right)
		\right|_{\alpha = \alpha_0}  (\tfrac{i}{\tau_0}) }{ J(\alpha_0) \frac{\partial}{\partial 
		z} 
		\widehat{\Theta}_{\alpha_0}(\frac{i}{\tau_0})} \right) \neq 0. \]
\end{remark}
Our main result is the following.
\begin{theorem}
	\label{th:main result hopf}
	Consider $b, f$ satisfying Assumptions~\ref{as:linear drift} and 
	\ref{as:hyp on f}.
		Let $\alpha_0, \tau_0 > 0$ be such that Assumptions~\ref{hyp:alpha is non 
		degenerate}, \ref{ass:i / b0 roots 
			of 
			the 
			characteristic equation}, \ref{ass:nonresonance
			condition} and \ref{ass:hopf condition} hold.
	Then, there exists a family of  $2 \pi \tau_v$-periodic
	solutions of \eqref{NL-equation}, parametrized by $v \in (-v_0, v_0)$, for some 
	$v_0 > 0$.
	More precisely, there exists a continuous curve $\{ (\nu_v(\cdot), \alpha_v,  
	\tau_v),~ v \in (-v_0, v_0) \}$ such that
	\begin{enumerate}
		\item For all $v \in (-v_0, v_0)$, $(\nu_v(t))_{t \in \mathbb{R}}$ is a $2 \pi 
		\tau_v$-periodic solution of \eqref{NL-equation} with $J = J(\alpha_v)$.
		\item The curve passes through $(\nu^\infty_{\alpha_0}, \alpha_0, \tau_0)$ at $v 
		= 0$. In particular we have for all $t \in \mathbb{R}$, $\nu_0(t) \equiv 
		\nu^\infty_{\alpha_0}$.
		\item The ``periodic current'' $\boldsymbol{a}_v$, defined 
		by
		\begin{equation}
			\label{eq:periodic current from nu_v}
			t \mapsto a_v(t) := J(\alpha_v) \int_{\mathbb{R}_+}{f(x) \nu_v(t, dx)},
		\end{equation}
		is continuous and $2 \pi \tau_v$-periodic. Moreover, its mean over one period
		is  $\alpha_v$:
		\[  \frac{1}{2 \pi \tau_v} \int_0^{2 \pi \tau_v}{ a_v(u) du } =  \alpha_v. \]
		\item Furthermore, $v$ is the amplitude of the first harmonic of 
		$\boldsymbol{a_v}$,
		that is 
		for all $v \in (-v_0, 
		v_0)$
		\[ \frac{1}{2 \pi \tau_v} \int_0^{2 \pi \tau_v}{ a_v(u)
			\cos (u / \tau_v) du} = v \quad \text{ and } \quad \frac{1}{2 \pi \tau_v} 
			\int_0^{2 \pi 
			\tau_v}{
			a_v(u)
			\sin (u / \tau_v) du} = 0. \]
	\end{enumerate}
	Every other periodic solution in a neighborhood of $\nu^\infty_{\alpha_0}$ is 
	obtained
	from a phase-shift of one such $\nu_v$. More precisely, there exists small enough
	constants $\epsilon_0, \epsilon_1 > 0$ (only depending on $b, f, \alpha_0$ and 
	$\tau_0$)
	such that if $(\nu(t))_{t \in \mathbb{R}}$
	is any $2 \pi \tau$-periodic solution of \eqref{NL-equation} for some value of $J > 
	0$
	such that
	\[  |\tau - \tau_0| < \epsilon_0 \quad \text{ and } \quad \sup_{t \in [0, 2 \pi \tau]} 
	\left|  J
	\int_{\mathbb{R}_+}{f(x) \nu(t, dx)} - \alpha_0 \right| < \epsilon_1,  \]
	then there exists a shift $\theta \in [0, 2 \pi \tau)$ and $v \in (-v_0, v_0)$ such 
	that $J =
	J(\alpha_v)$ and
	\[  \forall t \in \mathbb{R}, \quad  \nu(t) = \nu_v( t + \theta). \]
\end{theorem}
\begin{remark}
	Given the ``periodic current'' $\boldsymbol{a}_v$ defined by \eqref{eq:periodic 
	current
		from nu_v}, the shape of the solution is known explicitly: for all $v \in (-v_0, 
		v_0)$, it
	holds that
	\[ \nu_v = \tilde{\nu}_{\boldsymbol{a}_v}, \]
	where $\tilde{\nu}_{\boldsymbol{a}_v}$, defined by \eqref{eq:definition de nu infty 
	a for periodic a} below, is known explicitly in terms of $b, f$ and 
	$\boldsymbol{a}_v$.
\end{remark}
\begin{notation}
	\label{notation:CT0 and CT00}
	For $T > 0$, we denote by $C^0_{T}$ the space of continuous and $T$-periodic 
	functions from $\mathbb{R}$ to $\mathbb{R}$ and by $C^{0,0}_{T}$ the 
	subspace of centered
	functions
	\[ C^{0,0}_{T} := \{ h \in C^0_{T}, \quad \int_{0}^{T}{ h(t) dt} = 0\}. \]
\end{notation}
We now give an outline of the proof of Theorem~\ref{th:main result hopf}. The proof 
is 
divided in two main parts. 

The first part is devoted to the study of an isolated neuron 
subject to a periodic external current. 
That is, given $\tau > 0$ and $\boldsymbol{a} \in C^0_{2 \pi \tau}$, we study the 
jump rate 
of an isolated neuron driven by $\boldsymbol{a}$. We give in 
Section~\ref{sec:priliminaries} 
estimates on the kernels $K_{\boldsymbol{a}}$ 
and $H_{\boldsymbol{a}}$.
We want to characterize the ``asymptotic'' jump rate of a neuron driven by this 
external 
periodic 
current. That is, informally
\[
\forall t \in \mathbb{R},\quad \rho_{\boldsymbol{a}}(t) = \lim_{k \in \mathbb{N},~k 
	\rightarrow 
	\infty} \r[][\boldsymbol{a}](t, -2 \pi k \tau).
\]
In order to characterize such limit $\rho_{\boldsymbol{a}}$, we introduce in 
Section~\ref{sec:non 
	homogeneous oscillations} a discrete-time Markov 
Chain corresponding to the phases of the successive spikes of the neuron driven by 
$\boldsymbol{a}$. We prove that this Markov Chain has a unique invariant measure, 
which is 
proportional to $\rho_{\boldsymbol{a}}$. This serves as a definition of 
$\rho_{\boldsymbol{a}}$.
Given this periodic jump rate $\rho_{\boldsymbol{a}} \in C^0_{2 \pi \tau}$, we give in 
Section~\ref{sec:periodic densities} an explicit description of the associated  
time-periodic probability densities, 
that we denote $(\tilde{\nu}_{\boldsymbol{a}}(t))_{t \in [0, 2 \pi \tau]}$.
Consequently, to find a $2 \pi \tau$-periodic solution of \eqref{NL-equation}, it is 
equivalent to find
$\boldsymbol{a} \in C^0_{2 \pi \tau}$ such that
\begin{equation}
	\boldsymbol{a} = J \rho_{\boldsymbol{a}}.
	\label{eq:periodic solution}
\end{equation}
One classical difficulty with Hopf bifurcation is that the period $2 \pi \tau$ itself is 
unknown: $\tau$
varies when the
interaction parameter $J$ varies. To address this problem, we make in 
Section~\ref{sec:reduction 2pi} a change of time to only consider $2\pi$-periodic 
functions. We define for all $t \in \mathbb{R}$
	\begin{equation}
		\label{eq:link rho a and rho a tau}
		\forall \boldsymbol{d} \in C^0_{2 \pi},\forall \tau > 0,\quad \rho_{\boldsymbol{d}, 
		\tau}(t) :=  
		\rho_{\boldsymbol{a}}(\tau t), \quad \text{ where } \quad a(t) := d(t / \tau). 
\end{equation} 
We shall see that this change of time has a simple probabilistic interpretation by 
scaling 
$b, f$ and $\boldsymbol{d}$ appropriately. In Section ~\ref{sec:regularity of rho}, we 
prove that the function $C^0_{2 \pi} \times \mathbb{R}^*_+ \ni (\boldsymbol{d}, \tau) 
\mapsto 
\rho_{\boldsymbol{d}, \tau} \in C^0_{2 
	\pi}$ is $\mathcal{C}^2$-Fréchet differentiable. Furthermore, 
	consider	\(\boldsymbol{d} = \alpha + \boldsymbol{h}\) with \( \boldsymbol{h} \in 
	C^{0,0}_{2 \pi}\) and $\alpha > 0$ is the mean 
	of $\boldsymbol{d}$ over one period.
	We prove that the mean number of spikes over one period only depends on 
	$\alpha$.
	The common value is obtained with the particular case \(\boldsymbol{h}\equiv 0\) 
	and 
	\eqref{eq:gamma alpha}. Thus, we prove
\begin{equation}
	\label{eq:mean over a period}
	\forall \boldsymbol{h} \in C^{0,0}_{2 \pi}, \forall \tau > 0, \quad 
	\frac{1}{2 
		\pi}\int_0^{2 \pi}{ \rho_{\alpha 
			+ \boldsymbol{h}, \tau}(u) 
		du} =  \gamma(\alpha).
\end{equation}
In the second part of the proof, we find self-consistent periodic solutions using the 
Lyapunov-Schmidt method. We introduce in Section~\ref{sec:fonctionnal G} the 
following 
functional
\[ C^{0,0}_{2 \pi} \times \mathbb{R}^*_+ \times \mathbb{R}^*_+ \ni (\boldsymbol{h}, 
\alpha, \tau) 
\mapsto 
G(\boldsymbol{h}, \alpha, \tau) := 
(\alpha + \boldsymbol{h}) - J(\alpha) \rho_{\alpha + \boldsymbol{h}, \tau}.\]
Using \eqref{eq:mean over a period}, this functional takes values in 
$C^{0,0}_{2 
	\pi}$. The roots of $G$, described by Proposition~\ref{prop:the zeros of G}, 
	match 
with 
the periodic solutions of \eqref{NL-equation}. For instance if $G(\boldsymbol{h}, 
\alpha, 
\tau) = 
0$, we set $a(t) := \alpha + h(t / \tau)$. This current $\boldsymbol{a}$ 
solves 
\eqref{eq:periodic 
	solution} with $J = J(\alpha)$ and so it can be used to define a  periodic solution 
	of  
\eqref{NL-equation}. Conversely, to any periodic solution of 
\eqref{NL-equation}, we can associate a root of $G$.
So Theorem~\ref{th:main result hopf} is equivalent to Proposition~\ref{prop:the 
zeros of 
	G}.
Sections~\ref{sec:linearization of G}, \ref{sec:fredholm operator},  
\ref{sec:lyapunov} and \ref{sec:2D} are then devoted to the proof of
Proposition~\ref{prop:the zeros of G}.
In Section~\ref{sec:linearization of G}, we prove that the linear operator 
$D_{\boldsymbol{h}}  G(0, \alpha, 
\tau)$ can be 
written using a convolution involving $\Theta_\alpha$, given by \eqref{eq:formule 
donnant 
	Theta.}. 
We then follow the method of \cite[Ch. 
I.8]{MR2859263}. In Section \ref{sec:fredholm operator}, we study the range and the 
kernel of 
$D_{\boldsymbol{h}}  G(0, \alpha_0, 
\tau_0)$: we prove that under the spectral Assumptions~\ref{ass:i / b0 roots of the 
	characteristic equation} and \ref{ass:nonresonance condition}, 
	$D_{\boldsymbol{h}}  
G(0, \alpha_0, 
\tau_0)$ is a Fredholm operator of index zero, with a kernel of dimension two. 
The problem of finding the roots of $G$ is a priori of infinite dimension 
($\boldsymbol{h}$ 
belongs 
to $C^{0,0}_{2 \pi}$).  In Section~\ref{sec:lyapunov} we apply the Lyapunov-Schmidt 
method 
to obtain an 
equivalent problem of dimension two. Finally in Section~\ref{sec:2D} we study the 
reduced 
2D-problem.

\section{Proof of Theorem~\ref{th:main result hopf}}
\label{sec:proof of main result}
\textit{Without risk of confusion, we alleviate the notation in the proofs: 
		we no more 
		use bold letters for small 
		perturbations \(h\) of a constant current \(\alpha_0\).}
\subsection{Preliminaries}
\label{sec:priliminaries}
Let $T > 0$, $s \in \mathbb{R}$ and  $\boldsymbol{a} \in C^0_T$ such that
\begin{equation}
	\label{eq:condition on b(0)}
	\inf_{t  \in [0, T]} a_t > -b(0). 
\end{equation}
For $x \geq 0$, we consider $\varphi^{\boldsymbol{a}}_{t, s}(x)$ the solution of the 
ODE
\begin{align}
	\label{eq:definition du flow}
	\frac{d}{dt} \varphi^{\boldsymbol{a}}_{t, s}(x) &= b(\varphi^{\boldsymbol{a}}_{t, 
	s}(x)) + a_t \\
	\varphi^{\boldsymbol{a}}_{s, s}(x)  &= x. \nonumber
\end{align}
By Assumption~\ref{as:linear drift}, this ODE has a unique solution. Moreover, the 
kernels $H^{\nu}_{\boldsymbol{a}}(t, s)$ and $K^{\nu}_{\boldsymbol{a}}(t, s)$, 
defined by \eqref{def:jump rate, survival, density of survival}, have  explicit 
expressions in term of the flow
\begin{align}
	\label{eq:explicit expression H}
	H^\nu_{\boldsymbol{a}}(t, s) &= 
	\int_{\mathbb{R}_+}{\EXP{-\int_s^t{f(\varphi^{\boldsymbol{a}}_{u,s}(x)) du}} 
	\nu(dx) }, \\
	K^\nu_{\boldsymbol{a}}(t, s) &= \int_{\mathbb{R}_+}{ f(\varphi^{\boldsymbol{a}}_{t, 
	s}(x)) \EXP{-\int_s^t{f(\varphi^{\boldsymbol{a}}_{u,s}(x)) du}} \nu(dx) }.
	\label{eq:explicit expression K}
\end{align}
The function $s \mapsto \varphi^{\boldsymbol{a}}_{t, s}(0)$ belongs to 
$\mathcal{C}^1((-\infty, t]; 
\mathbb{R}_+)$ and
\begin{equation}
	\label{eq: d/ds du flow}
	\frac{d}{ds}  \varphi^{\boldsymbol{a}}_{t, s}(0) = -\left[ b(0) + a_s \right] \EXP{ 
	\int_s^t{
			b'(\varphi^{\boldsymbol{a}}_{\theta, s}(0) ) d\theta }}.
\end{equation}
In particular, under the assumption \eqref{eq:condition on b(0)}, $s 
\mapsto \varphi^{\boldsymbol{a}}_{t, s}(0)$ is strictly decreasing on $(-\infty, t]$, for 
all $t$.  
Define then
\begin{equation}
	\sigma_{\boldsymbol{a}}(t) := \lim_{s \rightarrow -\infty} 
	\varphi^{\boldsymbol{a}}_{t, s}(0) \in 
	\mathbb{R}^*_+ \cup \{+ \infty \}.
	\label{def:sigma_a for a periodic.}
\end{equation}
Given $\boldsymbol{d} \in C^0_T$ and $\eta > 0$, we consider the following open balls 
of $C^0_T$:
\begin{equation}
	\label{eq:open ball of C_T}
	B^T_\eta(\boldsymbol{d}) := \{ \boldsymbol{a} \in C^0_T,\quad \sup_{t \in [0, T]} 
	|a_t - 
	d_t| < \eta \}.
\end{equation}
\begin{lemma}
	\label{lem: prop of sigma_a near alpha_0}
	Let $T >0$ and $b: \mathbb{R}_+ \rightarrow \mathbb{R}$ such that
	Assumption~\ref{as:linear drift} holds. Let $\alpha_0 > 0$ satisfying
	Assumption~\ref{hyp:alpha is non degenerate}. There exists $\eta_0 > 0$ such 
	that 
	for
	all $\boldsymbol{a} \in B^T_{\eta_0}(\alpha_0)$, it holds that
	\begin{enumerate}
		\item If $\sigma_{\alpha_0} = \infty$, then for all $t \in [0, T]$, 
		$\sigma_{\boldsymbol{a}}(t) = +\infty$.
		\item If $\sigma_{\alpha_0} < \infty$, then the function $t \mapsto 
		\sigma_{\boldsymbol{a}}(t)$
		belongs to $C^0_T$ and
		\[ \inf_{\boldsymbol{a} \in B^T_{\eta_0}(\alpha_0)}  \inf_{t \in [0, T]}
		\sigma_{\boldsymbol{a}}(t) > 0. \]
	\end{enumerate}
\end{lemma}
\begin{proof}
	Assume first that $\sigma_{\alpha_0} = \infty$, and let $\eta_0 := \frac{1}{2} \left[ \inf_{x 
	\geq 0}
	b(x) + \alpha_0 \right]$, which is strictly positive by assumption. Then it holds that
	$ \inf_{t \geq 0} \inf_{x \geq 0} b(x) + a_t \geq \eta_0 $
	and so
	\begin{equation}
		\label{eq:the flow grows at least linearly when sigma is infinite}
		\varphi^{\boldsymbol{a}}_{t, s}(0) \geq \eta_0 (t-s).
	\end{equation}
	Letting $s$ tend to $-\infty$, we deduce that 
	$\sigma_{\boldsymbol{a}}(t) = +\infty$.\\
Assume now that $\sigma_{\alpha_0} < \infty$. Using 
\eqref{hyp: sigma_alpha is finite},  we apply the implicit function theorem to
\[ 
(x, \eta) \mapsto b(x) + \alpha_0 + \eta
\]
at $(\sigma_{\alpha_0}, 0)$: there exists $\eta_0 > 0$ 
and a function $\eta \mapsto 
\sigma_{\alpha_0+\eta} \in \mathcal{C}^1([0, \eta_0]; \mathbb{R}^*_+)$ such that
\[ \forall \eta \in [0, \eta_0],\quad  \sigma_{\alpha_0} \leq \sigma_{\alpha_0 + \eta} < \infty \quad 
\text{ and } \quad  \sigma_{\alpha_0 + \eta} = \inf_{x \geq 0} \{ b(x) + \alpha_0 + \eta = 0 \}. \]
In addition, we choose $\eta_0$ such that $\eta_0 < b(0) + \alpha_0$.
Let $\boldsymbol{a} \in C^0_T$ such 
that
$\sup_{t \in [0, T]} |a_t - \alpha_0| \leq \eta_0$. It holds that
\begin{equation} \forall t \geq s,\quad  \varphi^{\boldsymbol{a}}_{t, s}(0) \leq \varphi^{\alpha_0 + 
		\eta_0}_{t, s}(0) \leq \sigma_{\alpha_0 + \eta_0}. 
	\label{eq:17032021}
\end{equation}
In particular $\sigma_{\boldsymbol{a}}(t) < \infty$. We prove that this function is right-continuous in 
$t$. We fix $t \geq s$ and $\eta \in [0, \eta_0]$, 
we have
\begin{align*} \varphi^{\boldsymbol{a}}_{t + \eta,s}(0) - \varphi^{\boldsymbol{a}}_{t ,s}(0) 
	&= \int_t^{t+\eta}{ b(\varphi^{\boldsymbol{a}}_{t+u, s}(0)) du} + \int_t^{t+\eta}{a_u du}.
\end{align*}
Let $A_0 := \sup_{x \in [0, \sigma_{\alpha_0 + \eta_0}]} |b(x)| < \infty$ so $|\varphi^{\boldsymbol{a}}_{t + \eta,s}(0) 
- \varphi^{\boldsymbol{a}}_{t ,s}(0)| \leq (A_0  
+ ||\boldsymbol{a}||_\infty) \eta$. 
Letting $s$ tend to $-\infty$ we deduce that $t \mapsto \sigma_{\boldsymbol{a}}(t)$ is
right-continuous. Left-continuity is proved similarly. Using 
$\varphi^{\boldsymbol{a}}_{t+T, s+T}(0) =
\varphi^{\boldsymbol{a}}_{t, s}(0)$, it holds that $t \mapsto
\sigma_{\boldsymbol{a}}(t)$ is $T$-periodic. Finally, because $s \mapsto
\varphi^{\boldsymbol{a}}_{t, s}(0)$ is strictly decreasing, and takes value $0$ when $s =
t$, we deduce that $\sigma_{\boldsymbol{a}}(t) > 0$. More precisely, let
\[ m_0 := - \min_{x \in [0, \sigma_{\alpha_0+\eta_0}]} b'(x).   \]
By \eqref{hyp: sigma_alpha is finite}, it holds that $m_0 > 0$. Moreover, using \eqref{eq: d/ds du 
	flow}, we deduce that
\[ \frac{d}{ds} \varphi^{\boldsymbol{a}}_{t, s}(0) \leq -(b(0) + \alpha_0 - \eta_0)
e^{-m_0(t-s)}, \]
and so
\begin{equation}
	\label{eq:minoration of the flow when sigma is finite}
	\forall s \leq t, \quad  \varphi^{\boldsymbol{a}}_{t, s}(0) \geq (b(0)+\alpha_0-\eta_0)
	\frac{1 - e^{-m_0(t-s) }}{m_0}.
\end{equation}
It ends the proof.
\end{proof}
\begin{lemma}
	\label{lem:estimates on the kernels}
	Grant Assumptions~\ref{as:linear drift} and \ref{as:hyp on f}. Let $\alpha_0 > 0$ 
	such
	that  Assumption~\ref{hyp:alpha is non degenerate} holds. There exists 
	$\lambda_0,
	\eta_0, s_0 > 0$ (only depending on $\alpha_0$ and $b$) such that
	for all $T > 0$, for all $\boldsymbol{a} \in B^T_{\eta_0}(\alpha_0)$, it holds that
	\begin{equation}
		\label{eq:210321}
		\forall t, s,\quad t-s \geq s_0 \implies \varphi^{\boldsymbol{a}}_{t, s}(0) \geq
		\lambda_0. 
	\end{equation}
	Moreover, if $\sigma_{\alpha_0} = \infty$, $\lambda_0$ can be chosen arbitrarily 
	large.
	Finally, it holds that
	\begin{align}
		\label{eq:210321:H}
		\sup_{T > 0} ~~ \sup_{\boldsymbol{a} \in B^T_{\eta_0}(\alpha_0)}  ~~ \sup_{t 
		\geq s}
		~~
		&
		H_{\boldsymbol{a}}(t, s) e^{f(\lambda_0)(t-s)} < \infty, \\
		\sup_{T > 0} ~~ \sup_{\boldsymbol{a} \in B^T_{\eta_0}(\alpha_0)}  ~~ \sup_{t 
		\geq s}
		~~
		&
		K_{\boldsymbol{a}}(t, s)  e^{f(\lambda_0)(t-s)} < \infty.
		\label{eq:210321:K}
	\end{align}
\end{lemma}
\begin{proof}
		The inequality~\eqref{eq:210321} is a direct consequence of 
		\eqref{eq:minoration 
			of the flow when sigma is finite} if $\sigma_{\alpha_0} < \infty$ and of 
			\eqref{eq:the flow 
			grows at least linearly when sigma is infinite} if $\sigma_{\alpha_0} = \infty$. 
			Then, 
		\eqref{eq:210321:H} follows from the explicit expression \eqref{eq:explicit 
		expression H} of 
		$H_{\boldsymbol{a}}$ and \eqref{eq:210321}. 
		
		If $\sigma_{\alpha_0} < \infty$ or $f$ is bounded on $\mathbb{R}_+$, 
		\eqref{eq:210321:K} 
		follows from \eqref{eq:210321:H} by \eqref{eq:explicit expression K}.
		
		It remains to prove \eqref{eq:210321:K} if $\sigma_{\alpha_0} = \infty$ and 
		$\lim_{x \rightarrow 
			\infty}f(x) = \infty$. Denote by $L$ the Lipschitz constant of $b$. The 
		Grönwall's lemma 
		gives the existence of a constant $C_1$ (only depending on $b$, $\alpha_0$, 
		$\eta_0$) such 
		that
		$\varphi^{\boldsymbol{a}}_{t, s}(0) \leq C_1 e^{L(t-s)}$. By 
		Remark~\ref{rk:comportement at infty}, we 
		know that $f(x) \leq C_2 x^p$. Overall, $f(\varphi^{\boldsymbol{a}}_{t, s}(0)) 
		\leq C e^{Lp(t-s)}.$ 
		Fix $\lambda_0 \geq 0$. There exists $\lambda_1$ such that $f(\lambda_1) \geq 
		f(\lambda_0) + Lp$ 
		and \eqref{eq:210321} holds with $\lambda_1$. It ends the proof of 
		\eqref{eq:210321:K}.
\end{proof}
\subsection{Study of the non-homogeneous linear equation}
\label{sec:non homogeneous oscillations}
In this section, we study the asymptotic jump rate of an ``isolated'' neuron driven by 
a periodic 
continuous function.
Grant Assumptions~\ref{as:linear drift}, \ref{as:hyp on f} and let
$\alpha_0 > 0$ such that Assumption~\ref{hyp:alpha is non degenerate} holds.
Let $\lambda_0, \eta_0 > 0$ be given by Lemma~\ref{lem:estimates on the kernels} 
and
$T > 0$. Consider $\boldsymbol{a} \in B^T_{\eta_0}(\alpha_0)$. Following the
terminology of \cite{CTV}, we say that $\boldsymbol{a}$ is the ``external current''.
Let $\r[][\boldsymbol{a}]$ be the solution of the Volterra equation 
$\r[][\boldsymbol{a}] = 
\K[][\boldsymbol{a}] + \K[][\boldsymbol{a}] * \r[][\boldsymbol{a}]$.
We consider the following limit
\[ \forall t \in [0, T], \quad \rho_{\boldsymbol{a}}(t) = \lim_{k \in \mathbb{N},~k 
\rightarrow \infty} 
r_{\boldsymbol{a}}(t, -kT). \]
The goal of this section is to show that $\rho_{\boldsymbol{a}}$ is well defined and 
to 
study some 
of its properties.
First, \eqref{eq:the volterra integral equation} and \eqref{eq:the time of the last spike 
	density 
	sum to 1} write 
\begin{align*}
	\forall t\in \mathbb{R}, \quad \r[][\boldsymbol{a}](t, -k T) &= 
	\K[][\boldsymbol{a}](t, 
	-kT) + \int_{-kT}^t{ \K[][\boldsymbol{a}](t, s) \r[][\boldsymbol{a}](s, -kT) ds }, \\
	1 &= \H[][\boldsymbol{a}](t, -kT) + \int_{-kT}^t{ \H[][\boldsymbol{a}](t, s) 
	\r[][\boldsymbol{a}](s, -kT) ds}.
\end{align*}
Letting $k \rightarrow \infty$, we find that $\rho_{\boldsymbol{a}}$ has to solve
\begin{align}
	\forall t \in \mathbb{R}, \quad  \rho_{\boldsymbol{a}}(t) &= \int_{-\infty}^t{ 
	\K[][\boldsymbol{a}](t, s) \rho_{\boldsymbol{a}}(s) ds. }
	\label{eq:asymptotic periodic jump rate} \\
	1 &= \int_{-\infty}^t{ \H[][\boldsymbol{a}](t, s) \rho_{\boldsymbol{a}}(s) ds}.
	\label{eq:asymptotic periodic jump rate normalization condition}
\end{align}
Note that if $\rho_{\boldsymbol{a}}$ is a solution of \eqref{eq:asymptotic periodic 
jump rate}, then it automatically holds that the function $t \mapsto  \int_{-\infty}^t{ 
\H[][\boldsymbol{a}](t, s) \rho_{\boldsymbol{a}}(s) ds}$ is constant (its derivative is 
null).
In Lemma~\ref{lem:space of solution is of dimension 1} below, we prove that the 
solutions
of
equation~\eqref{eq:asymptotic periodic jump rate} form a linear space of dimension 
1.
Consequently \eqref{eq:asymptotic periodic jump rate} together with
\eqref{eq:asymptotic periodic jump rate normalization condition} have a unique 
solution: this will 
serve as the definition of  $\rho_{\boldsymbol{a}}$.

\subsubsection*{A probabilistic interpretation of \eqref{eq:asymptotic periodic 
jump rate} and \eqref{eq:asymptotic periodic jump rate normalization condition}}
Let $x$ be a $T$-periodic solution of \eqref{eq:asymptotic periodic jump rate}. We 
have for all $t \in  
[0, T]$
\begin{align*}
	x(t) = & \int_{-\infty}^T{ \K[][\boldsymbol{a}](t, s) x(s) ds }, \quad \text{(with the convention } 
	K_{\boldsymbol{a}}(t, s) = 0 \text{ for $s > t$)} \\
	= & \sum_{k \geq 0} { \int_{-kT}^{T - kT } {   \K[][\boldsymbol{a}](t, s) x(s) ds } } \\
	= & \sum_{k \geq 0} { \int_{0}^{T} {   \K[][\boldsymbol{a}](t, u - kT) x(u) du } } 
	\quad \text {(by 
		the change of variable $u = s + k T$)}.
\end{align*}
Define for all $t,s \in [0, T]$
\[ K^T_{\boldsymbol{a}}(t, s) := \sum_{k \geq 0}{ \K[][\boldsymbol{a}](t, s - k T) }. \]
Note that by Lemma~\ref{lem:estimates on the kernels} we have normal convergence 
since
\[ \forall t,s \in [0, T], \quad K_{\boldsymbol{a}}(t, s - kT) \leq C e^{-f(\lambda_0)k 
 T  }, \]
for some constant $C$ only depending on $b,f$, $\alpha_0, \eta_0$ and $\lambda_0$.
We deduce that $x$ solves
\begin{equation}
	x(t) = \int_0^T{ K^T_{\boldsymbol{a}}(t, s) x(s) ds}.
	\label{eq:limit equation periodic on [0,T]}
\end{equation}
Using that $\boldsymbol{a}$ is $T$-periodic, we have
\begin{equation}
	\forall t \geq s,\quad K_{\boldsymbol{a}}(t + T, s+T) = K_{\boldsymbol{a}}(t,s). 
	\label{eq: Ka t+T, s+T}
\end{equation}
Moreover, $K_{\boldsymbol{a}}$ is a probability density so
\begin{equation} \forall s \in \mathbb{R},\quad \int_s^\infty{ K_{\boldsymbol{a}}(t, s) 
dt} = 1. 
	\label{eq:Ka density}
\end{equation}
From \eqref{eq: Ka t+T, s+T} and \eqref{eq:Ka density}, we deduce that
\begin{equation} 
	\forall s \in [0, T],\quad \int_0^T{ 
		K^T_{\boldsymbol{a}}(t, s) dt} = 1.
	\label{eq:identity integral K_T}
\end{equation}
In view of \eqref{eq:identity integral K_T}, $K^T_{\boldsymbol{a}}( \cdot , s) $ can be 
seen 
as the transition probability kernel of a Markov Chain acting on the continuous 
space $[0, 
T]$. The interpretation of this Markov Chain is the following. Let 
$(\Y[\boldsymbol{a}][\nu][t])_{t \geq 0}$ be the solution of 
\eqref{non-homegeneous-SDE}, starting at time $0$ with law $\nu$ and driven by 
the 
$T$-periodic current $\boldsymbol{a}$. Define $(\tau_i)_{i \geq 1}$ the times of the 
successive jumps of $(\Y[\boldsymbol{a}][\nu][t])_{t \geq 0}$. Let $\phi_i \in [0, T)$
and $\Delta_i \in \mathbb{N}$ be defined by:
\begin{equation}
	\label{eq:proba interpretation Markov}
	\phi_i := \tau_i - \floor*{\frac{\tau_i}{T}}T, \quad  \tau_{i+1} - \tau_i =: \Delta_{i+1}  T 
	+ \phi_{i+1} - \phi_i.
\end{equation}
That is, $\phi_i$ is the \textit{phase} of the $i$-ith jump, while $\Delta_i$ is the 
number of 
``revolutions'' between $\tau_{i-1}$ and $\tau_i$:
\[ \Delta_i = \# \{k \in \mathbb{N},\quad kT \in [\tau_{i-1}, \tau_i) \}. \]
In other words, if one considers that a period is a ``lap'', \(\Delta_i\) is the number of 
times 
we cross the start line of the lap between two spikes.

Then, $(\phi_i, \Delta_i)_{i \geq 0}$ is Markov, with a transition probability given by
\[ \forall A \in \mathcal{B}([0, T]),~\forall n \in \mathbb{N},\quad \mathbb{P}(\phi_{i+1} 
\in A, \Delta_{i+1} = n | \phi_i) = \int_A{ \K[][\boldsymbol{a}](t+ n T, \phi_i)  dt}.\]
\textbf{In particular, $(\phi_i)_{i \geq 0}$ is Markov, with a transition probability given 
by 
	$K^T_{\boldsymbol{a}}$}.
With some slight abuse of notations, we also write $K^T_{\boldsymbol{a}}$ for the 
linear operator which maps $y \in L^1([0, T])$ to
\begin{equation}
	\label{eq:notation as a linear operator}
	K_{\boldsymbol{a}}^T(y) := t \mapsto \int_0^T{ K^T_{\boldsymbol{a}}(t, s) y(s) ds } 
	\in
	L^1([0, T]).
\end{equation}
\begin{lemma}
	\label{lem:compact operator}
	Let $\boldsymbol{a} \in C^0_T$.
	The linear operator $K^T_{\boldsymbol{a}}: L^1([0, T]) \rightarrow L^1([0, T])$ is a 
	compact operator. Moreover, if $y \in L^1([0, T])$, then $K^T_{\boldsymbol{a}} (y) 
	\in 
	C^0_T$.
\end{lemma}
\begin{proof}
	First, the function $[0, T]^2 \ni (t, s) \mapsto K^T_{\boldsymbol{a}}(t, s)$ is 
	(uniformly)
	continuous. Let $\epsilon > 0$, there exists $\eta > 0$ such that
	\[ |t-t'| + |s-s'| \leq \eta \implies |K^T_{\boldsymbol{a}}(t, s) - 
	K^T_{\boldsymbol{a}}(t', s') |
	\leq \epsilon.\]
	It follows that
	\[ \left| K^T_{\boldsymbol{a}} (y)(t) - K^T_{\boldsymbol{a}} (y)(t') \right| \leq 
	\int_0^T{ 
		\left|K^T_{\boldsymbol{a}}(t, s) - K^T_{\boldsymbol{a}}(t', s)\right| \left| y(s) 
		\right| ds } \leq \epsilon 
	||y||_1, \]
	and so the function $K^T_{\boldsymbol{a}} (y)$ is continuous.
	Note that
	\[ \forall s \in [0, T],\quad  K^T_{\boldsymbol{a}}(T, s) = K^T_{\boldsymbol{a}}(0, 
	s),\]
	and so $K^T_{\boldsymbol{a}} (y)$ can be extended to a $T$-periodic 
	function. 
	Altogether, $K^T_{\boldsymbol{a}} (y) \in 
	C^0_T$. To prove that $K^T_{\boldsymbol{a}}$ is a compact operator, we use the  
	Weierstrass 
	approximation Theorem: there exists a sequence of polynomial functions $(t, s) 
	\mapsto P_n(t, s)$ 
	such that $\sup_{t, s \in [0, T]} \left| P_n(t, s) - K^T_{\boldsymbol{a}}(t, s)  \right| 
	\rightarrow_n 0$ 
	as $n \rightarrow \infty$. For each $ n \in \mathbb{N}$, the linear operator 
	$L^1([0, T]) \ni y 
	\mapsto P_n(y) := t \mapsto \int_0^T{P_n(t, s) y(s) ds}$ is of finite-rank. 
	Moreover, the sequence 
	$P_n$ converges to $K^T_{\boldsymbol{a}}$ for the norm operator, and so  
	$K^T_{\boldsymbol{a}}$ 
	is a compact operator (as the limit of finite-rank operators, see \cite[Ch. 
	6]{MR2759829}).
\end{proof}
\begin{lemma}
	\label{lem:space of solution is of dimension 1}
	Let $\boldsymbol{a} \in C^0_T$.
	The Markov Chain $(\phi_i)_{i \geq 0}$ with transition probability kernel 
	$K^T_{\boldsymbol{a}}$ has a unique invariant probability measure 
	$\pi_{\boldsymbol{a}} \in C^0_T$.
	Moreover, the solutions of \eqref{eq:limit equation periodic on [0,T]} in $L^1([0, 
	T])$ 
	span a vector space of dimension $1$.
\end{lemma}
	\begin{proof}
		\textit{Step 1: any solution of \eqref{eq:limit equation periodic on [0,T]} has a 
		constant sign}. Let 
		$x 
		\in L^1([0, T])$ be a solution of \eqref{eq:limit equation periodic on [0,T]}.
		Because the kernel $K^T_{\boldsymbol{a}}$ is strictly positive and continuous 
		on $[0, T]^2$, it 
		holds that
		\begin{equation}
			\label{eq:minoration uniforme KT}
			\delta := \inf_{t,s \in [0, T]} K^T_{\boldsymbol{a}}(t, s) > 0.
		\end{equation} 
		We write $x_+$ for the positive part of $x$, $x_-$ for its negative part and 
		define $\beta := \min( 
		||x_+||_1, ||x_-||_1)$. 
		We have 
		$K^T_{\boldsymbol{a}}(x_+)(t) \geq \delta \beta$ and 
		$K^T_{\boldsymbol{a}}(x_-)(t) 
		\geq \delta \beta$. Consequently
		\begin{align*}
			|| K^T_{\boldsymbol{a}} (x) ||_1 =  || K^T_{\boldsymbol{a}} (x_+) -  
			K^T_{\boldsymbol{a}}(x_-) 
			||_1 \leq  ||K^T_{\boldsymbol{a}} (x_+) -   \delta \beta||_1 + 
			||K^T_{\boldsymbol{a}} (x_-) -   \delta \beta||_1 = ||x||_1 - 2 \delta \beta.
		\end{align*}
		To obtain the right-most equality, we used that for all $y \in L^1([0, T])$, $y 
		\geq 0$ yields 
		$||K^T_{\boldsymbol{a}} y||_1 = ||y||_1$.
		But the identity $ K^T_{\boldsymbol{a}} (x) = x$ implies that $\beta = 0$ and so 
		either 
		$x_+$ or $x_-$ is a.e. null.\\
		\textit{Step 2: existence and uniqueness of the invariant probability measure}. 
		Let 
		$(K^T_{\boldsymbol{a}})^\prime: L^\infty([0, T]) \rightarrow L^\infty([0, T]) $ be 
		the dual operator 
		of $K^T_{\boldsymbol{a}}$. We have:
		\[  \forall v \in L^\infty([0, T]),\quad (K^T_{\boldsymbol{a}})^\prime(v) = s 
		\mapsto  \int_0^T{ K^T_{\boldsymbol{a}}(t, s) v(t) dt}. \]
		From \eqref{eq:identity integral K_T}, we deduce that $1$ is an eigenvalue of 
		$(K^T_{\boldsymbol{a}})^\prime$ (its associated eigenvector is 
		$\boldsymbol{1}$, the constant 
		function equal to $1$).
		Denoting by $N$ the null space, the Fredholm alternative \cite[Th. 
		6.6]{MR2759829} yields 
		$\text{dim}~N (I - 
		K^T_{\boldsymbol{a}}) = 
		\text{dim}~N (I - (K^T_{\boldsymbol{a}})^\prime)$. So there exists 
		$\pi_{\boldsymbol{a}} \in L^1([0, 
		T])$ such that:
		\[ \pi_{\boldsymbol{a}} = K^T_{\boldsymbol{a}} (\pi_{\boldsymbol{a}}) ,~  
		||\pi_{\boldsymbol{a}}||_1 = 1. \]
		By Step 1, $\pi_{\boldsymbol{a}}$ can be chosen positive, and by 
		Lemma~\ref{lem:compact operator}, 
		$\pi_{\boldsymbol{a}} \in C^0_T$. Uniqueness follows directly from Step 1: if 
		$\pi_1, \pi_2$ are two invariant probability measures, then $x = \pi_1 - \pi_2$ 
		solves \eqref{eq:limit 
			equation periodic on [0,T]} and so it has a constant sign. Because the mass 
			of $x$ is null, we 
		deduce 
		that $x = 0$. 
		
		By Steps 1 and 2, we deduce that the solutions of \eqref{eq:limit equation 
		periodic on [0,T]} 
		in $L^1([0, T])$ 
		are the $\{\lambda \pi_{\boldsymbol{a}}, ~\lambda \in \mathbb{R} \}$. 
	\end{proof}
	\begin{remark}
		The estimate \eqref{eq:minoration uniforme KT} is a strong version of the 
		Doeblin’s condition. It 
		holds that
		\[ \inf_{s \in [0, T]} K^T_{\boldsymbol{a}}(\cdot, s) \geq \delta T 
		\text{Unif}(\cdot), \]
		where $\text{Unif}$ is the uniform distribution on $[0, T]$. A classical coupling 
		argument shows 
		that for all $i \geq 1$, $||\mathcal{L}(\phi_i) - \pi_{\boldsymbol{a}}||_{\text{TV}} 
		\leq (1 - \delta T)^i$, 
		where $(\phi_i)$ is the Markov Chain defined by \eqref{eq:proba interpretation 
		Markov} and 
		$||\cdot||_{\text{TV}}$ denotes the total 
		variation 
		distance between probability measures. This argument provides an alternative 
		proof of the 
		existence and uniqueness of 
		$\pi_{\boldsymbol{a}}$.
	\end{remark}
We define for all $\theta \in \mathbb{R}$ the following shift operator
\[
\begin{array}{cccc}
	S_\theta: & C^0_T &\rightarrow & C^0_T  \\
	&  x & \mapsto & (x(t + \theta))_{t}.
\end{array}
\]
\begin{corollary}
	Given $\boldsymbol{a} \in C^0_T$, equations~\eqref{eq:asymptotic periodic jump 
		rate} and
	\eqref{eq:asymptotic periodic jump rate normalization condition} have a unique 
	solution
	$\rho_{\boldsymbol{a}} \in C^0_T$. Moreover, it holds that for all $\theta \in 
	\mathbb{R}$,
	\begin{equation}
		\label{eq:rho under translation}
		\rho_{S_\theta
			(\boldsymbol{a})} = S_\theta (\rho_{\boldsymbol{a}}).
	\end{equation}
\end{corollary}
\begin{proof}
	Note that there is a one to one mapping between 
		periodic solutions of \eqref{eq:asymptotic periodic jump rate} and solutions of 
		\eqref{eq:limit 
			equation periodic on [0,T]}. So by Lemma~\ref{lem:space of solution is of 
			dimension 1}, the 
	solution $\rho_{\boldsymbol{a}}$ of
	equations~\eqref{eq:asymptotic periodic jump rate} and \eqref{eq:asymptotic 
	periodic
		jump rate normalization condition} is $\rho_{\boldsymbol{a}} =  \frac{ 
		\pi_{\boldsymbol{a}} }{ 
		c_{\boldsymbol{a}} }$,
	where $\pi_{\boldsymbol{a}} $ is the invariant measure (on $[0, T]$) of the 
	Markov Chain
	with transition probability kernel $K^T_{\boldsymbol{a}}$ and 
	$c_{\boldsymbol{a}}$ is given by
	\[ c_{\boldsymbol{a}} := \int_{-\infty}^t{ H_{\boldsymbol{a}}(t, s) 
	\pi_{\boldsymbol{a}}(s)
		ds}.  \]
	Note that $c_{\boldsymbol{a}}$ is constant in time. Define for all $t, s \in [0, T]$:
	\[ H^T_{\boldsymbol{a}}(t, s) := \sum_{k \geq 0}{H_{\boldsymbol{a}}(t, s - kT)}. \]
	Using the same notation that in \eqref{eq:notation as a linear 
	operator}, we have 
	$c_{\boldsymbol{a}} 
	= H^T_{\boldsymbol{a}} (\pi_{\boldsymbol{a}})$. Moreover, we have
	\[ \forall t, s, \theta \in \mathbb{R}, \quad \varphi^{S_\theta(\boldsymbol{a})}_{t, 
	s}(0) =
	\varphi^{ \boldsymbol{a}}_{t+\theta,
		s+\theta}(0),  \]
	because both sides satisfy the same ODE with the same initial condition at $t = 
	s$. We
	deduce from \eqref{eq:explicit expression H} and \eqref{eq:explicit expression K} 
	that
	\[ H_{S_\theta (\boldsymbol{a})}(t, s) = H_{\boldsymbol{a}}(t+\theta, 
	s+\theta)\quad \text{ and }
	\quad
	K_{S_\theta (\boldsymbol{a})}(t, s) = K_{\boldsymbol{a}}(t+\theta, s+\theta). \]
	So $S_\theta (\rho_{\boldsymbol{a}})$ solves
	\eqref{eq:asymptotic periodic jump rate} and \eqref{eq:asymptotic periodic jump 
	rate
		normalization condition}, where the kernels are replaced by $K_{S_\theta 
		(\boldsymbol{a})}$
	and $H_{S_\theta (\boldsymbol{a})}$. By uniqueness it follows that $\rho_{S_\theta
		(\boldsymbol{a})} = S_\theta (\rho_{\boldsymbol{a}})$.
\end{proof}
\begin{remark}
	Using that $\int_0^T{ \pi_{\boldsymbol{a}}(s) ds} = 1$, we find that the average 
	number of spikes over one period $[0, T]$ is given by
	\begin{align*}
		\frac{1}{T} \int_0^T{ \rho_{\boldsymbol{a}}(s) ds} &= \frac{1}{c_{\boldsymbol{a}} 
		T}.
	\end{align*}
	The probabilistic interpretation of $c_{\boldsymbol{a}}$ is the following: 
	remembering the Markov chain defined by \eqref{eq:proba interpretation Markov}, 
	we have
	\[ \mathbb{P}( \Delta_{i+1} > k | \phi_i) = H_{\boldsymbol{a}}((k+1)T, \phi_i), \]
	and so, if $\mathcal{L}(\phi_i) = \pi_{\boldsymbol{a}}$, we deduce that
	\[
	\E \Delta_{i+1} = \E \E (\Delta_{i+1} | \phi_i) = \E \left[ \sum_{k \geq 0} 
	\mathbb{P}(\Delta_{i+1} > k | \phi_i) \right] = H^T_{\boldsymbol{a}} 
	(\pi_{\boldsymbol{a}}) = c_{\boldsymbol{a}}.
	\]
	In other words, $c_{\boldsymbol{a}}$ is the expected number of ``revolutions'' 
	between two 
	successive spikes, assuming the phase of each spike follows its 
	invariant 
	measure $\pi_{\boldsymbol{a}}$. We shall see in Proposition~\ref{prop:c alpha +h 
	= c alpha} that 
	$c_{\boldsymbol{a}}$ only depends on the \textit{mean} of $\boldsymbol{a}$. 
	Furthermore, it holds 
	that for $\boldsymbol{a} \equiv \alpha > 0$
	\[ c_\alpha = H^T_\alpha(1/T) = \frac{1}{T} \int_{0}^\infty{ H_\alpha(t) dt} = 
	\frac{1}{T 
		\gamma(\alpha)},  \]
	and so for all $t$, $\rho_\alpha(t) = \gamma(\alpha)$.
\end{remark}
\subsection{Shape of the solutions}
\label{sec:periodic densities}
Let $\boldsymbol{a} \in C^0_T$ such that \eqref{eq:condition on b(0)} holds.
Let $\sigma_{\boldsymbol{a}}(t)$ be defined by \eqref{def:sigma_a for a periodic.}, 
such that $s \mapsto \varphi^{\boldsymbol{a}}_{t, s}(0)$ is a bijection from $(-\infty, 
t]$ to $[0, \sigma_{\boldsymbol{a}}(t))$.
We denote by $x \mapsto \beta^{\boldsymbol{a}}_t(x)$ its inverse.
Note that $t \mapsto \sigma_{\boldsymbol{a}}(t)$ is $T$-periodic and
\[ \forall t \in \mathbb{R}, \forall x \in [0, \sigma_{\boldsymbol{a}}(t)), \quad 
\beta^{\boldsymbol{a}}_{t+T}(x) = \beta^{\boldsymbol{a}}_t(x) + T.\]
Using that $\varphi^{\boldsymbol{a}}_{t, t}(0) = 0$, we have 
$\beta^{\boldsymbol{a}}_t(0) = t$.
\begin{notation}
	Given $\boldsymbol{a} \in C^0_T$, we define for all $t \in \mathbb{R}$
	\begin{equation}
		\label{eq:definition de nu infty a for periodic a}
		\tilde{\nu}_{\boldsymbol{a}} (t, x) :=  \frac{\rho_{\boldsymbol{a}}( 
		\beta^{\boldsymbol{a}}_t(x) )}{b(0) 
			+ a( \beta^{\boldsymbol{a}}_t(x) ) } \EXP{-\int_{\beta^{\boldsymbol{a}}_t(x)}^t{ 
				(f+b')(\varphi^{\boldsymbol{a}}_{\theta, \beta^{\boldsymbol{a}}_t(x)}(0)) 
				d\theta }} \indica{[0, 
			\sigma_{\boldsymbol{a}}(t) )}(x),
	\end{equation}
	where $\rho_{\boldsymbol{a}}$ is the unique solution of the equations 
	\eqref{eq:asymptotic periodic jump rate} and \eqref{eq:asymptotic periodic jump 
	rate normalization condition}.
\end{notation}
By the change of variables $u = \beta^{\boldsymbol{a}}_t(x)$, one obtains that for 
any non-negative 
measurable test function $g$
\begin{equation}
	\int_0^\infty{ g(x) \tilde{\nu}_{\boldsymbol{a}} (t, x) dx } = \int_{-\infty}^t{ 
		g(\varphi^{\boldsymbol{a}}_{t, u}(0)) \rho_{\boldsymbol{a}}(u)  
		H_{\boldsymbol{a}}(t, u) du}.
	\label{eq:link test function flow nu infty a}
\end{equation}
Note moreover that when  $\boldsymbol{a}$ is constant and equal to 
$\alpha > 0$ 
($\boldsymbol{a} \equiv \alpha$), \eqref{eq:definition de nu infty a for periodic a} 
matches with the 
definition of the invariant measure $\nu^\infty_\alpha$ given by \eqref{eq:the 
invariant measure for 
	alpha}:
\[ \forall t \in \mathbb{R},\quad \sigma_\alpha(t) = \sigma_\alpha \quad \text{ and } 
\quad  \tilde{\nu}_\alpha(t) = \nu^\infty_\alpha. \]
The main result of this section is
\begin{proposition}
	\label{prop:the periodic solution of non homogenous current}
	Let $\boldsymbol{a} \in C^0_T$ such that $\inf_{t \in \mathbb{R}} a_t > -b(0)$. It 
	holds that 
	$(\tilde{\nu}_{\boldsymbol{a}}(t, \cdot))_t$ is the unique $T$-periodic solution of 
	\eqref{non-homegeneous-SDE}.
\end{proposition}
\begin{proof} \textbf{Existence}.
	We first prove that $\tilde{\nu}_{\boldsymbol{a}}(t, \cdot)$ is indeed a 
	$T$-periodic solution. We 
	follow the same strategy that \cite[Prop.~26]{CTV}.
	First note that, by \eqref{eq:link test function flow nu infty a}, one has
	\[ \int_0^\infty{ f(x) \tilde{\nu}_{\boldsymbol{a}}(t, x)dx}  = \int_{-\infty}^t{ 
	K_{\boldsymbol{a}}(t, u) 
		\rho_{\boldsymbol{a}}(u) du} = \rho_{\boldsymbol{a}}(t).\]
	Consider the solution $(Y^{\boldsymbol{a}, \tilde{\nu}_{\boldsymbol{a}}(0)}_{t, 
	0})$ of 
	\eqref{non-homegeneous-SDE} starting with law $\tilde{\nu}_{\boldsymbol{a}}(0)$ 
	at time $t = 0$ and let 
	$r^{\tilde{\nu}_{\boldsymbol{a}}(0)}_{\boldsymbol{a}} (t) = \E f(Y^{\boldsymbol{a}, 
		\tilde{\nu}_{\boldsymbol{a}}(0)}_{t, 0})$. \\
	\textbf{Claim}: It holds that for all $t \geq 0$, 
	$r^{\tilde{\nu}_{\boldsymbol{a}}(0)}_{\boldsymbol{a}} (t) = 
	\rho_{\boldsymbol{a}}(t)$.
	\begin{proof}[Proof of the Claim.]
		Recall that $r^{\tilde{\nu}_{\boldsymbol{a}}(0)}_{\boldsymbol{a}} (t)$ is the 
		unique solution of the Volterra equation
		\[ r^{\tilde{\nu}_{\boldsymbol{a}}(0)}_{\boldsymbol{a}} = 
		K^{\tilde{\nu}_{\boldsymbol{a}}(0)}_{\boldsymbol{a}} + K_{\boldsymbol{a}} * 
		r^{\tilde{\nu}_{\boldsymbol{a}}(0)}_{\boldsymbol{a}}. \]
		So, to prove the claim it suffices to show that 
		$\rho_{\boldsymbol{a}}$ also solves this equation.
		For all $u \leq 0 \leq t$, one has
		\[  K^{\varphi^{\boldsymbol{a}}_{0, u}(0)}_{\boldsymbol{a}}(t, 0) 
		H_{\boldsymbol{a}}(0, u) = K_{\boldsymbol{a}}(t, u). \]
		Consequently, we deduce from~\eqref{eq:link test function flow nu infty a} that
		\begin{align*}
			K^{\tilde{\nu}_{\boldsymbol{a}}(0)}_{\boldsymbol{a}}(t, 0) = \int_{-\infty}^{0}{ 
			K_{\boldsymbol{a}}(t, u) \rho_{\boldsymbol{a}}(u) du}.
		\end{align*}
		So
		\[ \rho_{\boldsymbol{a}}(t) = \int_{-\infty}^t{ K_{\boldsymbol{a}}(t, u) 
		\rho_{\boldsymbol{a}}(u) du} = 
		K^{\tilde{\nu}_{\boldsymbol{a}}(0)}_{\boldsymbol{a}}(t, 0)
		+ \int_0^t{ K_{\boldsymbol{a}}(t, u) \rho_{\boldsymbol{a}}(u) du}, \]
		and the conclusion follows.
	\end{proof}
	Finally, using \cite[Prop.~19]{CTV} and the claim, we deduce that for any 
	non-negative measurable 
	function $g$
	\[
	\E g(Y^{\boldsymbol{a}, \tilde{\nu}_{\boldsymbol{a}}(0)}_{t, 0}) = \int_0^t{ 
	g(\varphi^{\boldsymbol{a}}_{t, u}(0)) H_{\boldsymbol{a}}(t, u) 
	\rho_{\boldsymbol{a}} (u) du}  + \int_0^\infty{ g(\varphi^{\boldsymbol{a}}_{t, 0}(x)) 
	H^x_{\boldsymbol{a}}(t, 0) \tilde{\nu}_{\boldsymbol{a}}(0, x) dx}.
	\]
	By \eqref{eq:link test function flow nu infty a} (with $t = 0$ and $g$ replaced by $x \mapsto
		g(\varphi^{\boldsymbol{a}}_{t, 0}(x)) H^x_{\boldsymbol{a}}(t, 0)$),
		 the second 
		term is equal to
	\[ \int_{-\infty}^0{g(\varphi^{\boldsymbol{a}}_{t, u}(0))  H_{\boldsymbol{a}}(t, u) 
		\rho_{\boldsymbol{a}}(u) du}, \]
	and so
	\[ \forall t \geq 0,\quad \E g(Y^{\boldsymbol{a}, \tilde{\nu}_{\boldsymbol{a}}(0)}_{t, 
	0}) = \int_{-\infty}^t{ g(\varphi^{\boldsymbol{a}}_{t, u}(0)) H_{\boldsymbol{a}}(t, u) 
	\rho_{\boldsymbol{a}} (u) du} \overset{\eqref{eq:link test function flow nu infty 
	a}}{=} \int_0^\infty{ g(x) \tilde{\nu}_{\boldsymbol{a}}(t, x) dx}. \]
	This ends the proof of the existence.
	
	\noindent \textbf{Uniqueness}. Consider $(\nu(t))_{t \in [0, T]}$ a $T$-periodic 
	solution of \eqref{non-homegeneous-SDE} and define $\rho(t) = \E 
	f(Y^{\boldsymbol{a}, \nu(0)}_{t, 0})$.
	The function $\rho$ is $T$-periodic. Moreover, it holds that for all $k \geq 0$, 
	$\rho(t) = \E f(Y^{\boldsymbol{a}, \nu(0)}_{t, -kT})$ and so \eqref{eq:the volterra 
	integral equation} and \eqref{eq:the time of the last spike density sum to 1} yields
	\begin{align*}
		\rho(t) &= K_{\boldsymbol{a}}^{\nu(0)}(t, -kT) + \int_{-kT}^t{ 
		K_{\boldsymbol{a}}(t, u) \rho(u) du} \\
		1 &=  H^{\nu(0)}_{\boldsymbol{a}}(t, -kT) + \int_{-kT}^t{ H_{\boldsymbol{a}}(t, 
		u) \rho(u) du}.
	\end{align*}
	Letting $k$ goes to infinity, we deduce that $\rho$ solves \eqref{eq:asymptotic 
	periodic jump rate} and \eqref{eq:asymptotic periodic jump rate normalization 
	condition}. By uniqueness, we deduce that for all $t$, $\rho(t) = 
	\rho_{\boldsymbol{a}}(t)$ (and so $\rho$ is continuous). Finally define $\tau_t$ 
	the time of the last spike of $Y^{\boldsymbol{a}, \nu(0)}_{t, -kT}$ before $t$ (with 
	the convention that $\tau_t = -kT$ if there is no spike between $-kT$ and $t$). 
	The law of $\tau_t$ is
	\[ \mathcal{L}(\tau_t) (du) = \delta_{-kT}(du) H^{\nu(0)}_{\boldsymbol{a}}(t, -kT) + 
	\rho_{\boldsymbol{a}}(u) H_{\boldsymbol{a}}(t, u) du. \]
	Consequently, for any non-negative test function $g$:
	\begin{align*}
		\E g(Y^{\boldsymbol{a}, \nu(0)}_{t, -kT}) &= \E g(Y^{\boldsymbol{a}, \nu(0)}_{t, 
		-kT} \indica{\tau_t = -kT}) +  \E g(\varphi^{\boldsymbol{a}}_{t, \tau_t}(0)) 
		\indica{\tau_t \in (-kT, t] } \\
		&= \int_0^\infty{ g(\varphi^{\boldsymbol{a}}_{t, -kT}(x)) H^x_{\boldsymbol{a}}(t, 
		-kT)\nu(0) (dx) } + 
		\int_{-kT}^t{g(\varphi^{\boldsymbol{a}}_{t, u}(0))  \rho_{\boldsymbol{a}}(u) 
		H_{\boldsymbol{a}}(t, u)   
			du}.
	\end{align*}
	Using that $\E g(Y^{\boldsymbol{a}, \nu(0)}_{t, -kT}) = \E g(Y^{\boldsymbol{a}, 
	\nu(0)}_{t,0})$ and letting again $k$ to infinity we deduce that
	\[ \E g(Y^{\boldsymbol{a}, \nu(0)}_{t, 0}) = 
	\int_{-\infty}^t{g(\varphi^{\boldsymbol{a}}_{t, u}(0)) \rho_{\boldsymbol{a}}(u) 
	H_{\boldsymbol{a}}(t, u) du}. \]
	So for all $t$, $\nu(t) \equiv \tilde{\nu}_{\boldsymbol{a}}(t)$.
\end{proof}

\subsection{Reduction to \texorpdfstring{$2\pi$}{2pi}-periodic functions}
\label{sec:reduction 2pi}
\textbf{Convention}: For now on, we  prefer to work with the \textit{reduced period} 
$\tau$, such that
\[ T =: 2 \pi \tau, \quad \tau > 0. \]
Consider $\boldsymbol{d} \in C^0_{2 \pi \tau}$ and let $\boldsymbol{a}$ be the 
$2\pi$-periodic function defined by:
\[ \forall t \in \mathbb{R}, \quad a(t) := d(\tau t). \]
We define
\[ \forall t \in \mathbb{R}, \quad \rho_{\boldsymbol{a}, \tau} (t) :=
\rho_{\boldsymbol{d}}(\tau t), \]
where $\rho_{\boldsymbol{d}}$ is the unique solution of \eqref{eq:asymptotic 
periodic
	jump rate} and \eqref{eq:asymptotic periodic jump rate normalization condition}
(with kernels \(\K[][\boldsymbol{d}]\) and \(\H[][\boldsymbol{d}]\)).
Because $\rho_{\boldsymbol{d}}$ is $2\pi \tau$-periodic, $\rho_{\boldsymbol{a}, 
\tau}$ is
$2 \pi$-periodic.
Note that when $\boldsymbol{a} \equiv \alpha$ is constant we have
\begin{equation} 
	\label{eq:rho alpha tau is gamma alpha}
	\forall \tau > 0, \forall t \in \mathbb{R}, \quad \rho_{\alpha, \tau}(t) = 
	\gamma(\alpha). 
\end{equation}
To better understand how $\rho_{\boldsymbol{a}, \tau}$ depends on
$\tau$, consider $(Y^{\boldsymbol{d}, \nu}_{t, s})$ the solution of 
\eqref{non-homegeneous-SDE}, 
starting with law $\nu$ and driven by $\boldsymbol{d}$. Note that for all $t \geq s$
\begin{align*}
	Y^{\boldsymbol{d}, \nu}_{\tau t, \tau s}  &= Y^{\boldsymbol{d}, \nu}_{\tau s, \tau s} 
	+ \int_{\tau 
		s}^{\tau t}{ b(Y^{\boldsymbol{d}, \nu}_{u, \tau s}) du } + \int_{\tau s}^{\tau t}{d_u 
		du} - \int_{\tau 
		s}^{\tau t}{ \int_{\mathbb{R}_+}{ Y^{\boldsymbol{d}, \nu}_{u-, \tau s} \indic{ \tau 
		z \leq \tau 
				f(Y^{\boldsymbol{d}, \nu}_{u-, \tau s})} \PM(du, dz)}} \\
	&= Y^{\boldsymbol{d}, \nu}_{\tau s, \tau s} + \int_{s}^{t}{\tau  b(Y^{\boldsymbol{d}, 
	\nu}_{\tau u, \tau 
			s}) du } + \int_{ s}^{t}{\tau a_u du} - \int_{s}^{t}{ \int_{\mathbb{R}_+}{ 
			Y^{\boldsymbol{d}, \nu}_{\tau 
				u-, \tau s} \indic{ z \leq \tau 
				f(Y^{\boldsymbol{d}, \nu}_{\tau u-, \tau s})} \tilde{\PM}(du, dz)}}. 
\end{align*}
Here, $\tilde{\PM} := g_* \PM$ is the push-forward measure of $\PM$ by the function
\[ g(t,z) := (\tau t, z / \tau). \]
Note that $\tilde{\PM}(du, dz)$ is again a Poisson measure of intensity $du dz$, and 
so 
$(Y^{\boldsymbol{d}, \nu}_{\tau t, \tau s})$ is a (weak) solution of 
\eqref{non-homegeneous-SDE} for $\tilde{f} := \tau f$, $\tilde{b} := \tau b$ and 
$\tilde{\boldsymbol{a}} := \tau \boldsymbol{a}$. So, in particular (taking $\nu = 
\delta_0$), if we 
define:
\begin{align}
	\frac{d}{dt} \Flow[\boldsymbol{a}, \tau][t, s](0) &= \tau b(\Flow[\boldsymbol{a}, 
	\tau][t,
	s](0))
	+
	\tau a(t) ; \quad  \Flow[\boldsymbol{a}, \tau][s, s](0) = 0, \nonumber \\
	H_{\boldsymbol{a},\tau}(t, s) &:= \EXP{-\int_s^t{ \tau f(\Flow[\boldsymbol{a}, 
	\tau][u,
			s](0))
			du}},  \nonumber \\
	K_{\boldsymbol{a},\tau}(t, s) &:=  \tau f(\Flow[\boldsymbol{a}, \tau][t, s](0))
	\EXP{-\int_s^t{ \tau f(\Flow[\boldsymbol{a}, \tau][u, s](0)) du}}, \label{eq:K alpha 
	tau}
\end{align}
we have
\begin{lemma}
	\label{lem:link between H et tilde H and so on}
	Let $\tau > 0$ and $\boldsymbol{a} \in C^0_{2 \pi}$. Set, for all $t \in \mathbb{R}$, 
	$d(t) := a( 
	\tfrac{t}{ \tau})$. Then it holds that 
	\[ \forall t \geq s, \quad H_{\boldsymbol{a},\tau}(t, s) = 
	H_{\boldsymbol{d}}(\tau t, \tau s) \quad \text{ and } \quad 
	K_{\boldsymbol{a},\tau}(t, s) = \tau 
	K_{\boldsymbol{d}}(\tau 
	t, \tau s). \]
\end{lemma}
In view of this result, we deduce that $\rho_{\boldsymbol{a}, \tau}$ solves
\begin{equation} \rho_{\boldsymbol{a}, \tau}(t) = \int_{-\infty}^t{ 
K_{\boldsymbol{a},\tau} (t, s)  \rho_{\boldsymbol{a}, \tau}(s) ds}, \quad 1 = \tau 
\int_{-\infty}^t{ H_{\boldsymbol{a},\tau} (t, s)  \rho_{\boldsymbol{a}, \tau}(s) ds},
	\label{eq:r tilde of beta and a}
\end{equation}
or equivalently, setting
\begin{equation} \forall t,s \in [0, 2\pi], \quad K^{2 \pi}_{\boldsymbol{a},\tau} (t, s) := 
\sum_{k \geq
		0}
	K_{\boldsymbol{a},\tau}(t, s- 2 \pi k) \quad \text{ and } \quad   H^{2
		\pi}_{\boldsymbol{a},\tau} (t, s) :=
	\sum_{k \geq 0}  H_{\boldsymbol{a},\tau}(t, s- 2 \pi k),
	\label{eq:definition of H 2pi a tau}
\end{equation}
one has, using the same operator notation as in \eqref{eq:notation as a linear 
operator}
\[  \rho_{\boldsymbol{a}, \tau} = K^{2 \pi}_{\boldsymbol{a},\tau} ( 
\rho_{\boldsymbol{a}, 
	\tau} ) ,\quad 1 = \tau H^{2 \pi}_{\boldsymbol{a},\tau} (\rho_{\boldsymbol{a}, \tau}). 
	\]
Note that $\rho_{\cdot, \tau}$ and $\rho_{\cdot}$ are linked by \eqref{eq:link rho a 
and rho 
	a 
	tau}.
Consequently equations \eqref{eq:r tilde of beta and a} define a unique $2 
\pi$-periodic 
continuous function
\begin{equation}
	\label{eq:link between tilde phi and tilde r}
	\rho_{\boldsymbol{a}, \tau} = \frac{ \pi_{\boldsymbol{a}, \tau}}{ c_{\boldsymbol{a}, 
	\tau}},
\end{equation}
where $\pi_{\boldsymbol{a}, \tau}$ is the unique invariant measure of the Markov 
Chain
with transition probability kernel $K^{2 \pi}_{\boldsymbol{a},\tau}$ and 
$c_{\boldsymbol{a},
	\tau}$ is the constant given by \[ c_{\boldsymbol{a}, \tau} := \tau H^{2
	\pi}_{\boldsymbol{a},\tau} (\pi_{\boldsymbol{a}, \tau}). \]

\subsection{Regularity of \texorpdfstring{$\rho$}{p}}
\label{sec:regularity of rho}
The goal of this section is to study the regularity of $\rho_{\boldsymbol{a}, \tau}$ 
with
respect to $\boldsymbol{a}$ and $\tau$. For $\eta_0 > 0$, recall that $B^{2 
\pi}_{\eta_0}$ is the 
open ball of $C^0_{2 \pi}$ defined by \eqref{eq:open ball of C_T}.
The main result of this section is
\begin{proposition}
	\label{prop:regularity of rho}
	Grant Assumptions~\ref{as:linear drift},
	\ref{as:hyp on f} and let $\alpha_0 > 0$ such that Assumption~\ref{hyp:alpha is 
	non
		degenerate} holds. Let $\tau_0 > 0$. There exists $\epsilon_0,\eta_0 > 0$ small 
		enough
	(only
	depending on
	$b$, $f$, $\alpha_0$ and $\tau_0$) such that the function
	\[
	\begin{array}{rcl}
		B^{2 \pi}_{\eta_0}(\alpha_0) \times (\tau_0 - \epsilon_0, \tau_0 + \epsilon_0) & 
		\rightarrow &
		C^0_{2
			\pi}  \\
		(\boldsymbol{a}, \tau) & \mapsto &  \rho_{\boldsymbol{a}, \tau}
	\end{array}
	\]
	is $\mathcal{C}^2$ Fréchet differentiable.
\end{proposition}
The proof of Proposition~\ref{prop:regularity of rho} relies on \eqref{eq:link between 
tilde phi and 
	tilde r} and on Lemma~\ref{lem:pi is C2} below, which states that the function 
	$(\boldsymbol{a}, 
\tau) \mapsto 
\pi_{\boldsymbol{a}, \tau}$ is $\mathcal{C}^2$. Recall 
Notation~\ref{notation:CT0 and 
		CT00}
\[ C^{0,0}_{2 \pi} := \{  u \in C^0_{2 \pi} | \int_0^{2 \pi}{ u(s) ds } = 0\}. \]
Let $\boldsymbol{a} \in B^{2 \pi}_{\eta_0}$ and $\tau > 0$. Because $\int_0^{2
	\pi}{\pi_{\boldsymbol{a}, \tau}(u) du } = 1$, the space $C^0_{2 \pi}$ can be 
decomposed in
the following way
\[ C^0_{2 \pi} =  \mbox{\normalfont Span}( \pi_{\boldsymbol{a}, \tau}) \oplus 
C^{0,0}_{2 
	\pi}. \]
We denote by $\left. K^{2 \pi}_{\boldsymbol{a},\tau} \right|_{C^0_{2 \pi}}$ the 
restriction of
$K^{2 \pi}_{\boldsymbol{a},\tau}$ to $C^0_{2 \pi}$ (recall that the linear operator $h 
\mapsto
K^{2 \pi}_{\boldsymbol{a},\tau} h$ is defined for all $h \in L^1([0, 2 \pi])$). Similarly, 
we
denote
by $\left. I \right|_{C^0_{2 \pi}}$ the identity operator on $C^0_{2 \pi}$. Given a linear 
operator $L$, we 
denote by $N(L)$ its kernel (null-space) and by $R(L)$ its range.
\begin{lemma}
	\label{lem:kernel and range of I-K}
	Grant Assumptions~\ref{as:linear drift} and \ref{as:hyp on f}, let $\alpha_0 > 0$ 
	such that 
	Assumption~\ref{hyp:alpha is non degenerate} holds and let $\boldsymbol{a} \in 
	B^{2 
		\pi}_{\eta_0}(\alpha_0)$, where $\eta_0 > 0$ is given by 
		Lemma~\ref{lem:estimates on the 
		kernels}.  It holds that
	\[ N (\left. I \right|_{C^0_{2 \pi}} - \left. K^{2 \pi}_{\boldsymbol{a},\tau} 
	\right|_{C^0_{2 
			\pi}}) 
	= \mbox{\normalfont Span}( \pi_{\boldsymbol{a}, \tau}) \quad \text{ and } \quad R 
	(\left. I \right|_{C^0_{2 \pi}} - \left. K^{2 \pi}_{\boldsymbol{a},\tau} \right|_{C^0_{2 
	\pi}}) 
	= 
	C^{0,0}_{2 \pi}.   \]
\end{lemma}
\begin{proof}
	We proved in Lemma~\ref{lem:space of solution is of dimension 1} that $N (I - 
	K^{2 
		\pi}_{\boldsymbol{a},\tau}) = \text{Span}( \pi_{\boldsymbol{a},
		\tau}) $. It remains to show that $R (\left.
	I \right|_{C^0_{2 \pi}} - \left. K^{2 \pi}_{\boldsymbol{a},\tau} \right|_{C^0_{2 \pi}}) = 
	C^{0,0}_{2 \pi}$.
	The Fredholm alternative~\cite[Th. 6.6]{MR2759829} yields
	\[ R (I - K^{2 \pi}_{\boldsymbol{a},\tau}) = N(I  - (K^{2 \pi}_{\boldsymbol{a},\tau} ) 
	^\prime )^{ \bot },   \]
	where $ (K^{2 \pi}_{\boldsymbol{a},\tau} ) ^\prime \in \mathcal{L}\left( L^\infty([0, 2 
	\pi]); 
	L^\infty([0, 2 \pi] ) \right)$ is the dual operator of \newline $K^{2 
	\pi}_{\boldsymbol{a},\tau} \in 
	\mathcal{L}\left( 
	L^1([0, 2 \pi]); L^1([0, 2 \pi] ) \right)$.
	In the proof of Lemma~\ref{lem:space of solution is of dimension 1}, it is shown 
	that
	\[ \boldsymbol{1} \in   N(I - (K^{2 \pi}_{\boldsymbol{a},\tau} ) ^\prime), \]
	where $\boldsymbol{1} $ denotes the constant function equal to $1$ on $[0, 2 
	\pi]$.
	The Fredholm alternative yields
	\[ \text{dim }  N(I - (K^{2 \pi}_{\boldsymbol{a},\tau} ) ^\prime) = \text{dim } N(I - 
	K^{2 \pi}_{\boldsymbol{a},\tau}) = 1.  \]
	So
	\[ N(I - (K^{2 \pi}_{\boldsymbol{a},\tau} ) ^\prime) =\mbox{\normalfont Span} ( 
	\boldsymbol{1}).\]
	It follows that
	\[ R (I  - K^{2 \pi}_{\boldsymbol{a},\tau})  = \mbox{\normalfont Span} ( 
	\boldsymbol{1}) ^{ \bot } = \{ u \in L^1 ([0, 2 \pi]) | \int_0^{2 \pi} { u(s) ds} = 0 \}.  \]
	Finally, using that for $h \in L^1 ([0, 2 \pi])$, one has $K^{2 
	\pi}_{\boldsymbol{a},\tau} h \in
	C^0_{2 \pi}$, one obtains the result for the restrictions to $C^0_{2 \pi}$. 
\end{proof}
As a consequence, the linear operator $I - K^{2 \pi}_{\boldsymbol{a},\tau}: C^{0,0}_{2 
	\pi}
\rightarrow C^{0,0}_{2 \pi}$ is invertible, with a continuous inverse.
\begin{lemma}
	\label{lem: C2 regularity of K and H}
	Grant Assumptions~\ref{as:linear drift},
	\ref{as:hyp on f} and let $\alpha_0 > 0$ such that Assumption~\ref{hyp:alpha is 
	non
		degenerate} holds. Let $\tau_0 > 0$.
	There exists $\eta_0, \epsilon_0 > 0$ small enough (only depending on
	$b$, $f$, $\alpha_0$ and $\tau_0$) such that the following function is
	$\mathcal{C}^2$
	Fréchet differentiable
	\[
	\begin{array}{rcl}
		B^{2 \pi}_{\eta_0}(\alpha_0) \times (\tau_0 - \epsilon_0, \tau_0 + \epsilon_0) & 
		\rightarrow
		&
		\mathcal{L}(C^0_{2 \pi}; C^0_{2 \pi}) \\
		(\boldsymbol{a}, \tau) & \mapsto &  H^{2 \pi}_{\boldsymbol{a}, \tau}.
	\end{array}
	\]
	The same result holds for $K^{2 \pi}_{\boldsymbol{a}, \tau}$.
\end{lemma}
\begin{proof}
	We only prove the result for $H$, the proof for $K$ being similar.
	Let $\epsilon_0 > 0$ be chosen arbitrary such that $\epsilon_0 < \tau_0$. \\
	\textit{Step 1}. We introduce relevant Banach spaces: $E$ denotes the set of 
	continuous functions
	\begin{align*}
		E &:= \mathcal{C}([0, 2\pi]^2; \mathbb{R}), \quad \text{ equipped with } ||w||_E 
		:= \sup_{t, s}
		|w(t,
		s)| \\
		E_0 &:= \{ w \in E,~ \forall s \in [0, 2\pi],~ w(2 \pi, s) =  w(0, s) \} , \quad \text{ 
		equipped with }
		||\cdot ||_E.
	\end{align*}
	We define the following application $\Phi$,
	\[
	\begin{array}{rcl}
		E_0 & \rightarrow
		&
		\mathcal{L}(C^0_{2 \pi}; C^0_{2 \pi}) \\
		w & \mapsto &  \Phi(w) := \left[ h \mapsto  \left( \int_0^{2 \pi}{w(t, s) h(s) ds} 
		\right)_{t \in [0, 2 \pi]}
		\right].
	\end{array}
	\]
	Note that $\Phi$ is linear and continuous, so in particular $\mathcal{C}^2$. So, to 
	prove the result,
	it suffices to show that
	\[
	\begin{array}{rcl}
		B^{2 \pi}_{\eta_0}(\alpha_0) \times (\tau_0 - \epsilon_0, \tau_0 + \epsilon_0) & 
		\rightarrow
		&
		E_0 \\
		(\boldsymbol{a}, \tau) & \mapsto &  (H^{2 \pi}_{\boldsymbol{a}, \tau}(t, s))_{t,s 
		\in [0, 2 \pi]^2}
	\end{array}
	\]
	is $\mathcal{C}^2$, where $H^{2 \pi}_{\boldsymbol{a}, \tau}(t, s)$ is explicitly 
	given by the series
	\eqref{eq:definition of H 2pi a tau}. \\
	\textit{Step 2}. Let $k \in \mathbb{N}$ be fixed. We prove that the function
	\[
	\begin{array}{rcl}
		B^{2 \pi}_{\eta_0}(\alpha_0) \times (\tau_0 - \epsilon_0, \tau_0 + \epsilon_0) & 
		\rightarrow
		&
		E \\
		(\boldsymbol{a}, \tau) & \mapsto &  (H_{\boldsymbol{a}, \tau}(t, s-2 \pi k))_{t,s 
		\in [0, 2 \pi]^2}
	\end{array}
	\]
	is $\mathcal{C}^2$. To proceed, we use the explicit expression  of 
	$H_{\boldsymbol{a}, \tau}(t, s)$,
	given by \eqref{eq:K alpha
		tau}. Note that we have first to show that the function $(\boldsymbol{a}, \tau)
	\mapsto \varphi^{\boldsymbol{a}, \tau}_{t, s}(0) \in \mathbb{R}$ is 
	$\mathcal{C}^2$. This
	follows (see \cite[Th. 3.10.2]{flett_1980}) from the fact that $b: \mathbb{R}_+ 
	\rightarrow \mathbb{R}$ is $\mathcal{C}^2$
	and so the solution of the ODE
	\eqref{eq:K alpha tau} is $\mathcal{C}^2$ with respect to $\boldsymbol{a}$ and 
	$\tau$.
	Moreover, we have for all $h \in C^0_{2 \pi}$,
	\[ D_{\boldsymbol{a}} \varphi^{\boldsymbol{a}, \tau}_{t, s}(0) \cdot h = \int_s^t{\tau 
	h(u)
		\EXP{\tau \int_u^t{b'(\varphi^{\boldsymbol{a}, \tau}_{\theta, s}(0)) d\theta} } 
		du}.   \]
	A similar expression holds for $\frac{d}{d \tau} \varphi^{\boldsymbol{a}, \tau}_{t, 
	s}(0)$.
	Using that $f$ is $\mathcal{C}^2$, we
	deduce that the function
	\[ (\boldsymbol{a}, \tau)
	\mapsto (H_{\boldsymbol{a}, \tau}(t, s- 2 \pi k))_{t,s} \in E \]
	is $\mathcal{C}^2$.
	Furthermore, we have for instance
	\[ D_{\boldsymbol{a}} H_{\boldsymbol{a}, \tau}(t, s) \cdot h = - H_{\boldsymbol{a},
		\tau}(t, s) \int_s^t{\tau f'( \varphi^{\boldsymbol{a}, \tau}_{u, s}(0)) \left[
		D_{\boldsymbol{a}}
		\varphi^{\boldsymbol{a}, \tau}_{u, s} \cdot h \right]  du }.\]
	So, proceeding as in the proof of Lemma~\ref{lem:estimates on the kernels}, we 
	deduce
	the
	existence of $\eta_0, \lambda_0, A_0> 0$ (only depending on  $b$, $f$, 
	$\alpha_0$,
	$\tau_0$ and $\epsilon_0$) such that for all $h \in C^0_{2 \pi}$ and for
	all $\tau \in (\tau_0 - \epsilon_0, \tau_0 + \epsilon_0)$, it holds that
	\[ \sup_{t,s \in [0, 2 \pi]^2} ~~ \sup_{\boldsymbol{a} \in B^{2 
	\pi}_{\eta_0}(\alpha_0)}|
	D_{\boldsymbol{a}}
	H_{\boldsymbol{a}, \tau}(t, s - 2 \pi k) \cdot h | \leq A_0 ||h||_\infty e^{- 2 \pi k 
		\lambda_0}. \]
	Similar estimates hold for the second derivative with respect to $\boldsymbol{a}$ 
	and for
	the first and second derivatives with respect to $\tau$. \\
	\textit{Step 3}.
	We have
	\[ \sum_{k \geq 0}{ \sup_{t,s \in [0, 2 \pi]^2}~~ \sup_{\boldsymbol{a} \in B^{2
				\pi}_{\eta_0}(\alpha_0)}~~ \sup_{h \in
			C^0_{2 \pi}, ||h||_\infty \leq 1}
		|D_{\boldsymbol{a}} H_{\boldsymbol{a},
			\tau}(t, s-2 \pi k) \cdot h| } \leq  \sum_{k \geq 0}{A_0 e^{- 2 \pi k \lambda_0}}  
			<
	\infty.   \]
	Using~\cite[Th. 3.6.1]{MR0223194}, we deduce that $\boldsymbol{a} 
	\mapsto (H^{2 
		\pi}_{\boldsymbol{a}, \tau}(t, s))_{t, s} \in E$ is Fréchet differentiable, with for all 
		$h \in C^0_{2 \pi}$
	\[D_{\boldsymbol{a}} H^{2 \pi}_{\boldsymbol{a}, \tau}(t, s) \cdot h = \sum_{k \geq 
	0}{
		D_{\boldsymbol{a}}  H_{\boldsymbol{a}, \tau}(t, s - 2 \pi k) \cdot h}.   \]
	Note that this last series converges again normally, and so $\boldsymbol{a} 
	\mapsto (H^{2
		\pi}_{\boldsymbol{a}, \tau}(t, s))_{t, s}$ is in fact $\mathcal{C}^1$.
	Applying again \cite[Th. 3.6.1]{MR0223194}, we prove similarly that
	$\boldsymbol{a} \mapsto H^{2 \pi}_{\boldsymbol{a}, \tau}(t, s)$ is 
	$\mathcal{C}^2$. The
	same arguments shows that $\tau \mapsto H^{2 \pi}_{\boldsymbol{a}, \tau}(t, s)$ 
	is
	$\mathcal{C}^2$.  \\
	\textit{Step 4.} It remains to prove that $(\boldsymbol{a}, \tau) \mapsto (H^{2 
	\pi}_{\boldsymbol{a},
		\tau}(t, s))_{t,s} \in E_0$ is $\mathcal{C}^2$ (we have proved the result for $E$, 
		not $E_0$, in the
	previous step). Let $t, s \in [0, 2 \pi]$ be fixed, define
	\[ w \in E, \quad \mathcal{E}^t_s(w) := w(t, s) \in \mathbb{R}. \]
	The application $\mathcal{E}^t_s$ is linear and continuous. Moreover, we have 
	seen that
	$H^{2 \pi}_{\boldsymbol{a}, \tau} \in E_0$, so
	\[ \forall s \in [0, 2 \pi], \quad  \mathcal{E}^{2 \pi}_s(H^{2 \pi}_{\boldsymbol{a}, 
	\tau}) =
	\mathcal{E}^{0}_s(H^{2 \pi}_{\boldsymbol{a}, \tau}).  \]
	Differentiating with respect to $\boldsymbol{a}$, we deduce that for all $h \in 
	C^0_{2 
		\pi}$,
	\[ \forall s \in [0, 2 \pi], \quad  \mathcal{E}^{2 \pi}_s(D_{\boldsymbol{a}} H^{2 
	\pi}_{\boldsymbol{a},
		\tau} \cdot h) =
	\mathcal{E}^{0}_s(D_{\boldsymbol{a}} H^{2 \pi}_{\boldsymbol{a},
		\tau} \cdot h),  \]
	and so $D_{\boldsymbol{a}} H^{2 \pi}_{\boldsymbol{a},
		\tau} \in \mathcal{L}(C^0_{2\pi}, E_0)$. The same results holds for the second 
	derivative with
	respect to $\boldsymbol{a}$ and the two derivatives with respect to $\tau$. This 
	ends the proof.
\end{proof}
\begin{lemma}
	\label{lem:pi is C2}
	Grant Assumptions~\ref{as:linear drift},
	\ref{as:hyp on f} and let $\alpha_0 > 0$ such that Assumption~\ref{hyp:alpha is 
	non
		degenerate} holds. Let $\tau_0 > 0$. There exists $\epsilon_0,\eta_0 > 0$ small 
		enough
	(only
	depending on
	$b$, $f$, $\alpha_0$ and $\tau_0$) such that the function
	\[
	\begin{array}{rcl}
		B^{2 \pi}_{\eta_0}(\alpha_0) \times (\tau_0 - \epsilon_0, \tau_0 + \epsilon_0) &
		\rightarrow &
		C^0_{2
			\pi}  \\
		(\boldsymbol{a}, \tau) & \mapsto &  \pi_{\boldsymbol{a}, \tau}
	\end{array}
	\]
	is $\mathcal{C}^2$ Fréchet differentiable.
\end{lemma}
\begin{remark}
	Recall that $\pi_{\boldsymbol{a}, \tau}$ is the unique invariant measure of the 
	Markov
	Chain having $K^{2 \pi}_{\boldsymbol{a},\tau}$ has kernel transition probability. 
	So, we 
	study the 
	smoothness 
	of the invariant measure with respect to the parameters
	$(\boldsymbol{a}, \tau)$, knowing the smoothness of the transition
	probability kernel $(\boldsymbol{a}, \tau) \mapsto K^{2 \pi}_{\boldsymbol{a},\tau}$.
	We refer to  \cite{MR830647} for such sensibility result in the setting of finite 
	discrete-time Markov Chains. Our approach is different and 
	based on 
	the implicit function theorem. In this proof, we consider independent 
	functions $\boldsymbol{a}$ 
		and $h$ (that is we 
		do not have $\boldsymbol{a} = \alpha_0 + h$). 
\end{remark}
\begin{proof}
	Let $\alpha_0$ and $\tau_0$ be fixed. Let $\delta_0, \epsilon_0 > 0$ be given by
	Lemma~\ref{lem: C2 regularity of K and H}. Consider the following 
	$\mathcal{C}^2$-Fréchet differentiable function:
	\[
	\begin{array}{rrcl}
		F: & C^{0,0}_{2 \pi} \times   B^{2 \pi}_{\eta_0}(\alpha_0) \times (\tau_0 - 
		\epsilon_0,
		\tau_0 +
		\epsilon_0)  &
		\rightarrow &
		C^{0,0}_{2 \pi}  \\
		& (h, \boldsymbol{a}, \tau) & \mapsto &  (\alpha_0 + h) -
		K^{2\pi}_{\boldsymbol{a}, \tau} (\alpha_0 + h)  .
	\end{array}
	\]
	It holds that $F(0,
	\alpha_0, \tau_0) = 0$. Moreover
	\[ D_h F(0, \alpha_0, \tau_0)  = I -  K^{2\pi}_{\alpha_0, \tau_0} \in 
	\mathcal{L}(C^{0,0}_{2 \pi},
	C^{0,0}_{2 \pi}), \]
	which is invertible with continuous inverse by Lemma~\ref{lem:kernel and range of 
	I-K}.
	So the implicit function theorem applies: there exists ($V^{0,0}_{2 \pi}$,
	$V^{0}_{2 \pi}$, $V_{\tau_0})$ open neighborhoods of $(0, \alpha_0,
	\tau_0)$ in $C^{0,0}_{2 \pi} \times C^0_{2 \pi} \times \mathbb{R}^*_+$ and a 
	$\mathcal{C}^2$-Fréchet differentiable function
	$U:  V^{0}_{2 \pi} \times  V_{\tau_0}
	\rightarrow V^{0,0}_{2 \pi}$ such that
	\[ \forall h, \boldsymbol{a}, \tau \in  V^{0,0}_{2 \pi} \times  V^{0}_{2 \pi} \times  
	V_{\tau_0},  \quad F(h,
	\boldsymbol{a}, \tau) = 0 \Longleftrightarrow h = U(\boldsymbol{a}, \tau). \]
	By uniqueness of the invariant measure of the Markov chain with transition kernel 
	$K^{2 \pi}_{\boldsymbol{a}, \tau}$, we deduce that
	\[ \pi_{\boldsymbol{a}, \tau} =  \alpha_0 + U(\boldsymbol{a}, \tau), \]
	which is a $\mathcal{C}^2$-Fréchet differentiable function of $(\boldsymbol{a}, 
	\tau)$.
\end{proof}
\begin{proof}[Proof of Proposition~\ref{prop:regularity of rho}]
	Recall that $\rho_{\boldsymbol{a}, \tau} = \frac{\pi_{\boldsymbol{a}, 
			\tau}}{c_{\boldsymbol{a}, 
			\tau}}$,
	where the constant $c_{\boldsymbol{a}, \tau}$ is given by
	\[ c_{\boldsymbol{a}, \tau} = \tau H^{2 \pi}_{\boldsymbol{a},\tau} 
	(\pi_{\boldsymbol{a}, 
		\tau}). \]
	Furthermore, it holds that $\pi_{\alpha_0, \tau_0}  = \frac{1}{2 \pi}$ and 
	$\rho_{\alpha_0, \tau_0}  = 
	\gamma(\alpha_0)$ (see \eqref{eq:rho alpha tau is gamma alpha}). So 
	$c_{\alpha_0, \tau_0} = \frac{1}{2 
		\pi 
		\gamma(\alpha_0)} > 0$.
	So for $\epsilon_0, \eta_0$ small enough, it holds that 
	\[ \forall \boldsymbol{a} \in B^{2 
		\pi}_{\eta_0}(\alpha_0), \forall \tau \in (\tau_0 - \epsilon_0, \tau_0 + \epsilon_0), 
		\quad
	c_{\boldsymbol{a}, \tau} > 0. \]
	Using Lemmas~\ref{lem: C2 regularity of K and H} and~\ref{lem:pi is C2}, it holds 
	that $c$ and
	$\rho$ are $\mathcal{C}^2$, which ends the proof.
\end{proof}
As a first application of this result, we prove that the mean number of spikes of a 
neuron
driven by a periodic input only depends on the mean of the input current.
\begin{proposition}
	\label{prop:c alpha +h = c alpha}
	Grant Assumptions~\ref{as:linear drift},
	\ref{as:hyp on f} and let $\alpha_0 > 0$ such that Assumption~\ref{hyp:alpha is 
	non
		degenerate} holds. Let $\tau_0 > 0$ and consider $\eta_0$ be given by
	Proposition~\ref{prop:regularity of
		rho}.
	Let $h \in C^{0,0}_{2 \pi}$ such that $\alpha_0 + h \in B^{2 \pi}_{\eta_0}(\alpha_0)$.
	It holds that
	\[ c_{\alpha_0 + h, \tau_0} = c_{\alpha_0, \tau_0} = \frac{1}{2 \pi \gamma(\alpha_0)}. 
	\]
	We denote by $c_{\alpha_0}$ this last quantity. In particular, the mean number of 
	spikes per period:
	\[ \frac{1}{2 \pi} \int_0^{2 \pi}{ \rho_{\alpha_0 + h, \tau_0}(u)du} = 
	\gamma(\alpha_0) \]
	only depends on $\alpha_0$ (which is the mean of the external current \( 
	(\alpha_0 + h(t))_{t \in [0, 2 
		\pi]}\)).
\end{proposition}
\begin{proof}
	Let $\boldsymbol{a} \in B^{2 \pi}_{\eta_0}(\alpha_0)$. We prove that
	\[ \forall h \in C^{0,0}_{2 \pi}, \quad  D_{\boldsymbol{a}} c_{\boldsymbol{a}, \tau_0} 
	\cdot h = 0. \]
	We have $c_{\boldsymbol{a}, \tau_0} = \tau_0 H^{2 \pi}_{\boldsymbol{a},\tau_0} 
	\left(
	\pi_{\boldsymbol{a}, \tau_0} \right)$.  Differentiating with respect to 
	$\boldsymbol{a}$, one gets
	\[
	D_{\boldsymbol{a}} c_{\boldsymbol{a}, \tau_0} \cdot h  = \tau_0 \left 
	[D_{\boldsymbol{a}}
	H^{2
		\pi}_{\boldsymbol{a},\tau_0}
	\cdot h \right] \left( \pi_{\boldsymbol{a}, \tau_0} \right) + \tau_0 H^{2 
	\pi}_{\boldsymbol{a},\tau_0}
	D_{\boldsymbol{a}}
	\pi_{\boldsymbol{a}, \tau_0} \cdot h.
	\]
	Recall that $\pi_{\boldsymbol{a}, \tau_0} = K^{2 \pi}_{\boldsymbol{a},\tau_0} 
	\pi_{\boldsymbol{a},
		\tau_0}$ so
	\[  D_{\boldsymbol{a}} \pi_{\boldsymbol{a}, \tau_0} \cdot h = \left[ 
	D_{\boldsymbol{a}}  K^{2 
		\pi}_{\boldsymbol{a},\tau_0} \cdot h \right] \pi_{\boldsymbol{a}, \tau_0} + K^{2 
		\pi}_{\boldsymbol{a}, 
		\tau_0} \left[ D_{\boldsymbol{a}} \pi_{\boldsymbol{a}, \tau_0} \cdot h \right].  \]
	So, using Lemma~\ref{lem:kernel and range of I-K}, one has
	\begin{equation}
		\label{eq:d pi a}
		D_{\boldsymbol{a}} \pi_{\boldsymbol{a}, \tau_0} \cdot h = \left[ I - K^{2 
			\pi}_{\boldsymbol{a},\tau_0} \right]^{-1} 
		\left [D_{\boldsymbol{a}}  K^{2 \pi}_{\boldsymbol{a},\tau_0} \cdot h \right] 
		\pi_{\boldsymbol{a}, \tau_0}.    
	\end{equation}
	Define on $C^{0,0}_{2 \pi}$ the linear operator
	\[ \forall h \in C^{0,0}_{2 \pi},\quad \mathbbm{1}^{2 \pi} (h)  (t) := \int_0^{2 \pi}{ 
		\indic{t \geq  s} h(s) ds} =  \int_0^t{h(s) ds}. \]
	We have
	\begin{equation}
		\label{eq:1*K a tau}
		1 * K_{\boldsymbol{a},\tau_0} = 1 - H_{\boldsymbol{a},\tau_0}, 
	\end{equation}
	so on $C^{0,0}_{2 \pi}$,
	\begin{equation} H^{2 \pi}_{\boldsymbol{a},\tau_0} =     \mathbbm{1}^{2 \pi} \left[ I 
	-  K ^{2
			\pi}_{\boldsymbol{a},\tau_0} \right].
		\label{eq:link between Ka 2pi et Ha 2pi}
	\end{equation}
	So
	\[  H^{2 \pi}_{\boldsymbol{a},\tau_0} \left[ I -  K ^{2 \pi}_{\boldsymbol{a},\tau_0} 
	\right]^{-1} =
	\mathbbm{1}^{2 \pi}.  \]
	Consequently, we have
	\[
	D_{\boldsymbol{a}} c_{\boldsymbol{a}, \tau_0} \cdot h  = \tau_0 \left 
	[D_{\boldsymbol{a}}
	H^{2 \pi}_{\boldsymbol{a},\tau_0} \cdot h
	\right] \left( \pi_{\boldsymbol{a}, \tau_0} \right) + \tau_0 \mathbbm{1}^{2 \pi} \left
	[D_{\boldsymbol{a}}  K^{2
		\pi}_{\boldsymbol{a},\tau_0} \cdot h \right] \pi_{\boldsymbol{a}, \tau_0}  \]
	Differentiating \eqref{eq:link between Ka 2pi et Ha 2pi}, one has
	\[ D_{\boldsymbol{a}}  H^{2 \pi}_{\boldsymbol{a},\tau_0} \cdot h = - 
	\mathbbm{1}^{2 \pi}
	\left [D_{\boldsymbol{a}}  K^{2
		\pi}_{\boldsymbol{a},\tau_0} \cdot h \right],\]
	and so for all $h \in C^{0,0}_{2 \pi}$,  $D_{\boldsymbol{a}} c_{\boldsymbol{a}, 
	\tau_0} 
	\cdot h =
	0$.
	Then for all $h \in C^{0,0}_{2 \pi}$ such that $\alpha_0 + h \in B^{2 
		\pi}_{\eta_0}(\alpha_0)$,
	one has
	\[ c_{\alpha_0+h, \tau_0} - c_{\alpha_0, \tau_0} = \int_0^1{ \left[ D_{\boldsymbol{a}} 
	c_{\alpha_0 + 
			th,
			\tau_0}
		\cdot h \right] dt} = 0.
	\]
	Finally we have $\pi_{\alpha_0, \tau_0} = \frac{1}{2 \pi}$ and, by \eqref{eq:rho 
	alpha tau is gamma 
		alpha}, $\rho_{\alpha_0, \tau_0} = 
	\gamma(\alpha_0)$. By definition \eqref{eq:link between tilde phi and tilde r}, we 
	have 
	$c_{\alpha_0, 
		\tau_0} = \frac{\pi_{\alpha_0, \tau_0}}{\rho_{\alpha_0, \tau_0}}$. It ends the 
		proof.
\end{proof}

\subsection{Strategy to handle the non-linear equation 
	\texorpdfstring{\eqref{NL-equation}}{}}
\label{sec:fonctionnal G}
Grant Assumptions~\ref{as:linear drift}, \ref{as:hyp on f} and let $\alpha_0 > 0$ such 
that 
Assumption~\ref{hyp:alpha is non degenerate} holds. Let $\tau_0 > 0$ be given by 
Assumption~\ref{ass:i / b0 roots of the characteristic equation}.
For $\eta_0, \epsilon_0 > 0$, define
$G:  B^{2 \pi}_{\eta_0}(\alpha_0) \cap C^{0,0}_{2 \pi} \times (\alpha_0 - \eta_0, 
\alpha_0 + 
\eta_0) \times  (\tau_0 - \epsilon_0, \tau_0 + \epsilon_0) \rightarrow C^{0,0}_{2 \pi}$ 
such that
\begin{equation}
	\label{eq:the function G}
	G(h, \alpha, \tau) :=  (\alpha + h) -  J(\alpha) \rho_{\alpha + h, \tau}.
\end{equation}
Using Propositions~\ref{prop:regularity of rho} and \ref{prop:c 
	alpha +h = c alpha}, we choose $\eta_0, \epsilon_0$ small enough such that $G$ 
	is 
$\mathcal{C}^2$-Fréchet differentiable and indeed takes values in $C^{0,0}_{2 \pi}$.
For any constant $\alpha, \tau > 0$, we have, by \eqref{eq:rho alpha tau is gamma 
alpha},  $ 
\rho_{\alpha, \tau}  = \gamma(\alpha)$. Recalling that $J(\alpha) \gamma(\alpha) = 
\alpha$, 
we have
\begin{equation}
	\label{eq:trivial solutions of G}
	\forall (\alpha, \tau) \in  (\alpha_0 - \eta_0, \alpha_0 + \eta_0) \times (\tau_0 - 
	\epsilon_0, 
	\tau_0 + 
	\epsilon_0),\quad G(0, \alpha, \tau) = 0.
\end{equation}
Those are the trivial roots of $G$. To construct the periodic solutions to 
\eqref{NL-equation}, we
find the non-trivial roots of $G$. In fact, Theorem~\ref{th:main result hopf} is 
deduced
from the following proposition.
\begin{proposition}
	\label{prop:the zeros of G}
	Consider $b, f$ and $\alpha_0, \tau_0 > 0$ such that Assumptions~\ref{as:linear 
	drift},
	\ref{as:hyp
		on f}, \ref{hyp:alpha is non
		degenerate}, \ref{ass:i / b0 roots of the characteristic equation}, 
		\ref{ass:nonresonance
		condition} and \ref{ass:hopf condition} hold. Let $G$ be defined by 
		\eqref{eq:the function G}.
	There exists $X \times V_{\alpha_0} \times V_{\tau_0}$ an open neighborhood of 
	$(0,
	\alpha_0,
	\tau_0)$ in $(C^{0,0}_{2\pi}, ||\cdot||_\infty) \times \mathbb{R}^*_+ \times 
	\mathbb{R}^*_+$  
	such that:
	\begin{enumerate}
		\item
		There exists a continuous curve $\{ (h_v, \alpha_v, \tau_v),~ v \in (-v_0, v_0) \}$ 
		passing through $(0, 
		\alpha_0, \tau_0)$ at 
		$v = 0$ and such that for all $v \in (-v_0, v_0)$
		\[ (h_v, \alpha_v, \tau_v) \in X \times V_{\alpha_0} \times V_{\tau_0} \quad \text{ 
		and }
		\quad G \left( h_v, \alpha_v, \tau_v \right) = 0. \]
		Moreover, it holds that
		\[ \forall v \in (-v_0, v_0), \quad \frac{1}{2 \pi} \int_0^{2 \pi}{h_v(t) \cos(t) dt} = v 
		\quad \text{ and } 
		\quad \frac{1}{2 \pi} \int_0^{2 \pi}{h_v(t) \sin(t) dt} = 0. \]
		In particular, $h_v \not\equiv 0$ for $v \neq 0$.
		\item For all $(h, \alpha, \tau) \in X \times V_{\alpha_0} \times V_{\tau_0}$, with 
		$h 
		\not\equiv 0$, it holds
		that
		\[ G(h, \alpha, \tau) = 0 \iff \left[ \exists v \in (-v_0, v_0), \exists \theta \in [0, 2 
		\pi),
		\quad
		(h, \alpha, \tau) \equiv (S_\theta(h_v), \alpha_v,
		\tau_v) \right].  \]
	\end{enumerate}
\end{proposition}
We here prove that our main result is a consequence  of this proposition.
\begin{proof}[Proof that Proposition~\ref{prop:the zeros of G} implies 
	Theorem~\ref{th:main result hopf}]
	Let $(h_v, \alpha_v, \tau_v)$ be the continuous curve 
	given by Proposition~\ref{prop:the zeros of G}. Define $\boldsymbol{a}_v$ 
	\[ 
	\forall t \in \mathbb{R}, \quad a_v(t) := \alpha_v + h_v( t / \tau_v ).
	\]
	The function $\boldsymbol{a}_v$ is $2 \pi \tau_v$-periodic and continuous. From 
	$G(h_v, \alpha_v,
	\tau_v) = 0$, we deduce that $\boldsymbol{a}_v$ solves 
	\eqref{eq:periodic solution}:
	\[ \boldsymbol{a}_v = J(\alpha_v) \rho_{\boldsymbol{a}_v}. \]
	Consider $\tilde{\nu}_{\boldsymbol{a}_v}$ defined by \eqref{eq:definition de nu 
	infty a for periodic 
		a}. By Proposition~\ref{prop:the periodic solution of non homogenous current}, 
	$(\tilde{\nu}_{\boldsymbol{a}_v}(t))$ is a \(2\pi\tau_v\)-periodic solution of 
	\eqref{NL-equation} and $(\tilde{\nu}_{\boldsymbol{a}_v}, \alpha_v, \tau_v)$
	satisfies all  the
	properties stated in Theorem~\ref{th:main result hopf}: this gives the existence 
	part 
	of the proof. We now prove uniqueness.
	
	Let $\epsilon_0 > 0$ small enough such that $(\tau_0 - \epsilon_0, \tau_0 + 
	\epsilon_0)
	\subset V_{\tau_0}$,   $V_{\tau_0}$ being given by Proposition~\ref{prop:the 
	zeros of G}.
	Let $J, \tau > 0$ be fixed,  consider $\nu(t)$ a $2 \pi  \tau$-periodic
	solution of \eqref{NL-equation} such that
	\[ |\tau-\tau_0| < \epsilon_0  \quad \text{and}
	\quad \sup_{t \in [0, 2 \pi \tau]}  \left| J \int_{\mathbb{R}_+}f(x) \nu(t, dx) - 
	\alpha_0  \right|
	< \epsilon_1, \]
	for some constant $\epsilon_1 > 0$ to be specified later.
	Define $\boldsymbol{a}$ 
	\[ \forall t \in \mathbb{R}, \quad a(t) := J \int_{\mathbb{R}_+}{f(x) \nu(t,dx)}. \]
	The function $\boldsymbol{a}$ is $2 \pi \tau$-periodic. 
	Let $(X_t)_{t \geq 0}$ be the solution of the non-linear equation 
	\eqref{NL-equation}, starting with the
	initial 
	condition $\nu(0) \in \mathcal{M}(f^2)$. The arguments of \cite[Lem. 
	24]{CTV} show that, under Assumptions~\ref{as:linear drift} and \ref{as:hyp on f}, 
	the 
	function $t \mapsto \E
	f(X_t)$ is continuous, and so $\boldsymbol{a} \in C^0_{2 \pi \tau}$.
	We write
	\[ a(t) =: \alpha + h(t / \tau), \]
	for some constant $\alpha$ and some $h \in C^{0,0}_{2 \pi}$. Because $\nu(t)$ 
	is a
	periodic solution of \eqref{NL-equation}, it holds that
	\[ \boldsymbol{a} = J \rho_{\boldsymbol{a}}, \]
	or equivalently,
	\begin{equation}\label{eq:20200721}
		\alpha + h = J \rho_{\alpha + h, \tau}.
	\end{equation}
	We have by assumption
	\begin{align*}
		| \alpha - \alpha_0| &= \left| \frac{1}{2 \pi} \int_0^{2 \pi}{ J \int_{\mathbb{R}_+}{ 
		f(x) \nu(\tau
				u, dx)}
			du} - \frac{1}{2 \pi} \int_0^{2 \pi}{ J(\alpha_0) \int_{\mathbb{R}_+}{ f(x)
				\nu^{\infty}_{\alpha_0}(dx)}
			du} \right|  < \epsilon_1. \\
	\end{align*}
	Recall that $\alpha_0$ satisfies Assumption~\ref{hyp:alpha is non degenerate}.
	By Lemma~\ref{lem: prop of sigma_a near alpha_0} and using the continuity of 
	$b'$, we can assume 
	that $\epsilon_1$ is small enough
	such that Assumption~\ref{hyp:alpha is non degenerate} is also satisfied by 
	$\alpha$.
	Let  $\eta_0$ be given by Proposition~\ref{prop:regularity of rho} ($\eta_0$ only
	depends on $b, f, \alpha_0$ and $\tau_0$).
	Provided that $\epsilon_1 \leq \eta_0$, we can apply
	Proposition~\ref{prop:c alpha +h = c alpha} at $(\alpha, \tau)$. It holds that
	\[ \frac{1}{2 \pi} \int_0^{2 \pi}{ \rho_{\alpha + h, \tau}(u) du} = \gamma(\alpha), \]
	so \eqref{eq:20200721} gives
	\[ \alpha = J \gamma(\alpha). \]
	This proves that $J = J(\alpha)$ and so \eqref{eq:20200721} yields $G(h, \alpha, \tau) = 0$.
	By the uniqueness part of Proposition~\ref{prop:the zeros of G}, there exists 
	$\theta
	\in [0, 2 \pi)$ and $v \in (-v_0, v_0)$ such that
	\[ \forall t, \quad h(t) = h_v(t + \theta),  \quad \alpha = \alpha_v, \quad \tau = 
	\tau_v. \]
	So, we deduce that $a(t) = \alpha_v + h_v\left(\frac{t + \theta}{\tau_v} \right)$ and 
	$J = J(\alpha_v)$. 
	This
	ends the proof.
\end{proof}
It remains to prove Proposition~\ref{prop:the zeros of G}.
\subsection{Linearization of  \texorpdfstring{$G$}{G}.}
\label{sec:linearization of G}
Define:
\begin{equation}
	\label{eq:definition de Theta_(alpha,tau)}
	\forall t \in \mathbb{R}, \quad \Theta_{\alpha, \tau} (t) := \tau \Theta_\alpha(\tau t) 
	\indica{\mathbb{R}_+}(t),
\end{equation}
where $\Theta_\alpha$ is given by~\eqref{eq:formule donnant Theta.}. The main 
result of
this section is the following.
\begin{proposition}
	\label{prop:the differential of G at x = 0}
	Let $h \in C^{0,0}_{2 \pi}$. It holds that
	\[ \left[ D_h G(0,\alpha, \tau) \cdot h \right] (t)=  h(t) -   J(\alpha) 
	\int_{\mathbb{R}}{ \Theta_{\alpha, \tau}(t-s) 
		h(s) ds } . \]
\end{proposition}
\noindent The proof of this proposition relies on Lemmas~\ref{lem:inverse of linear 
	operator I - K2pialpha} and \ref{lem:link between DaKalphatau et Xi} below. Let $h 
	\in 
C^{0,0}_{2 \pi}$. The definition  of $G$ yields
\[  D_h G(0,\alpha, \tau) \cdot h = h - J(\alpha) D_{\boldsymbol{a}} \rho_{\alpha, \tau}  
\cdot h. \]
By equation \eqref{eq:link between tilde phi and tilde r}
and Proposition~\ref{prop:c alpha +h = c alpha}, one has
\[
D_{\boldsymbol{a}} \rho_{\alpha, \tau}  \cdot h = \frac{1}{c_\alpha}  D_{\boldsymbol{a}}
\pi_{\alpha, \tau}  \cdot h.
\]
Recall that $\pi_\alpha$ is the uniform law on $[0, 2\pi]$. To compute 
$D_{\boldsymbol{a}} 
\pi_{\alpha, \tau}  \cdot h$, we use \eqref{eq:d pi a} with 
$\boldsymbol{a} \equiv \alpha$:
\begin{equation}
	\label{eq:DaPi}
	D_{\boldsymbol{a}} \pi_{\alpha, \tau}  \cdot h = (I - K^{2 \pi}_{\alpha, \tau})^{-1} 
	\left[
	D_{\boldsymbol{a}} K^{2 \pi}_{\alpha,
		\tau} \cdot h \right]( \tfrac{1}{ 2 \pi}). 
\end{equation}
The next lemma is devoted to the computation of $(I - K^{2 \pi}_{\alpha, \tau})^{-1}$.
Consider $t \mapsto r_\alpha(t)$ the solution of the convolution Volterra integral 
equation
\eqref{eq:the volterra integral equation} (with $\nu = \delta_0$ and $\boldsymbol{a} =
\alpha$). That is, $r_\alpha$ solves $r_\alpha = K_\alpha + K_\alpha * r_\alpha$.
By \cite[Prop. 37]{CTV}, there exists a function $\xi_\alpha \in 
L^1(\mathbb{R}_+)$
such that for all $t \geq 0$, 
\[ r_\alpha(t) = \gamma(\alpha) + \xi_\alpha(t). \] 
Define for all $t \geq 0$, $r_{\alpha, \tau} (t) := \tau r_\alpha( \tau t)$. It solves
\begin{equation}
	\label{eq:volterra for r alpha tau}
	r_{\alpha, \tau} = K_{\alpha, \tau} + K_{\alpha, \tau} *  r_{\alpha, \tau},
\end{equation}
where $K_{\alpha, \tau}$ is given by \eqref{eq:K alpha tau}. Similarly, let $\xi_{\alpha, 
\tau}(t) := 
\tau \xi_{\alpha}(\tau t).$ We have
\[ r_{\alpha, \tau} (t) = \tau \gamma(\alpha) + \xi_{\alpha, \tau}(t). \]
Recall that by definition, we have
\[ K^{2 \pi}_{\alpha, \tau}(h) (t) = \int_0^{2\pi}{K^{2 \pi}_{\alpha, \tau}(t, s) h(s) ds} =  
\int_{-\infty}^t{ 
	K_{\alpha, \tau}(t-s) h(s) ds}. \]
and
	\[ 
	H^{2 \pi}_{\alpha, \tau}(h) (t) = \int_0^{2\pi}{H^{2 \pi}_{\alpha, \tau}(t, s) h(s) ds} 
	=  \int_{-\infty}^t{ 
		H_{\alpha, \tau}(t-s) h(s) ds}. 
	\]
\begin{lemma}
	\label{lem:inverse of linear operator I - K2pialpha}
	The inverse of the linear operator 
	$I - K^{2 \pi}_{\alpha,\tau}: C^{0,0}_{2 \pi} \rightarrow C^{0,0}_{2 \pi}$ is 
	given by $I + r^{2 \pi}_{\alpha, \tau}$ where for all $h \in C^{0,0}_{2 
		\pi}$ and $t \in [0, 2\pi]$
	\begin{align*}
		r^{2 \pi}_{\alpha, \tau} (h) &:= \tau \gamma(\alpha) \Gamma(h) +  \xi^{2 
			\pi}_{\alpha, \tau}(h), \\
		\Gamma(h)(t) &:= \int_0^t{ h(s) ds} - \frac{1}{2 
			\pi} 
		\int_0^{2 \pi}{ \int_0^s{h(u)du} ds}, \\
		\xi^{2 \pi}_{\alpha, \tau}(h)(t)  &:= \int_{- \infty}^t{\xi_{\alpha, \tau}(t 
			- s) h(s) ds }.
	\end{align*}
\end{lemma}
\begin{proof}
	Note that $\Gamma(h)$ is the only primitive of $h$ which belongs to $C^{0,0}_{2 
		\pi}$. Moreover,  
	because $t \mapsto 
	\xi_{\alpha, \tau}(t) \in L^1(\mathbb{R}_+)$, we have for $h \in C^{0,0}_{2 \pi}$:
	\[
	\int_0^{2\pi}\int_{- \infty}^t{\xi_{\alpha, \tau}(t 
		- s) h(s) ds }dt = \int_0^{2\pi}\int_{0}^{\infty}{\xi_{\alpha, \tau}(u) h(t-u) du }dt = 
	\int_{0}^{\infty}\int_0^{2\pi}{\xi_{\alpha, \tau}(u) h(t-u) dt }du = 0.
	\]
	So $\xi^{2 \pi}_{\alpha, \tau}(h) \in C^{0,0}_{2 \pi}$. Altogether, $r^{2 
	\pi}_{\alpha, \tau} \in 
		C^{0,0}_{2 
			\pi}$. 
	To conclude, we have to show that on $C^{0,0}_{2 \pi}$
	\[  K^{2 \pi}_{\alpha,\tau} \circ r^{2 \pi}_{\alpha, \tau}   = r^{2 \pi}_{\alpha, \tau} \circ 
	K^{2 
		\pi}_{\alpha,\tau} = 
	r^{2 \pi}_{\alpha, \tau} - K^{2 \pi}_{\alpha,\tau}.  \]
	Note that for all $t \in [0, 2 \pi]$,
	\[ \frac{d}{dt} \left[ \Gamma(h)(t) - H^{2 \pi}_{\alpha, \tau}(h)(t)\right] = K^{2 
	\pi}_{\alpha, \tau}(h)(t). 
	\]
	Because $\Gamma(h), H^{2 \pi}_{\alpha, \tau}(h) \in C^{0,0}_{2 \pi}$, we deduce 
	that
	\[ \Gamma(K^{2 \pi}_{\alpha, \tau}(h)) = \Gamma(h) - H^{2 \pi}_{\alpha, \tau}. \]
	Moreover, we have (using that $\xi_{\alpha, \tau}, K_{\alpha, \tau} \in 
	L^1(\mathbb{R}_+)$)
	\begin{align*}
		\xi^{2 \pi}_{\alpha, \tau}(K^{2 \pi}_{\alpha,\tau}  (h))(t) &= \int_{-\infty}^t{ 
		\xi_{\alpha, \tau}(t -
			s)  \int_{-\infty}^s{ K_{\alpha, \tau}(s - u) h(u) du ds }} \\
		&= \int_{-\infty}^t{ h(u) 
			\int_u^t{ \xi_{\alpha, \tau} (t-s) K_{\alpha, \tau}
				(s-u) ds} du	} \\
		&= \int_{-\infty}^t{ h(u) (\xi_{\alpha, \tau} * K_{\alpha, \tau})(t - u)  du.	}
	\end{align*}
	Using \eqref{eq:1*K a tau} and \eqref{eq:volterra for r alpha tau}, we deduce the 
	identity
	\begin{equation}
		\label{eq:K*xi}
		K_{\alpha, \tau} * \xi_{\alpha, \tau} = \xi_{\alpha, \tau} *  K_{\alpha, \tau} = 
		\xi_{\alpha, 
			\tau} -  K_{\alpha, \tau} + \tau 
		\gamma(\alpha ) 
		H_{\alpha, \tau} .
	\end{equation}
	So
	\[  \xi^{2 \pi}_{\alpha, \tau}(K^{2 \pi}_{\alpha,\tau}  (h)) = \xi^{2 \pi}_{\alpha, \tau} (h) 
	- K^{2 
		\pi}_{\alpha,\tau} (h) + \tau \gamma(\alpha) H^{2
		\pi}_{\alpha,\tau} (h). \]
	Altogether, 
	\[ r^{2 \pi}_{\alpha, \tau} (K^{2 \pi}_{\alpha,\tau}(h)) = r^{2 \pi}_{\alpha, \tau} (h) - 
	K^{2 
		\pi}_{\alpha,\tau} (h). \]
	We now prove that $K^{2 \pi}_{\alpha,\tau} ( r^{2 \pi}_{\alpha, \tau}(h)) = r^{2 
	\pi}_{\alpha, 
		\tau}(h) 
	- K^{2 
		\pi}_{\alpha,\tau}(h)$. Using \eqref{eq:K*xi}, we have $K^{2 \pi}_{\alpha, 
		\tau}(\xi^{2 
		\pi}_{\alpha,\tau}  (h)) = \xi^{2 \pi}_{\alpha, \tau}(K^{2 \pi}_{\alpha,\tau}  (h))  $. 
	Moreover, because $K^{2 \pi}_{\alpha, \tau}(1) = 1$, we have
	\begin{align*}
		K^{2 \pi}_{\alpha, \tau}(\Gamma(h)) (t) &= \int_{-\infty}^t{ K_{\alpha, \tau}(t-s) 
		\int_0^s{ 
				h(u) du} ds} 
		- \frac{1}{2 
			\pi} 
		\int_0^{2 \pi}{ \int_0^s{h(u) du} ds} \\
		&= \left[ H_{\alpha, \tau}(t-s) \int_0^s{ h(u) du} \right]^t_{-\infty} - 
		\int^t_{-\infty}{  
			H_{\alpha, 
				\tau}(t-s) h(s)ds} - 
		\frac{1}{2 
			\pi} 
		\int_0^{2 \pi}{ \int_0^s{h(u) du} ds} \\
		&= \Gamma(h)(t) - H^{2 \pi}_{\alpha, \tau}(h)(t) = \Gamma(K^{2 \pi}_{\alpha, 
		\tau}(h)) (t).
	\end{align*}
	It ends the proof.
\end{proof}
	Let $(\alpha, \tau) \in (\alpha_0 - \eta_0, \alpha_0 + \eta_0) \times (\tau_0 - 
	\epsilon_0, \tau_0 + \epsilon_0)$ and $h \in C^{0,0}_{2 \pi}$. 
	Using \eqref{eq:link between tilde phi and tilde r} and Proposition~\ref{prop:c 
	alpha +h = 
		c alpha}, it holds
	that
\begin{equation}
	\label{eq:first step to compute the differentiate of rho}
	D_{\boldsymbol{a}} \rho_{\alpha, \tau}  \cdot h = \frac{1}{c_{\alpha, \tau}}  
	D_{\boldsymbol{a}}
	\pi_{\alpha, \tau}  \cdot h  \overset{\eqref{eq:DaPi}}{=} (I + r^{2 \pi}_{\alpha, \tau}) 
	\left[ D_{\boldsymbol{a}} 
	K^{2
		\pi}_{\alpha,\tau} \cdot h \right] (\gamma(\alpha)).
\end{equation}
Consider $\Xi_\alpha(t)$ be defined by \eqref{eq:definition de Xi alpha} 
and define for all 
$t \geq 0$, 
$\Xi_{\alpha, \tau}(t) := \tau \Xi_\alpha(\tau t)$. We also denote by $\Xi^{2 
	\pi}_{\alpha, \tau}$ the linear operator
\[ \forall h \in C^0_{2 \pi}, \forall t \in [0, 2\pi], \quad \Xi^{2 \pi}_{\alpha, \tau}(h)(t) := 
\int_{-\infty}^t{ 
	\Xi_{\alpha, \tau}(t-u)h(u) 
	du }. \]
\begin{lemma}
	\label{lem:link between DaKalphatau et Xi}
	For all  $h \in C^{0}_{2 \pi}$ we have $\left[ D_{\boldsymbol{a}} K^{2
		\pi}_{\alpha,\tau} \cdot h \right] (\gamma(\alpha)) = \Xi^{2 \pi}_{\alpha, \tau}(h)$.
\end{lemma}
\begin{proof}
	Given $h \in C^{0}_{2 \pi}$, we have
	\begin{align*}
		\left[ D_{\boldsymbol{a}} K^{2 \pi}_{\alpha,\tau} \cdot h  \right] (\gamma(\alpha)) 
		(t) = 
		\gamma(\alpha)  \int_{-\infty}^t{ \left[ D_{\boldsymbol{a}} K_{\alpha, \tau} \cdot  h
			\right] (t, s) ds }
	\end{align*}
	So we have to prove that
	\begin{equation} \forall h \in C^{0}_{2 \pi},\quad \gamma(\alpha)  \int_{-\infty}^t{ 
	\left[ 
			D_{\boldsymbol{a}}
			K_{\alpha, \tau}
			\cdot  h \right] (t, s) ds } = \int_{-\infty}^t{ \Xi_{\alpha, \tau}( t-s)  h(s) ds}. 
		\label{eq:link DaKtau Xiatau}
	\end{equation}
When $\tau = 1$,  we know by Lemma~\ref{lem:link between H et tilde 
	H 
			and so on} that
		\(K_{\alpha, 1} = K_{\alpha}\), \(H_{\alpha, 1} = H_{\alpha}\), etc..
		In  \cite{cormier2020meanfield}, eq. (72) 	
		gives
		\[ \gamma(\alpha) \int_{-\infty}^t{ \left[ D_{\boldsymbol{a}} H_{\alpha}
			\cdot  h \right] (t, s) ds }  = -\int_{\mathbb{R}}{ \Psi_{\alpha}( t-s)  h(s) ds}, \]
		where $\Psi_\alpha(t)$ is given by \eqref{eq:une formule pour Psi_alpha}. 
		Using that $\Psi_\alpha(0) = 0$, $\Xi_\alpha(t) \overset{\eqref{eq:definition de 
		Xi alpha}}{=} 
		\frac{d}{dt} \Psi_\alpha(t)$ and
		\[   \int_{-\infty}^t{ \left[ D_{\boldsymbol{a}} K_{\alpha}
			\cdot  h \right] (t, s) ds } = -\frac{d}{dt}  \int_{-\infty}^t{ \left[ 
			D_{\boldsymbol{a}}
			H_{\alpha}
			\cdot  h \right] (t, s) ds }, \]
		we deduce \eqref{eq:link DaKtau Xiatau} with $\tau = 1$. The result for $\tau 
		\neq 1$ can be deduced 
	from 
	the case $\tau = 1$. Indeed, given $\alpha
	>0$ and $h \in
	C^{0}_{2 \pi}$, define $\tilde{f} := \tau f$, $\tilde{b} := \tau b$, $\tilde{\alpha} := 
	\tau \alpha$,
	and $\tilde{h} := \tau h$. By applying the result for $\tilde{\tau} := 1$, $\tilde{b}$,
	$\tilde{f}$, $\tilde{\alpha}$ and $\tilde{h}$, we obtain exactly the stated equality.
\end{proof}
\begin{proof}[Proof of Proposition~\ref{prop:the differential of G at x = 0}]
	We use Lemma~\ref{lem:link between DaKalphatau et Xi} together with 
	\eqref{eq:first step
		to compute the differentiate of rho}.
	For all $h \in C^{0,0}_{2 \pi}$, one obtains
	\[ D_{\boldsymbol{a}} \rho_{\alpha, \tau}  \cdot h = \Xi^{2 \pi}_{\alpha, \tau} (h) + 
	r^{2 \pi}_{\alpha,
		\tau}(\Xi^{2 \pi}_{\alpha,
		\tau} (h)).  \]
	The definition of $r^{2 \pi}_{\alpha, \tau}$ yields
	\[ r^{2 \pi}_{\alpha, \tau}(\Xi^{2 \pi}_{\alpha, \tau} (h)) = \tau \gamma(\alpha) 
	\Gamma(\Xi^{2 
		\pi}_{\alpha, \tau} (h)) 
	+  \xi^{2 \pi}_{\alpha, \tau} (\Xi^{2 \pi}_{\alpha, \tau} (h)). \]
	Let $\Psi_{\alpha, \tau}(t) := \Psi_\alpha(\tau t)$, such that $\frac{d}{dt} 
	\Psi_{\alpha, \tau}(t) = 
	\Xi_{\alpha, \tau}(t)$. From the identity
	\[ \frac{d}{dt} \int_{-\infty}^t{ \Psi_{\alpha, \tau}(t - u) h(u) du} = \int_{-\infty}^t{ 
	\Xi_{\alpha, \tau}(t - 
		u) h(u) du}, \]
	we find that
	\[ \Gamma(\Xi^{2 \pi}_{\alpha, \tau}( h) ) (t) = \int_{-\infty}^t{ \Psi_{\alpha, \tau}(t-u) 
	h(u) du} = 
	\int_{-\infty}^t{ (1 * \Xi_{\alpha, \tau})(t-u) h(u) du }. \]
	So
	\begin{align*}
		\left[ D_{\boldsymbol{a}} \rho_{\alpha, \tau}  \cdot  h \right](t) &= 
		\int_{-\infty}^t{ \Xi_{\alpha, \tau}(t-u) 
			h(u) du} + \tau \gamma(\alpha) \int_{-\infty}^t{ (1 * \Xi_{\alpha, \tau})(t-u) 
			h(u) du} \\
		& \quad + \int_{-\infty}^t{ 
			\xi_{\alpha, \tau}(t-u) \int_{-\infty}^u{ \Xi_{\alpha, \tau}(u-\theta) h(\theta) 
			d\theta} du}. 
	\end{align*}
	Fubini's Theorem yields
	\[ \int_{-\infty}^t{ 
		\xi_{\alpha, \tau}(t-u) \int_{-\infty}^u{ \Xi_{\alpha, \tau}(u-\theta) h(\theta) 
		d\theta} du}  = 
	\int_{-\infty}^t{ 
		(\xi_{\alpha, \tau} * \Xi_{\alpha, \tau})(t-\theta) h(\theta) d\theta }. \]
	Finally, we have
	\begin{align*}
		\Xi_{\alpha, \tau} + \tau \gamma(\alpha) (1 * \Xi_{\alpha, \tau}) + \xi_{\alpha, \tau} 
		* \Xi_{\alpha, \tau} &= 
		\Xi_{\alpha, \tau} + r_{\alpha, \tau} * \Xi_{\alpha, \tau} \quad \text{ (because 
		$r_{\alpha, \tau} = \tau 
			\gamma(\alpha) + \xi_{\alpha, \tau}$)  }\\
		& \overset{\eqref{eq:link Theta Xi and r}}{=} \Theta_{\alpha, \tau},
	\end{align*}
	so
	\[ \left[ D_{\boldsymbol{a}} \rho_{\alpha, \tau}  \cdot  h \right](t) = \int_{-\infty}^t{ 
	\Theta_{\alpha, \tau} 
		(t-u) h(u) du}. \]
	It ends the proof.
\end{proof}

\subsection{The linearization of  \texorpdfstring{$G$}{G} at  \texorpdfstring{$(0, 
		\alpha_0, \tau_0)$}{(0, a0,t0)} is a Fredholm operator}
\label{sec:fredholm operator}
For notational convenience we now write
\[ 
B_0 := D_h G(0,\alpha_0, \tau_0). 
\]
\begin{proposition}
	\label{prop:ker and im of B0}
	We have  $N(B_0) = R(Q),~ R(B_0) = N(Q)$,  where $Q$ is the following projector 
	on 
	$C^{0,0}_{2 \pi}$:
	\begin{equation}
		\label{eq:definition projector Q}
		\forall z \in C^{0,0}_{2 \pi},\quad Q(z) (t) := \left[ \frac{1}{2 \pi} \int_0^{2 \pi} { 
		z(s) 
			e^{-i s}
			ds} \right] e^{i t} + \left[ \frac{1}{2 \pi} \int_0^{2 \pi} { z(s) e^{i s} ds} \right] 
			e^{-i t}.
	\end{equation}
\end{proposition}
\begin{remark}
	In particular, $B_0 \in \mathcal{L}(C^{0,0}_{2 \pi}, C^{0,0}_{2 \pi})$  is a Fredholm 
	operator of index 0, with $\text{dim } N(B_0) = 2$.
\end{remark}
\begin{proof}
	First, let $h \in N(B_0)$. One has for all $t \in \mathbb{R}$
	\[ h(t) = J(\alpha_0) \int_{\mathbb{R}}{ \Theta_{\alpha_0, \tau_0}(t - s) h(s) ds}.  \]
	Consider for all $n \in \mathbb{Z}$
	\[ \tilde{h}_n := \frac{1}{2 \pi} \int_0^{2 \pi}{ h(s) e^{- in s} ds} \]
	the $n$-th Fourier coefficient of $h$. We have
	\[ \forall  n \in \mathbb{Z},  \quad \tilde{h}_n = J(\alpha_0) 
	\widehat{\Theta}_{\alpha_0, \tau_0}(i n) \tilde{h}_n. \]
	Assumption \ref{ass:nonresonance condition} ensures that
	\[ \forall n \in \mathbb{Z}  \backslash \{ -1, 1\}, \quad J(\alpha_0) 
	\widehat{\Theta}_{\alpha_0, \tau_0}(i n)  \neq 1,\]
	and so
	\[  \forall n \in \mathbb{Z}  \backslash \{ -1, 1\}, \quad  \tilde{h}_n = 0. \]
	We deduce that $h \in R(Q)$. Conversely, if $h \in R(Q)$, there exists $c \in 
	\mathbb{C}$ such that
	\[ h(t) = c e^{it } + \bar{c} e^{-it} \]
	and so
	\begin{align*}
		J(\alpha_0)  \int_{\mathbb{R}}{ \Theta_{\alpha_0, \tau_0}(t - s) h(s) ds} & = c 
		e^{it}  J(\alpha_0) \int_{\mathbb{R}}{ \Theta_{\alpha_0, \tau_0}(s) e^{-is} ds} + 
		\bar{c} e^{-it } J(\alpha_0)  \int_{\mathbb{R}}{ \Theta_{\alpha_0, \tau_0}(s) e^{is} 
		ds} \\
		& = c e^{it}   J(\alpha_0) \widehat{\Theta}_{\alpha_0, \tau_0}(i) + \bar{c} e^{-it }  
		J(\alpha_0) \widehat{\Theta}_{\alpha_0, \tau_0}(-i) \\
		& = h(t).
	\end{align*}
	We used here that $J(\alpha_0) \widehat{\Theta}_{\alpha_0, \tau_0}(i)  = 
	J(\alpha_0)
	\widehat{\Theta}_{\alpha_0, \tau_0}(-i) = 1$ (Assumption ~\ref{ass:i / b0 roots of 
	the
		characteristic equation}).
	This proves that $N(B_0) = R(Q)$.  Consider now $k \in R(B_0)$, there exists 
	$h \in 
		C^{0,0}_{2 \pi}$ such that $B_0(h) = k$. We have for all $t \in \mathbb{R}$
	\[ 
	h(t) -  J(\alpha_0) \int_{\mathbb{R}}{ \Theta_{\alpha_0, \tau_0}(t - s) h(s)}ds = k(t). 
	\]
	Using that $J(\alpha_0) \widehat{\Theta}_{\alpha_0, \tau_0}(i) = 1$, we  deduce 
	that
	\[ \frac{1}{2 \pi} \int_{0}^{2 \pi}{ k(s) e^{-is } ds} = \left[ \frac{1}{2 \pi} \int_{0}^{2 
	\pi}{ h(s) e^{-is } ds} \right](1 - J(\alpha_0) \widehat{\Theta}_{\alpha_0, \tau_0}(i) ) 
	= 0. \]
	Similarly, $\frac{1}{2 \pi} \int_{0}^{2 \pi}{ k(s) e^{is } ds} = 0$ and so $k \in N(Q)$. It
	remains to show that $N(Q) \subset R(B_0)$. Consider $h \in N(Q)$ and let \[ 
	\tilde{h}_n
	:= \frac{1}{2 \pi} \int_0^{2 \pi}{ h(s) e^{- in s} ds} \] 
be its $n$-th Fourier coefficient. We
have $\tilde{h}_1 = \tilde{h}_{-1} = 0$.  
The function $h$ is continuous, and so $h$ belongs to $L^2([0, 2 \pi])$. We deduce that
\[ \sum_{n \in \mathbb{Z} \backslash \{-1, 1 \} } { |\tilde{h}_n|^2 } < \infty. \]
Define
\[ \forall n \in \mathbb{Z} \backslash \{-1, 1\},\quad \epsilon_n := \frac{J(\alpha_0) \widehat{\Theta}_{\alpha_0, 
\tau_0}(i n)}{1- J(\alpha_0) \widehat{\Theta}_{\alpha_0, \tau_0}(i n)}. \]
 Using \cite[Lem. 33, 34]{cormier2020meanfield}, the function $t \mapsto \Psi_{\alpha_0}(t)$, explicitly 
	given by 
	\eqref{eq:une formule pour Psi_alpha}, is 
	$\mathcal{C}^1$
	and its derivative $\Xi_{\alpha_0}(t) = \frac{d}{dt} \Psi_{\alpha_0}(t)$ belongs to $L^1(\mathbb{R}_+)$. The same 
	holds 
	true 
	for $t \mapsto \Xi_{\alpha_0}(t)$. So, using \eqref{eq:08-20-21} and \eqref{eq:definition de Theta_(alpha,tau)}, 
	we 
	deduce that $t \mapsto \Theta_{\alpha_0, 
		\tau_0}(t)$ is $\mathcal{C}^1$ and its derivative belongs to $L^1(\mathbb{R}_+)$. This gives the 
existence of a constant $C$ such that for $n \in
\mathbb{Z}$,
	\[ |n| > 1 \implies |\epsilon_n| \leq \frac{C}{|n|}. \]
	We deduce that
	\[  \sum_{n \in \mathbb{Z} \backslash \{-1, 1 \} } { |n \epsilon_n \tilde{h}_n|^2 } < 
	\infty.\]
	Consequently, defining
	\[ \forall t \in \mathbb{R},\quad w(t) := \sum_{n \in \mathbb{Z} \backslash \{-1, 1 \} 
	}{ \epsilon_n \tilde{h}_n e^{i n t}}, \]
	it holds that $w \in H^1([0, 2 \pi])$, and so $w$ is continuous (see for instance 
	\cite[Th.
	8.2]{MR2759829}).
	Finally, let $k := h + w$. It holds that $k \in 
		C^{0,0}_{2 \pi}$ and the $n$-th Fourier 
	coefficient of $k$ is equals to $ \frac{\tilde{h}_n}{1-J(\alpha_0) 
		\widehat{\Theta}_{\alpha_0, \tau_0}(i n)} $. We deduce that $B_0(k) = h$. This 
		ends 
	the proof.
\end{proof}

\subsection{The Lyapunov-Schmidt reduction method}
\label{sec:lyapunov}

The problem of finding the roots of $G$ defined by \eqref{eq:the function G}  is
an infinite dimensional problem. We use the method of Lyapunov-Schmidt to
obtain an equivalent problem of finite-dimension - here of dimension 2. The 
equation $G
= 0$ is equivalent to
\begin{align*}
	Q G (Qh + (I-Q) h, \alpha, \tau) &= 0 \\
	(I-Q) G (Qh + (I-Q) h, \alpha, \tau) &= 0 
\end{align*}
where the projector $Q$ is defined by \eqref{eq:definition projector Q}.
Define the following function  $W$:
\[
\begin{array}{rrcl}
	W:  & U_2 \times W_2 \times V_{\alpha_0} \times  V_{\tau_0} &\rightarrow & 
	R(B_0)  \\
	& (v, w, \alpha, \tau) & \mapsto &  (I-Q)G(v+w, \alpha, \tau),
\end{array}
\]
where $U_2 \times W_2$ are open neighborhood of $(0, 0)$ in $N(B_0) \times R 
(B_0).$

We have $W(0, 0, \alpha_0, \tau_0) = 0$ and $D_w W(0, 0, \alpha_0, \tau_0) = (I-Q) 
D_h G(0, 
\alpha_0, \tau_0) = (I-Q) B_0 \in \mathcal{L}(R(B_0), R(B_0)) $ which is bijective with 
continuous 
inverse.
The implicit function theorem applies: there exists a $\mathcal{C}^1$ function $\psi:
N(B_0) \times V_{\alpha_0} \times V_{\tau_0} \mapsto R(B_0)$ such that
\begin{align*}
	&W(v, w, \alpha, \tau) = 0 \text{ for } (v, w, \alpha, \tau) \in U_2 \times W_2 \times 
	V_{\alpha_0} \times V_{\tau_0} \text{ is equivalent to } \\
	& w = \psi(v, \alpha, \tau).
\end{align*}
Again, the neighborhoods $U_2, W_2, V_{\tau_0}, V_{\alpha_0}$ may be shrunk in 
this
construction. We deduce that
\begin{align}
	\label{eq:lyapounov schmidt}
	&G(h, \alpha, \tau) = 0 \text{ for } (h, \alpha, \tau) \in X \times V_{\alpha_0} \times 
	V_{\tau_0} \text{ 
		is equivalent to } \\
	& Q G( Qh + \psi(Qh, \alpha, \tau), \alpha, \tau) = 0. \label{eq:reduced LP}
\end{align}
	Indeed, if $G(h, \alpha, \tau) = 0$, we have in particular $W(Qh, (I-Q)h, \alpha, 
	\tau) 
	= 0$ and so $(I-Q)h = \psi(Qh, \alpha, \tau)$: this gives \eqref{eq:reduced LP}. 
	Reciprocally, if 
	\eqref{eq:reduced LP} holds, we set $h = Qh + \psi(Qh, \alpha, \tau)$ and obtain 
	\eqref{eq:lyapounov schmidt}.
Note that for all $\theta \in \mathbb{R}$, we have for all $\tau > 0$ and 
$\boldsymbol{a} \in 
C^0_{2
	\pi},~\rho_{S_\theta(\boldsymbol{a}), \tau} = S_\theta (\rho_{\boldsymbol{a}, 
	\tau})$. It follows by 
definition of $G$ that
\[ G(S_\theta(h), \alpha, \tau) = S_\theta (G(h, \alpha, \tau)). \]
Moreover, it is clear that the projection $Q$ commutes with $S_\theta$ (for all 
$\theta \in \mathbb{R},~S_\theta 
Q = Q S_\theta$) and by the local uniqueness of the implicit  function theorem, we 
deduce that
\[ \psi(S_\theta(v), \alpha, \tau) = S_\theta (\psi(v, \alpha, \tau)). \]
Using that any element $Q h \in N(B_0)$ can be written
\[ Qh = t \mapsto c e^{it} + \bar{c} e^{-it} := c e_0 + \bar{c} \bar{e}_0 \]
for some $c \in \mathbb{C}$ and using the definition of $Q$, we deduce that 
\eqref{eq:lyapounov schmidt} is equivalent to the complex equation:
\[  \hat{\Phi} (c, \alpha, \tau) = 0 \text{ for } (c,\alpha, \tau) \in  V_0 \times 
V_{\alpha_0} \times  V_{\tau_0} \text{, where } \]
\[ \hat{\Phi} (c, \alpha, \tau) :=  \frac{1}{2 \pi}  \int_0^{2 \pi} { G(  c e_0 + \bar{c} 
\bar{e}_0 + \psi( c e_0 + \bar{c} \bar{e}_0, \alpha, \tau), \alpha, \tau)_t  } e^{-it }dt \]
and  $V_0$ is an open neighborhood of $0$ in $\mathbb{C}$.
We have moreover
\[ \forall \theta \in \mathbb{R},\quad \hat{\Phi} (c e^{i \theta}, \alpha, \tau) = e^{i 
\theta} \hat{\Phi} (c, \alpha, \tau), \]
and so \eqref{eq:lyapounov schmidt}  is equivalent to
\[  \hat{\Phi} (v, \alpha, \tau) = 0 \text{ for } v \in (-v_0, v_0).  \]
Note that $ \hat{\Phi} (-v, \alpha, \tau) =  -\hat{\Phi} (v, \alpha, \tau)$ and in 
particular
\[ \forall \alpha, \tau \in  V_{\alpha_0} \times V_{\tau_0},\quad   \hat{\Phi} (0, \alpha, 
\tau) = 0. \]
This is coherent with \eqref{eq:trivial solutions of G}. In order to eliminate these 
trivial
solutions, following \cite{MR2859263}, we set for $v \in (-v_0, v_0) \setminus \{ 0 \}$:
\begin{align*}
	\tilde{\Phi} (v, \alpha, \tau) :=& \frac{\hat{\Phi} (v, \alpha, \tau)}{v} \\
	= & \int_0^1{ D_v \hat{\Phi} (\theta v, \alpha, \tau) d \theta}.
\end{align*}
To summarize, we have proved that
\begin{lemma}
	\label{lem:the reduction to 2D}
	There exists $v_0 > 0$ and  open neighborhoods $X \times V_{\alpha_0} \times
	V_{\tau_0}$ of $(0, \alpha_0, \tau_0)$  in $C^{0,0}_{2 \pi} \times \mathbb{R}^*_+ 
	\times
	\mathbb{R}^*_+$ such that the problem
	\[  G(h, \alpha, \tau) = 0 \text{ for } (h, \alpha, \tau) \in X \times V_{\alpha_0} \times
	V_{\tau_0} \text{ with } h \neq 0 \]
	is equivalent to
	\[  \tilde{\Phi} (v, \alpha, \tau)  = 0 \text{ for } (v, \alpha, \tau) \in (-v_0, v_0) \times 
	V_{\alpha_0} \times V_{\tau_0}. \]
\end{lemma}
The next section is devoted to the study of this reduced problem.

\subsection{Study of the reduced 2D-problem}
\label{sec:2D}
We denote by $\cos$ the cosinus function, such that $v e_0 + v \bar{e_0} = 2 v 
\cos.$
\begin{lemma}
	We have:
	\begin{enumerate}
		\item $\tilde{\Phi} (0, \alpha_0, \tau_0) = 0$.
		\item $D_\tau \tilde{\Phi} (0, \alpha_0, \tau_0) = \frac{1}{2 \pi} \int_0^{2 \pi}{  
		\left[ D^2_{h
				\tau} G(0, \alpha_0, \tau_0) \cdot 2\cos \right]_t e^{-it } dt}. $
		\item $D_\alpha \tilde{\Phi} (0, \alpha_0, \tau_0) = \frac{1}{2 \pi} \int_0^{2 \pi}{  
		\left[ D^2_{h
				\alpha} G(0, \alpha_0, \tau_0) \cdot 2\cos \right]_t e^{-it } dt}. $
	\end{enumerate}
\end{lemma}
\begin{proof}
	We have $\tilde{\Phi}(0, \alpha_0, \tau_0) = D_v \hat{\Phi}(0, \alpha_0, \tau_0)$ and
	\[ D_v \hat{\Phi}(0, \alpha_0, \tau_0) = \frac{1}{2 \pi} \int_0^{2 \pi}{ D_h G(0, 
	\alpha_0,
		\tau_0) \cdot \left[ 2 \cos + D_v \psi(0, \alpha_0, \tau_0) \cdot 2 \cos \right]_t 
		e^{-i t} dt }. \]
	Moreover, it holds that (see \cite[Coroll. 1.2.4]{MR2859263})
	\[ D_v \psi(0, \alpha_0, \tau_0) \cdot \cos = 0 \]
	and $\cos \in N(D_h G(0, \alpha_0, \tau_0))$, so $\tilde{\Phi}(0, \alpha_0, \tau_0) = 
	0$. To
	prove the second point (the third point is proved similarly), we have $D_\tau 
	\tilde{\Phi} (0,
	\alpha_0, \tau_0) = D^2_{v \tau} \hat{\Phi}(0, \alpha_0, \tau_0)$. Moreover,
	\begin{align*}
		D_\tau \hat{\Phi}(v, \alpha, \tau) = &\frac{1}{2 \pi} \int_0^{2 \pi}{ D_\tau G(2 v 
		\cos + \psi(2r
			\cos, \alpha, \tau), \alpha, \tau)_t e^{-i t} dt} \\
		& \quad\quad +  \frac{1}{2 \pi} \int_0^{2 \pi}{ \left[ D_h G(2 r\cos + \psi(2 v \cos, 
		\alpha,
			\tau), \alpha, \tau) \cdot  D_\tau \psi(2 v \cos, \alpha, \tau) \right]_t e^{-it} dt}. 
			\\
	\end{align*}
	So
	\begin{align*}
		D^2_{v \tau} \hat{\Phi}(0, \alpha_0, \tau_0) &= \frac{1}{2 \pi} \int_0^{2 \pi}{ \left[ 
		D^2_{h
				\tau}  G(0, \alpha_0, \tau_0)  \cdot \left(2 \cos + D_v \psi(0, \alpha_0, 
				\tau_0) \cdot 2 \cos
			\right) \right]_t  e^{-i t} dt}  \\
		& \quad \quad + \frac{1}{2 \pi} \int_0^{2 \pi}{ \left[D_h G(0, \alpha_0, \tau_0) 
		\cdot D^2_{v
				\tau } \psi(0, \alpha_0, \tau_0) \cdot 2 \cos \right]_t e^{-it} dt }  \\
		& \quad \quad + \frac{1}{2 \pi} \int_0^{2 \pi}{ D^2_{hh} G(0, \alpha_0, \tau_0) 
		\cdot [2 \cos
			+ D_v \psi(0, \alpha_0, \tau_0) \cdot 2 \cos, D_\tau \psi(0, \alpha_0, 
			\tau_0)]_t  e^{-it}dt}.
	\end{align*}
	Note that for all $\alpha, \tau$ in the neighborhood of $\alpha_0, \tau_0$, one has
	\[ \psi(0, \alpha, \tau) = 0,  \]
	so $D_\tau \psi(0,  \alpha_0, \tau_0) = 0$. Consequently the third term 
	is null.
	Recall now that $B_0 := D_h G(0, \alpha_0, \tau_0)$ and by 
	Proposition~\ref{prop:ker and im of 
		B0}, it holds that $Q B_0 = 0$. So the second term is also null.
	Finally, using again that $D_v \psi(0, \alpha_0, \tau_0) \cdot \cos  = 0$ we obtain 
	the
	stated formula.
\end{proof}
By Proposition~\ref{prop:the differential of G at x = 0}, we have for all $h \in 
C^{0,0}_{2 \pi}$
\[ D_h G(0, \alpha, \tau) \cdot h = h - J(\alpha) \Theta_{\alpha, \tau} * h, \]
where the function $\Theta_{\alpha, \tau }$ is given by equation \eqref{eq:definition 
de
	Theta_(alpha,tau)}. It follows that
\[ D^2_{h \tau} G(0, \alpha_0, \tau_0) \cdot 2\cos  = - 2 J(\alpha_0) \frac{\partial 
}{\partial
	\tau } \left. \left(  \Theta_{\alpha_0, \tau} * \cos  \right) \right|_{\tau = \tau_0} ,    \]
and so we have
\[ D_\tau \tilde{\Phi} (0, \alpha_0, \tau_0) = - J(\alpha_0)  \frac{\partial }{\partial \tau } 
\left. \widehat{\Theta}_{\alpha_0, \tau} (i) \right|_{\tau = \tau_0} . \]
Similarly,
\[ D_\alpha \tilde{\Phi} (0, \alpha_0, \tau_0) = -  \frac{\partial }{\partial \alpha }  \left. 
\left(
J(\alpha)  \widehat{\Theta}_{\alpha, \tau_0}(i) \right) \right|_{\alpha = \alpha_0}. \]
\begin{lemma}
	\label{lem:link between Theta and Zalpha}
	Write $J(\alpha_0) \frac{\partial}{\partial z} \widehat{\Theta}_{\alpha_0}( \tfrac{i}{ 
	\tau_0} 
	)  =: x_0 + i y_0$. It holds that
	\begin{enumerate}
		\item $  D_\tau \tilde{\Phi} (0, \alpha_0, \tau_0)= (i x_0 -y_0) / \tau^2_0.$
		\item $ D_\alpha \tilde{\Phi} (0, \alpha_0, \tau_0) = \mathfrak{Z}_0'(\alpha_0)(x_0 
		+ i y_0), $
		where $\mathfrak{Z}_0'(\alpha_0)$ is defined in Lemma~\ref{lem:definition of
			mu(alpha)}.
	\end{enumerate}
\end{lemma}
\begin{proof}
	From $\Theta_{\alpha, \tau}(t) = \tau \Theta_\alpha(\tau t)$, we have
	\[  \frac{\partial}{\partial \tau } \Theta_{\alpha, \tau} (t) = \tfrac{1}{\tau} \left[ \tau
	\Theta_\alpha(\tau t) + \tau \Pi_\alpha(\tau t) \right],\quad \text{ with } \quad 
	\Pi_\alpha(t) :=
	t \frac{\partial}{\partial t} \Theta_\alpha (t).   \]
	So
	\[ \widehat{ \left[  \frac{\partial }{\partial \tau } \Theta_{\alpha, \tau} \right] } (z) = 
	\frac{1}{\tau}\left[  \widehat{\Theta}_\alpha\left( \frac{z}{\tau }\right) + 
	\widehat{\Pi}_\alpha\left( 
	\frac{z}{\tau }\right)  \right].  \]
	Moreover, an integration by parts shows that
	\begin{align*}
		\widehat{\Pi}_\alpha(z) &= \int_{0}^\infty{ e^{-zt }  t \frac{\partial}{\partial t} 
		\Theta_\alpha (t) dt} \\
		&= -\widehat{\Theta}_\alpha(z) + z \int_0^\infty{ e^{-zt} t \Theta_\alpha(t) dt}. \\
		&=    -\widehat{\Theta}_\alpha(z)  - z \frac{ \partial}{\partial z} 
		\widehat{\Theta}_\alpha(z).
	\end{align*}
	Choosing $z = i$ ends  the proof of the first point. Define now
	\[ \Delta(z, \alpha) := J(\alpha) \widehat{\Theta}_\alpha (z) -1. \]
	By the definition of $\mathfrak{Z}_0(\alpha)$ (see Lemma~\ref{lem:definition of
		mu(alpha)}), we have
	\[ \forall \alpha \in V_{\alpha_0},\quad  \Delta(\mathfrak{Z}_0(\alpha), \alpha) = 0. \]
	We differentiate with respect to $\alpha$ and obtain
	\[ \frac{\partial}{\partial z} \Delta(\mathfrak{Z}_0(\alpha),\alpha) 
	\mathfrak{Z}_0'(\alpha) +
	\frac{\partial}{\partial \alpha} \Delta(\mathfrak{Z}_0(\alpha),\alpha) = 0.\]
	Evaluating this expression at $\alpha = \alpha_0$ gives
	\[ \frac{\partial }{\partial \alpha }   \left. \left(   J(\alpha) \widehat{\Theta}_{\alpha} 
	\right)
	\right|_{\alpha = \alpha_0}  (\tfrac{i}{\tau_0}) = -\mathfrak{Z}_0'(\alpha_0) (x_0 + i 
	y_0), \]
	which concludes the proof.
\end{proof}
\begin{lemma}
	There exists $v_0 > 0$, an open neighborhood  $V_{\alpha_0} \times V_{\tau_0}$   
	of
	$(\alpha_0, \tau_0)$ in $(\mathbb{R}^*_+)^2$ and two functions $v \mapsto \tau_v,
	\alpha_v \in \mathcal{C}^1((-v_0, v_0))$ such that for all $(v, \alpha, \tau) \in (-v_0, 
	v_0)
	\times V_{\alpha_0} \times V_{\tau_0}$ we have
	\[ \tilde{\Phi}(v, \alpha, \tau) = 0 \iff \tau = \tau_v \text{ and } \alpha = \alpha_v. \]
\end{lemma}
\begin{proof}
	We decompose $\tilde{\Phi}$ into real part and imaginary part (without changing 
	the notations), such that now
	\[ \tilde{\Phi}: (-v_0, v_0) \times V_{\alpha_0} \times V_{\tau_0} \rightarrow 
	\mathbb{R}^2.  \]
	We have $\tilde{\Phi}(0, \alpha_0, \tau_0) = 0$ and
	\begin{align*}
		D_{(\alpha, \tau)} \tilde{\Phi}(0, \alpha_0, \tau_0) =  &
		\begin{pmatrix}
			\Re  D_\alpha \tilde{\Phi} (0, \alpha_0, \tau_0) & \quad \Re  D_\tau \tilde{\Phi} 
			(0,
			\alpha_0, \tau_0)  \\
			\Im  D_\alpha \tilde{\Phi} (0, \alpha_0, \tau_0) & \quad  \Im  D_\tau \tilde{\Phi} 
			(0,
			\alpha_0, \tau_0)  \\
		\end{pmatrix}  \\
		=  & \begin{pmatrix}
			x_0  \Re  \mathfrak{Z}_0'(\alpha_0) - y_0 \Im \mathfrak{Z}_0'(\alpha_0) & 
			\quad -
			\frac{y_0}{\tau^2_0} \\
			x_0 \Im \mathfrak{Z}_0'(\alpha_0) + y_0 \Re \mathfrak{Z}_0'(\alpha_0) &  \quad
			\frac{x_0}{\tau^2_0} \\
		\end{pmatrix}.
	\end{align*}

	The determinant of this matrix is $\frac{\Re \mathfrak{Z}_0'(\alpha_0)}{\tau_0^2}( 
	x^2_0 +
	y_0^2)$ and this quantity is non-null by Assumptions~\ref{ass:i / b0 roots of the
		characteristic equation} and \ref{ass:hopf condition}.
	Consequently, the implicit function theorem applies and gives the result.
\end{proof}
The proof of Proposition~\ref{prop:the zeros of G} then follows immediately from 
this
result and  Lemma~\ref{lem:the reduction to 2D}. This ends the proof of
Theorem~\ref{th:main result hopf}.
\section{An explicit example}
\label{sec:toy model}
We now give a simple example of functions $f$ and $b$ such that Hopf
bifurcations
occurs and such that the spectral assumptions of Theorem~\ref{th:main result hopf} 
can
be analytically verified. 
First, by \cite[Th. 21]{cormier2020meanfield}, if
	\begin{equation} \forall x \geq 0, \quad f(x)+ b'(x) \geq 0, 
		\label{eq:sufficient condition stab}
	\end{equation}
	then any invariant probability measure of \eqref{NL-equation} is locally stable. So, 
	to have Hopf 
	bifurcations, the 
	drift $b$ has to be sufficiently attractive to break \eqref{eq:sufficient condition 
	stab}.
Our minimal example satisfies all the assumptions of
Theorem~\ref{th:main result hopf}, except Assumption~\ref{as:hyp on f}, because 
the
function $f$ we consider is not continuous. Indeed, to simplify the computation, we
consider the step function
\[
\forall  x \in \mathbb{R}_+, \quad f(x) := %
\begin{cases}
	0 & \text{for } 0 \leq x < 1,\\
	1 / \beta & \text{for } x \geq 1,
\end{cases}
\]
where $\beta >0$ is a (small) parameter of the model.
\subsection{Some generalities when  \texorpdfstring{$f$}{f} is a step function}
We shall specify later the exact
shape of $b$, for now we only assume that
\[ \inf_{x \in [0, 1]} b(x) > 0.  \]
This ensures in particular that the Dirac mass 
at \(0\)
is not an invariant measure. 
We
now consider some fixed constant $\alpha \geq 0$.
Let, for all $x \in [0, 1]$
\[ t^*_\alpha(x) := \inf\{ t \geq 0,~ \varphi^\alpha_t(x) = 1 \}, \]
the time required for the deterministic flow to hit $1$, starting from $x$.  A simple
computation shows that
\[ t^*_\alpha(x) = \int_x^1{\frac{dy}{b(y) + \alpha}}.\]
Let $H^x_\alpha(t)$ be defined by \eqref{def:jump rate, survival, density of
	survival} (with
$\nu = \delta_x, ~\boldsymbol{a} \equiv \alpha$ and $s = 0$). Using the
explicit shape of
$f$, we find for
all $x \in [0, 1]$,
\[
H^x_\alpha(t) := 
\begin{cases}
	1 & \text{for } 0 \leq t < t^*_\alpha(x), \\
	e^{- \frac{t - t^*_\alpha(x)}{\beta }} & \text{for } t \geq  t^*_\alpha(x). %
\end{cases}
\]
Moreover, 
\begin{equation} 
	\label{eq:04112020}
	\forall x > 1, \quad H^x_\alpha(t) = e^{- t / \beta}. 
\end{equation}
Altogether,
\[ \forall z \in \mathbb{C} \text{ with } \Re(z) > -1/\beta \quad \widehat{H}_\alpha(z) = 
\frac{1-e^{-z 
		t^*_\alpha(0)}}{z} + \frac{e^{-z t^*_\alpha(0)}}{z+1/\beta}.   \]
Note that in particular (using that $1/\gamma(\alpha) = \widehat{H}_\alpha(0)$)
\[ 1 / \gamma(\alpha) =  t^*_\alpha(0) + \beta.\]
So
\begin{equation}
	\label{eq:formula for J alpha toy}
	J(\alpha) := \frac{\alpha}{\gamma(\alpha)} = \int_0^1{ \frac{dy}{1+
			b(y)/\alpha}} + \alpha
	\beta 
\end{equation}
is a strictly increasing function of $\alpha$: 
for a fixed value of $J > 0$,
there is a unique  \(\alpha > 0\) solution  of $\alpha = J\gamma(\alpha)$
and the corresponding $\nu^\infty_\alpha$  is the unique invariant measure of 
\eqref{NL-equation}.
Let $\sigma_\alpha = \lim_{t \rightarrow \infty} \varphi^\alpha_t(0)$. 
This invariant
measure is given by
\[
\nu^\infty_{\alpha}(x) =  
\begin{cases}
	\frac{\gamma(\alpha)}{b(x) + \alpha} & \text{for } x \in [0, 1), \\
	\frac{\gamma(\alpha)}{b(x) + \alpha} \EXP{-\frac{1}{\beta} \int_1^x{
			\frac{dy}{b(y) + \alpha} 
	} } & \text{for } x \in [1, \sigma_{\alpha})  \\
	0 & \text{ otherwise. }
\end{cases}
\]
Moreover, for $x \in [0, 1]$ and $t > t^*_\alpha(x)$,
\[ \frac{d}{dx} H^x_\alpha(t) = - \frac{1}{\beta} \frac{ e^{-\frac{t - 
t^*_\alpha(x)}{\beta}}}{b(x) + \alpha}. \]
So the Laplace transform of $\frac{d}{dx}H^x_\alpha(t)$ is, for all $z 
\in \mathbb{C}$
with
$\Re(z) > -1/\beta$
\[ 
\forall x \in [0, 1],  \quad \int_0^\infty{ e^{-zt } \frac{d}{dx} H^x_\alpha(t) dt} = -\frac{
	e^{-t^*_\alpha(x) z}}{b(x) + \alpha} \frac{1}{1 + \beta z}. 
\]
Let $\Psi_\alpha$ be defined by \eqref{eq:expression de Psi_alpha}. For 
all $z \in \mathbb{C}$ with 
$\Re(z) > -\beta$, one has
\begin{align*}
	J(\alpha) \widehat{\Psi}_\alpha(z) &= - \frac{\alpha}{\gamma(\alpha)} 
	\int_0^{\sigma_\alpha}{ \int_0^\infty{  e^{-zt } \frac{d}{dx} H^x_\alpha(t) dt }~ 
	\nu^\infty_\alpha(x) dx} \\
	&= \frac{\alpha}{1 + \beta z} \int_0^1{ \frac{e^{-t^*_\alpha(x) z}}{(b(x) + \alpha)^2}  
		dx}.
\end{align*}
Indeed, using \eqref{eq:04112020}, it holds that $\frac{d}{dx} 
H^x_\alpha(t) = 0$ for $x > 
	1$. 
Finally, the change of variable
\[ x= \varphi^\alpha_u(0), \quad u \in [0, t^*_\alpha(0)),  \]
such that $t^*_\alpha(x) = t^*_\alpha(0) - u$, shows that
\[  J(\alpha) \widehat{\Psi}_\alpha(z) =  \frac{\alpha e^{-z t^*_\alpha(0)} }{ 1 + \beta z}
\int_0^{t^*_\alpha(0)}{ \frac{e^{u z}}{b(\varphi^\alpha_u(0)) + \alpha}  du}.  \]
Using \cite[Remark~35]{cormier2020meanfield}, the (local) stability of 
the invariant 
measure 
$\nu^\infty_\alpha$ is given by the
location of the
roots of the following holomorphic function, defined for all $\Re(z) > -1/\beta$:
\[   \boxed{ J(\alpha) \widehat{\Psi}_\alpha(z) - \widehat{H}_\alpha(z) =
	\frac{\alpha  e^{-z
			t^*_\alpha(0)} }{ 1 + \beta z} \int_0^{t^*_\alpha(0)}{ \frac{e^{u
				z}}{b(\varphi^\alpha_u(0)) +
			\alpha}  du} -  \frac{1-e^{-z t^*_\alpha(0)}}{z} - \frac{\beta e^{-z
			t^*_\alpha(0)}}{1 + \beta
		z}}.  \]
\subsection{A linear drift  \texorpdfstring{$b$}{b}.}
We now specify the shape of $b$. For some parameter $m > 1$, we 
choose:
\[ \forall x \geq 0,\quad b(x) = m - x, \]
It holds that $b(x) + \alpha = \sigma_\alpha -x$ with $\sigma_\alpha = m + \alpha$.
We have $\varphi^\alpha_u(0) = \sigma_\alpha(1-e^{-u})$ and so
\[ t^*_\alpha(0) = \log \left(\frac{\sigma_\alpha}{\sigma_\alpha - 1}\right). \]
Finally
\[ \int_0^{t^*_\alpha(0)}{ \frac{e^{u z}}{b(\varphi^\alpha_u(0)) + \alpha}  du} =
\frac{1}{\sigma_\alpha} \int_0^{t^*_\alpha(0)}{e^{(z + 1) u}} du = \frac{1}{
	\sigma_\alpha } \frac{e^{(z+1) t^*_\alpha(0)} - 1}{z+1},   \]
so
\[  J(\alpha) \widehat{\Psi}_\alpha(z) - \widehat{H}_\alpha(z) = \frac{\alpha}{
	\sigma_\alpha} \frac{ e^{ t^*_\alpha(0) } - e^{-z t^*_\alpha(0)} }{(1 + \beta
	z)(z+1)}
-  \frac{1-e^{-z t^*_\alpha(0)}}{z} - \frac{\beta e^{-z t^*_\alpha(0)}}{1 + \beta z}.  \]
Consequently, we have to study the complex solutions of
\begin{equation} 
	\label{eq:0511201559}
	\Re(z) > -1/\beta, \quad  \frac{\alpha}{m + \alpha -1 }
	\frac{ 1 - \left( \frac{m + \alpha}{m + \alpha - 1} \right)^{-(z +1)} }{(1 + \beta 
	z)(z+1)}  -
	\frac{1-\left( \frac{m + \alpha}{m +
			\alpha - 1} \right)^{-z}}{z} - \frac{ \beta \left( \frac{m +
			\alpha}{m + \alpha - 1} \right)^{-z}}{1+\beta z} = 0. 
\end{equation}
\begin{remark}
	In fact this analysis can be easily extended to any linear drift
	\[ b(x) = \kappa(m - x), \]
	with $\kappa, m \in  \mathbb{R}$. Indeed, adapting slightly the proof of \cite[Th.  
	21]{cormier2020meanfield} when $\kappa \leq 0$, it holds that $f + b' \geq
	0$ and
	so the unique non trivial invariant measure is locally stable: there is no
	Hopf bifurcation. If on the other hand $\kappa > 0$, by setting
	\[ \tilde{\kappa} = 1, \quad \tilde{\alpha} = \frac{\alpha}{\kappa}, \quad \tilde{m} = m
	\quad \tilde{\beta} = \kappa \beta,  \]
	we can easily reduce the problem to $\kappa = 1$.
\end{remark}
We now make the following change of variable
\[ \omega := \log \left(\frac{m + \alpha}{m + \alpha - 1} \right)  \quad \text{ and } \quad
\delta := \frac{\alpha}{m + \alpha -1},  \]
with $\omega > 0$ et $\delta \in (0, 1)$. That is, we have
\begin{equation}
	\label{eq:change of variables alpha and m}
	\alpha = \frac{\delta}{e^\omega -1 } \quad \text{ and } \quad m = 1 + \frac{1 - 
	\delta}{e^\omega -1}. 
\end{equation}
With this change of variable, \eqref{eq:0511201559} becomes
\begin{equation}
	\label{eq:explicit holomorphic function zero}
	\Re(z) > -1/\beta,\quad  \delta \frac{1}{1 + \beta z} \frac{1 - e^{-\omega(z+1)}}{1+z} 
	- 
	\frac{1-e^{-\omega z}}{z} - \frac{\beta e^{- \omega z}}{1 + \beta z} = 0. 
\end{equation}
	Recall that the strictly increasing function $\alpha \mapsto J(\alpha)$ is given by 
	\eqref{eq:formula for 
		J alpha toy}. With \eqref{eq:change of variables alpha and m}, we have
	\[ J'(\alpha) =  \beta + \omega - \delta(1-e^{-\omega}) \neq 0. \]
We deduce that  $z = 0$ is not a solution of \eqref{eq:explicit holomorphic function 
zero}. Multiplying 
by $(1 + \beta z)z$ on
both side of \eqref{eq:explicit holomorphic function zero}, we
finally find that we have to study the zeros of
\[ \Re(z) > -1/\beta,\quad  U(\beta, \delta, \omega, z) = 0, \]
with
\begin{equation}
	\label{eq:definiiton de U toy model hopf}
	\boxed{ U(\beta, \delta, \omega, z) := \delta \frac{z}{z+1} (1 - e^{-\omega(z+1)}) + 
		e^{-\omega z} 
		- (1 + \beta z). } 
\end{equation}
\subsection{On the roots of  \texorpdfstring{$U$}{U}}

\subsubsection*{An explicit parametrization of the purely imaginary roots}
We now describe all the imaginary roots of $U$. If $z = iy,~ y \geq 0$, the
equation
$U(\beta,
\delta, \omega, z) = 0$ yields
\begin{equation}
	\label{eq:toy model imaginary roots}
	\left \{
	\begin{array}{rcl}
		\cos(\omega y) + \sin(\omega y) y (1 - \delta e^{-\omega}) &=&1 - \beta y^2 \\
		-\sin(\omega y) + \cos(\omega y) y (1 - \delta e^{-\omega})&=&y(1 + \beta - 
		\delta).
	\end{array}
	\right.
\end{equation}
For $\omega > 0$ et $y \geq 0$ fixed, \eqref{eq:toy model imaginary roots}
admits a unique solution in $(\beta, \delta)$, given by
\begin{align}
	\label{eq:beta0 omega y}
	\beta^0_\omega(y) & := \frac{(1 + e^\omega)(1 - \cos(\omega y)) -(e^\omega-1)  y 
	\sin(\omega y) }{y^2 
		e^\omega - y^2 \cos(\omega y) - y \sin(\omega y)} , \\
	\delta^0_\omega(y) &:= \frac{ e^\omega(1 + y^2)(1 - \cos(\omega y)) }{y^2 
	e^\omega - y^2 \cos(\omega 
		y) - y \sin(\omega y)}. \nonumber
\end{align}

\begin{figure}[ht]
	\centering
	\subfloat[]{
		\begin{tikzpicture}[scale=0.7]
			\begin{axis}[
				xlabel={$ \beta^0_\omega $},
				ylabel={$ \delta^0_\omega $},
				title={Parametric plot in $(\beta, \delta)$},
				grid=major,
				grid style=dashed,
				tick label style={
					/pgf/number format/fixed,
					/pgf/number format/fixed zerofill,
					/pgf/number format/precision=2
				},
				xmin=-0.1,
				xmax=0.22,
				]
				\addplot[no marks, blue, very thick] table[ignore chars={(,)},col 
				sep=comma]
				{parametriccurve.dat};
				;
			\end{axis}
		\end{tikzpicture}
	}
	\qquad
	\subfloat[]{
		\begin{tikzpicture}[scale=0.7]
			\begin{axis}[
				xlabel={$ \beta $},
				ylabel={$ J $},
				title={Parametric plot in $(\beta, J)$},
				grid=major,
				grid style=dashed,
				tick label style={
					/pgf/number format/fixed,
					/pgf/number format/fixed zerofill,
					/pgf/number format/precision=2
				},
				xmin=0,
				xmax=0.1,
				]
				\addplot[color={red}, only marks, mark size=1.2pt] table[ignore 
				chars={(,)},col
				sep=comma]
				{parametric2.dat};
			\end{axis}
		\end{tikzpicture}
	}
	\caption{Description of the purely imaginary roots of $U$. (a) The parametric 
	curve  
		$(\beta^0_\omega(y),
		\delta^0_\omega(y))$, plotted
		with $\omega = 1$ and  $y \in [0, 15.5 \pi]$. Each point of the 
		curve
		corresponds to
		a purely imaginary roots of $U$. (b) Purely imaginary solutions of $U$
		plotted in the
		plane $(\beta, J)$, the value of $m$ being fixed ($m = 3/2$).}
	\label{fig:example}%
\end{figure}
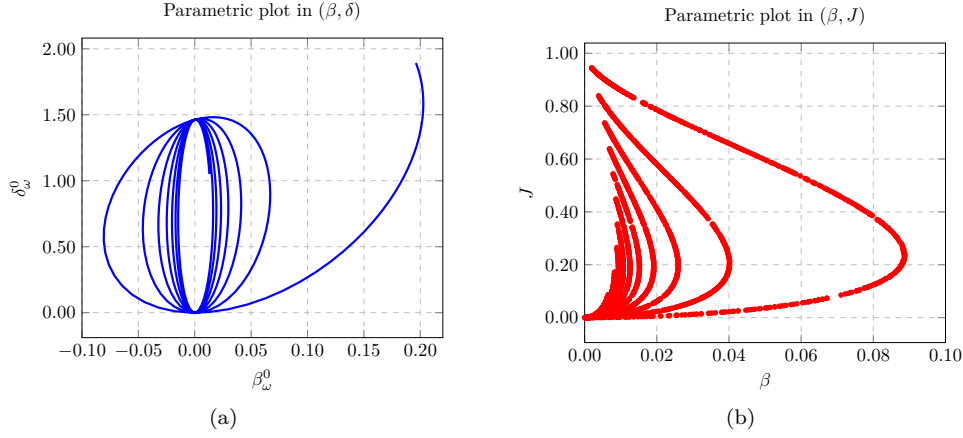
\begin{proposition}
	\label{prop:parametric curve}
	The parametric curve  $(\beta^0_\omega(y), \delta^0_\omega(y))_{y > 0}$
	admits exactly
	two multiple points given by
	\[ (0, 0) \quad \text{ and } \quad (0, \frac{2}{1 + e^{-\omega}}). \]
	Apart from those two points, the curve does not intersect itself.
\end{proposition}
\begin{proof}
	Squaring the two equations of \eqref{eq:toy model imaginary roots} and
	summing the
	result, one gets
	\[ 1 + y^2(1 - \delta e^{-\omega})^2 =(1 - \beta y^2)^2 + y^2 (1+\beta - \delta)^2,  \]
	that is
	\begin{equation} (1 - \delta e^{-\omega})^2 = -2 \beta + \beta^2 y^2 + (1 + \beta - 
	\delta)^2.
		\label{eq:link between delta beta and y}
	\end{equation}
	Note that if $\beta \neq 0$, for fixed values of $\delta, \beta$, there is a
	unique $y$
	satisfying this equation. This proves that all the multiple points are located
	on the axis
	$\beta = 0$. When $\beta = 0$, the equation becomes
	\[(1 - \delta e^{-\omega})^2 = (1 - \delta)^2,  \]
	whose solutions are
	\[ \delta = 0 \quad \text {and } \quad \delta = \frac{2}{1 + e^{-\omega}}. \]
	Those are indeed multiple points. For $(0, 0)$ for instance, it suffices to
	consider $y =
	\frac{2 \pi k}{\omega}, k \in \mathbb{N}^*$. This ends the proof.
\end{proof}
\subsection{Construction of the bifurcation point satisfying all the spectral
	assumptions. }
Let $\omega_0 > 0$ being fixed, chosen arbitrarily. Let $y_0 := \frac{2
	\pi}{\omega_0}(1-
\frac{\epsilon_0}{\omega_0})$ with $\epsilon_0 > 0$ (small) to be chosen later.
Let $\beta_0 := \beta^0_{\omega_0}(y_0)$ and $d_0 :=
\delta^0_{\omega_0}(y_0)$.
We have
\[ \beta_0 = \epsilon_0 + \mathcal{O}(\epsilon_0^2) \quad \text{ as } \epsilon_0
\rightarrow
0.
\]
and
\[ d_0 = \frac{e^{\omega_0}}{2 (e^{\omega_0} - 1)} \left(1 + \frac{(2
	\pi)^2}{\omega^2_0}
\right) \epsilon_0^2 + \mathcal{O}(\epsilon_0^3) \quad \text{ as } 
\epsilon_0
\rightarrow
0. \]
We then have from \eqref{eq:definiiton de U toy model hopf}
\[ \frac{\partial U}{\partial z}(\beta_0, d_0, \omega_0, i y_0) = - \omega_0 -
(1 + 2 i \pi)
\epsilon_0 +  \mathcal{O}(\epsilon_0^2) \quad \text{ as } \epsilon_0
\rightarrow
0. \]
This quantity is non-null provided that $\epsilon_0$ is sufficiently small. The
implicit function theorem applies and gives the existence of a $\mathcal{C}^1$
function $(\beta, \delta, \omega) \mapsto z_0(\beta, \delta, \omega)$ 
defined in the neighborhood of $(\beta_0, d_0, \omega_0)$ such that
\[ U(\beta, \delta, \omega, z_0(\beta, \delta, \omega)) = 0, \quad \text{ with }
\quad
z_0(\beta_0, d_0, \omega_0) = i y_0. \]
Furthermore, one has
\[ \frac{\partial}{\partial \delta} z_0(\beta_0, d_0, \omega_0) = -
\frac{\frac{\partial
		U}{\partial \delta} (\beta_0, d_0, \omega_0, i y_0)}{ \frac{\partial U}{\partial
		z} (\beta_0,
	d_0, \omega_0, i y_0) } \overset{\eqref{eq:definiiton de U toy model 
	hopf}}{=} 2 \pi \frac{1 - 
	e^{-\omega_0}}{\omega_0} \frac{2
	\pi + i
	\omega_0}{(2 \pi)^2 + \omega_0^2} + \mathcal{O}(\epsilon_0) \quad \text{ as }
\epsilon_0
\rightarrow
0 \]
and
\[ \frac{\partial}{\partial \omega} z_0(\beta_0, d_0, \omega_0) = -
\frac{\frac{\partial
		U}{\partial \omega} (\beta_0, d_0, \omega_0, i y_0)}{ \frac{\partial
		U}{\partial z}
	(\beta_0, d_0, \omega_0, i y_0) } \overset{\eqref{eq:definiiton de U toy 
	model hopf}}{=} - 
\frac{2 i \pi}{\omega_0} +
\mathcal{O}(\epsilon_0)
\quad \text{ as } \epsilon_0
\rightarrow
0.\]
We finally set
\[ \alpha_0 := \frac{d_0}{e^{\omega_0} -1} , \quad m_0 := 1 + \frac{1 - 
d_0}{e^{\omega_0} -1},  \]
and
\[ \mathfrak{Z}_0(\alpha) := z_0(\beta_0, \frac{\alpha}{m_0+\alpha-1}, \log\left( 
\frac{m_0 +
	\alpha}{m_0 + \alpha -1} \right)), \]
such that
\begin{align*} \frac{d}{d \alpha} \mathfrak{Z}_0 (\alpha_0) = & 2 \pi \frac{1 - 
		e^{-\omega_0}}{\omega_0}
	\frac{2 \pi + i \omega_0}{(2 \pi)^2 + \omega_0^2} \frac{m_0 -1 }{(m_0 - 1 - 
	\alpha_0)^2} \\
	&\quad + \frac{2 i \pi}{\omega_0} \frac{1}{(m_0 -1 + \alpha_0)(m_0 + \alpha_0)} +
	\mathcal{O}(\epsilon_0)
	\quad \text{ as } \epsilon_0
	\rightarrow
	0. 
\end{align*}
The second term on the right hand side is purely imaginary. So
\[  \Re \frac{d}{d \alpha} \mathfrak{Z}_0 (\alpha_0) =  \frac{1 - 
e^{-\omega_0}}{\omega_0}
\frac{(2 \pi)^2}{(2 \pi)^2 + \omega_0^2} \frac{m_0 -1 }{(m_0 - 1 - \alpha_0)^2} +
\mathcal{O}(\epsilon_0)
\quad \text{ as } \epsilon_0
\rightarrow
0. \]
This quantity is strictly positive provided that $\epsilon_0$ is small enough.
By choosing the parameters of the model to be $\beta = \beta_0$ and $m =
m_0$, the
Assumptions~\ref{ass:i / b0 roots of the characteristic equation},
\ref{ass:nonresonance
	condition} and \ref{ass:hopf condition} are satisfied at the point $\alpha =
\alpha_0$. In particular, Assumption~\ref{ass:nonresonance
		condition} follows from Proposition~\ref{prop:parametric curve}.

\providecommand{\bysame}{\leavevmode\hbox to3em{\hrulefill}\thinspace}
\providecommand{\MR}{\relax\ifhmode\unskip\space\fi MR }
\providecommand{\MRhref}[2]{%
	\href{http://www.ams.org/mathscinet-getitem?mr=#1}{#2}
}
\providecommand{\href}[2]{#2}

\section*{Acknowledgements}
		The authors  thank the anonymous referees 
	for their comments which contributed to clarify the presentation of this work.
	
	\noindent This project/research has received funding from the European Union’s 
	Horizon 2020 Framework Programme for Research and Innovation under the 
	Specific 
	Grant Agreement No. 945539 (Human Brain Project SGA3).

\end{document}